%% file: PnP_Module_Renke.tex
\documentclass[a4paper]{article}
\usepackage{geometry}
\usepackage{a4wide}
\usepackage{graphicx,xcolor}
\graphicspath{{./images/}}
\usepackage{caption}
\usepackage{subcaption}
\usepackage{empheq}
\usepackage{amsthm}
\usepackage{amssymb, latexsym, mathrsfs}
\usepackage{amsmath}
\usepackage{dsfont}
\usepackage{enumerate}
\usepackage{algorithm}
\usepackage{color}
\usepackage{transparent}
\usepackage{url}
\usepackage[affil-it]{authblk}
\usepackage{booktabs}
\usepackage{graphicx}
\usepackage{multirow}
\usepackage[american,cuteinductors,smartlabels]{circuitikz}
\usepackage{tikz}
\usepackage{schemabloc}
\usetikzlibrary{shapes,arrows}

\theoremstyle{plain}
\newtheorem{thm}{Theorem}
\newtheorem{assum}{Assumption}
\newtheorem{cor}{Corollary}
\newtheorem{rmk}{Remark}
\newtheorem{prop}{Proposition}
\newtheorem{lem}{Lemma}

\input{macro}
\begin{document}
	\title{\LARGE \bf Hierarchical Plug-and-Play Voltage/Current Controller of DC Microgrid Clusters with Grid-Forming/Feeding Converters: Line-independent Primary Stabilization and Leader-based Distributed Secondary Regulation}
        
          \author[1]{Renke Han%
       \thanks{Electronic address:
         \texttt{rha@et.aau.dk}; Corresponding author}}
        \author[2]{Michele Tucci%
       \thanks{Electronic address: \texttt{michele.tucci02@universitadipavia.it}}}
    \author[3]{Raffaele Soloperto%
    	\thanks{Electronic address: \texttt{raffaele.soloperto@ist.uni-stuttgart.de}} }
    \author[4]{Andrea Martinelli%
    	\thanks{Electronic address: \texttt{andrea5.martinelli@mail.polimi.it}} }
     \author[1]{Josep M. Guerrero%
    	\thanks{Electronic address: \texttt{joz@et.aau.dk}} }
     \author[5]{Giancarlo Ferrari-Trecate%
       \thanks{Electronic address: \texttt{giancarlo.ferraritrecate@epfl.ch}\\This work has received support from the Swiss National Science Foundation
       	  	under the COFLEX project (grant number 200021-169906)} }

     \affil[1]{Department of Energy Technology, Aalborg University, Aalborg, Denmark}
     \affil[2]{Dipartimento di Ingegneria Industriale e
     	dell'Informazione\\Universit\`a degli Studi di Pavia}  
     \affil[3]{Institute of System Theory and Automatic Control, University of Stuttgart, Stuttgart, Germany}
     \affil[4]{Dipartimento di Elettronica, Informazione e Bioingegneria, Politecnico di Milano, 20133 Milano, Italy.}
     \affil[5]{Automatic Control Laboratory, \'Ecole Polytechnique F\'ed\'erale de Lausanne (EPFL), 1015 Lausanne, Switzerland.}
     
     \date{\textbf{Technical Report}\\ July, 2017}

     \maketitle
     \begin{abstract}
 Considering the single MG composed of grid-forming/feeding converters and the MG clusters, the hierarchical Plug-and-Play (PnP) voltage/current controller of MG clusters is proposed. Different from existing methods, the main contributions are provided as follows:
 \begin{itemize}
 \item In a single MG, a PnP controller for the current-controlled distributed generation units (CDGUs) is proposed to achieve grid-feeding current tracking while guaranteeing the stability of the whole system. Moreover, the set of stabilizing controllers for CDGUs is characterized explicitly in terms of simple inequalities on the control coefficients. With the proposed controller, CDGUs can  plug-in/out of the MG seamlessly without knowing any information of the MG system and without changing control coefficients for other units.
 \item Interconnected with singel consisting of CDGU and voltage-controlled DGUs (VDGU), MG clusters are formed. To be specific, the CDGU  is used for renewable energy sources (RES) to feed current and VDGU is used for energy storage system (ESS) to provide voltage support. A PnP voltage/current controller is proposed to achieve simultaneous grid-forming/feeding function irrespective of the power line parameters. Also in this case, the stabilizing controller is related only to local parameters of a MG and is characterized by explicit inequalities. With the proposed controller, MGs can  plug-in/out of the MG clusters seamlessly without knowing any information of the system and changing coefficients for other MGs. 
 \item For the system with interconnection of MGs, a leader-based voltage/current distributed secondary controller is proposed to achieve both the voltage and current regulation without specifying the individual setpoints for each MGs. The proposed controller requires communication network and each controller exchanges information with its communication neighbors only. By approximating the primary PnP controller with unitary gains, the model of leader-based secondary controller with the PI interface is established and the stability of the closed-loop MG is proven by Lyapunov theory.
 \end{itemize}
 
 Proofs of the closed-loop stability of proposed system for CDGUs and MG clusters exploits structured Lyapunov functions, the LaSalle invariance theorem and properties of graph Laplacians. Finally, theoretical results are demonstrated by hardware-in-loop tests.
 \end{abstract}

\newpage
\section{Introduction}
\label{sec:intro}

With the increasing penetration of renewable energies into modern electric systems, the concept of microgrid (MG) receives increasing attention from both electric industry and academia. One MG should be formed by interconnecting a number of renewable energy sources (RESes), energy storage systems (ESSes) and different types of loads, which can be realistic if the final user is able to generate, store, control, and manage part of the energy that it will consume \cite{guerrero2011hierarchical,7995119}. Power converters are the key components applied in both ac and dc MGs to interface different sorts of energy resources and loads into the system. To be specific, in ac MG, power converters can be classified into grid-forming and grid-feeding converters \cite{6200347}, and the same classification can also be applied for dc MGs. While remarkable progress has been made in improving the performance of ac MGs during the past decade, dc MGs (which are studied in this paper) have been recognized as more and more attractive due to higher efficiency, more natural interface to many types of RESes and ESSes \cite{dragicevic2015dc}. 

Grid-forming converters can be seen as the interface between ESSes and the system to provide voltage support in the dc MG. In order to achieve simultaneous voltage support and communication-less current sharing among ESSes, voltage-current (V-I) droop control \cite{guerrero2011hierarchical} is widely adopted by imposing virtual impedance for the output voltages, but voltage deviations and current sharing errors still exist due to different line impedances. Meanwhile, another key challenge is that the stability of connected ESSes is sensitive to the chosen virtual impedances which should be designed taking the specific MG topology and the values of line impedances into consideration \cite{7946266, shafiee2014hierarchical,6816073}. In addition, the droop controller combined with inner voltage-current control loop forms the decentralized primary control level in which at least five control coefficients must be designed \cite{guerrero2011hierarchical}. Recently, an alternative class of decentralized primary controllers, called PnP controller according to the terminology used in \cite{Riverso2014a, 7040312}, has been proposed in \cite{7419890}. PnP controllers form a decentralized control architecture where each regulator can be synthesized using information about the corresponding ESSes \cite{7934339} or at most, parameters of the power lines connected to the ESS \cite{7419890}. In particular, the latter pieces of information are not required in the design procedure of \cite{7934339} which is therefore termed line-independent method. The main feature of the PnP controller is to preserve the global stability of the whole MG independently of the MG topology. Moreover, when ESSes are plugged-in/out of the system, local controllers can be designed on the fly, without knowing the model of other ESSes and yet preserving global stability of the new MG. However, in both \cite{7419890} and \cite{7934339}, the synthesis of a PnP controller requires to solve a convex optimization problem, if unfeasible, the plug-in/out of corresponding ESSes should be denied. 

The proposed controllers in \cite{7419890,7934339} are only applied for grid-forming converters. However, grid-feeding converters for CDGUs should be also considered when RESes such as PV source are joined in dc MGs. The current-based PnP controller should be designed for grid-feeding converters to track current reference given by e.g. maximum power point tracking (MPPT) algorithm. Meanwhile, the current stabilization should also be guaranteed. In \cite{zhao2015distributed}, a current-based PI primary droop control is proposed considering the constant current load, however, if the current reference and the constant current load are different, the voltage deviations can become large. In addition, while several literature \cite{7015592,7546855,6497633} considered the problem of energy management operation between RESes and ESSes, the global stability problem about MG and MG clusters has always been ignored from the point view of system level.
 
In this paper, main contributions are concluded as follows:
\begin{enumerate}[(i)]
	\item Considering the grid-feeding converters in single MG, the current-based PnP controller is proposed for CDGUs to achieve current tracking. In order to guarantee the current stability of the MG joined by CDGUs, the control coefficients of each controller only need to fulfill simple inequalities. Hence, different from the method in \cite{7419890,7934339}, no optimization problem need to be solved for designing local regulators which means the design of stabilizing regulators is always feasible independent of system parameters.
	\item Considering the MG clusters interconnected with MGs composed of grid-forming/feeding converters, a PnP voltage/current controller is proposed for the system to achieve both the voltage and current tracking simultaneously. The set of control coefficients is characterized explicitly through a set of inequalities. Hence, the controller design is always feasible and does not require to solve an optimization problem. It is proven that the global stability can be guaranteed by implementing PnP controller for each MG, which is independent of line impedances.
	\item As in \cite{7934339}, the proofs of closed-loop asymptotic stability of using the proposed controller for MGs and MG clusters exploit structured Lyapunov functions, the LaSalle invariance theorem and properties of graph Laplacians. This shows that these tools offer a feasible theoretical framework for analyzing different kinds of MGs equipped with various types PnP decentralized control architectures. 
	\item  For MG clusters, a leader-based voltage/current distributed secondary controller is proposed to achieve both the voltage and current tracking with the information from the higher control level. Each MG only requires its own information and the information of its neighbours on the communication network graph. Instead of implementing only integral controller as the interface between primary and secondary control level, PI controller is applied as the interface to improve the dynamic control performance. By approximating the primary PnP controller with unitary gains, the model of leader-based secondary controller with the PI interface is established whose stability is proven by Lyapunov theory. 
\end{enumerate} 

The paper is structured as follows. In Section \ref{sec:Model_C} and \ref{sec:aug_sys}, the CDGU model and proposed current-based PnP controllers are introduced. In Section \ref{sec:pnp_design}, the closed-loop stability for CDGU is proven. In Section \ref{PV/Battery_Model} and \ref{sec:aug_sys_for_V_C}, the proposed voltage/current PnP controller for MGs are introduced. In Section \ref{sec:pnp_design_modules}, the closed-loop stability for MG clusters is proven. The leader-based voltage/current distributed secondary controller and its stability proof are introduced in Section \ref{leader-based controller}. Finally, the hardware-in-loop tests are described in Section \ref{sec:simulation_results}.

\textbf{Notation.} We use $P>0$ (resp. $P\geq 0$) for indicating the
real symmetric matrix $P$ is positive-definite
(resp. positive-semidefinite). Let
	$A\in\mathbb{R}^{n\times m}$ be a
	matrix inducing the linear map $A:\mathbb{R}^m\rightarrow \mathbb{R}^n$. $I\in\mathbb{R}^{n\times n}$ represent unit matrix.
	The average of a vector $v\in\mathbb{R}^n$ is $ \left\langle v \right\rangle  =\frac{1}{n}\sum v_{i}$. We denote with $H^1$ the subspace composed by all vectors with zero average i.e. $H^1=\{v\in\mathbb{R}^n: \left\langle v \right\rangle=0\}$. The space orthogonal to $H^1$ is $H_{\perp}^1$. It holds $H_{\perp}^1=\{\alpha\mbf{1_n}: \alpha\in\mathbb{R}\}$ and $dim(H_{\perp}^1)=1$ \cite{ferrari2006analysis}. Moreover, the decomposition $\mathbb{R}^n=H^1\oplus H_{\perp}^1$ is direct \cite{callier2012linear}.
\section{ Grid-Feeding Converters of Current-controlled DGUs in dc Microgrid}
          \label{sec:Model_C}      
\subsection{ Electrical model of CDGUs}
\label{sec:el_model_C}
In this subsection, the electrical model for CDGUs is described. The control objective for CDGU is to feed current for the MG according to a given current reference. The electrical scheme of the $i$-th CDGU is represented within upper part of Fig. \ref{fig:ctrl_part_C}. It is assumed that loads including both a resistive load and a current disturbance($I_{Li}$) are unknown. 

We consider a system composed of $N$ CDGUs and define the set $\DD^C=\{1,\dots,N\}$. Two CDGUs are neighbors if there is a
power line connecting them.
$\NN_{i}^C\subset\DD^C$ denotes the subset of neighbors of CDGU
$i$. The neighboring relation is symmetric which means
$j\in\NN_{i}^C$ implies
$i\in\NN_{j}^C$. Furthermore, let $\mathcal
E=\{(i,j):$ $i\in\DD^C,$ $j\in\NN_{i}^C\}$ collect unordered pairs of
indices associated to lines. Each line is described by a $RL$ model.
The topology of the multiple CDGUs is then described by the
undirected graph $\GG_{el}$ with nodes $\DD^C$ and
edges $\EE$.

From Fig. \ref{fig:ctrl_part_C}, by applying Kirchoff's voltage and current laws, and exploiting QSL approximation
of power lines \cite{7419890,8026170}, the model of CDGU $i$ is obtained
\begin{equation}
	\label{eq:CDGU}
	\text{CDGU}~i:\hspace{-4mm}\quad\left\lbrace
	\begin{aligned}
		\frac{dV_{i}}{dt} &= \frac{1}{C_{ti}}I_{ti}^C+\sum\limits_{j\in\NN_i}\left(\frac{V_j}{C_{ti} R_{ij}}-\frac{V_i}{C_{ti}R_{ij}}\right)-\frac{1}{C_{ti}}(I_{Li}+\frac{V_i}{R_{Li}})\\
		\frac{dI_{ti}^C}{dt} &= -\frac{1}{L_{ti}^C}V_{i}-\frac{R_{ti}^C}{L_{ti}^C}I_{ti}^C+\frac{1}{L_{ti}^C}V_{ti}^C\\
	\end{aligned}
	\right.
\end{equation}
where variables $V_i$, $I_{ti}^C$, are the $i$-th PCC voltage and
filter current, respectively, $V_{ti}^C$ represents the command to the
converter, and $R_{ti}^C$, $L_{ti}^C$ and $C_{ti}$ represent the electrical
parameters of converters. Moreover, $V_{j}$ is the voltage at the
PCC of each neighboring CDGU $j\in\NN_{i}^C$ and $R_{ij}$ is the resistance
of the power line connecting CDGUs $i$ and $j$. 
\begin{figure}
	\centering
	\includegraphics[scale=0.8]{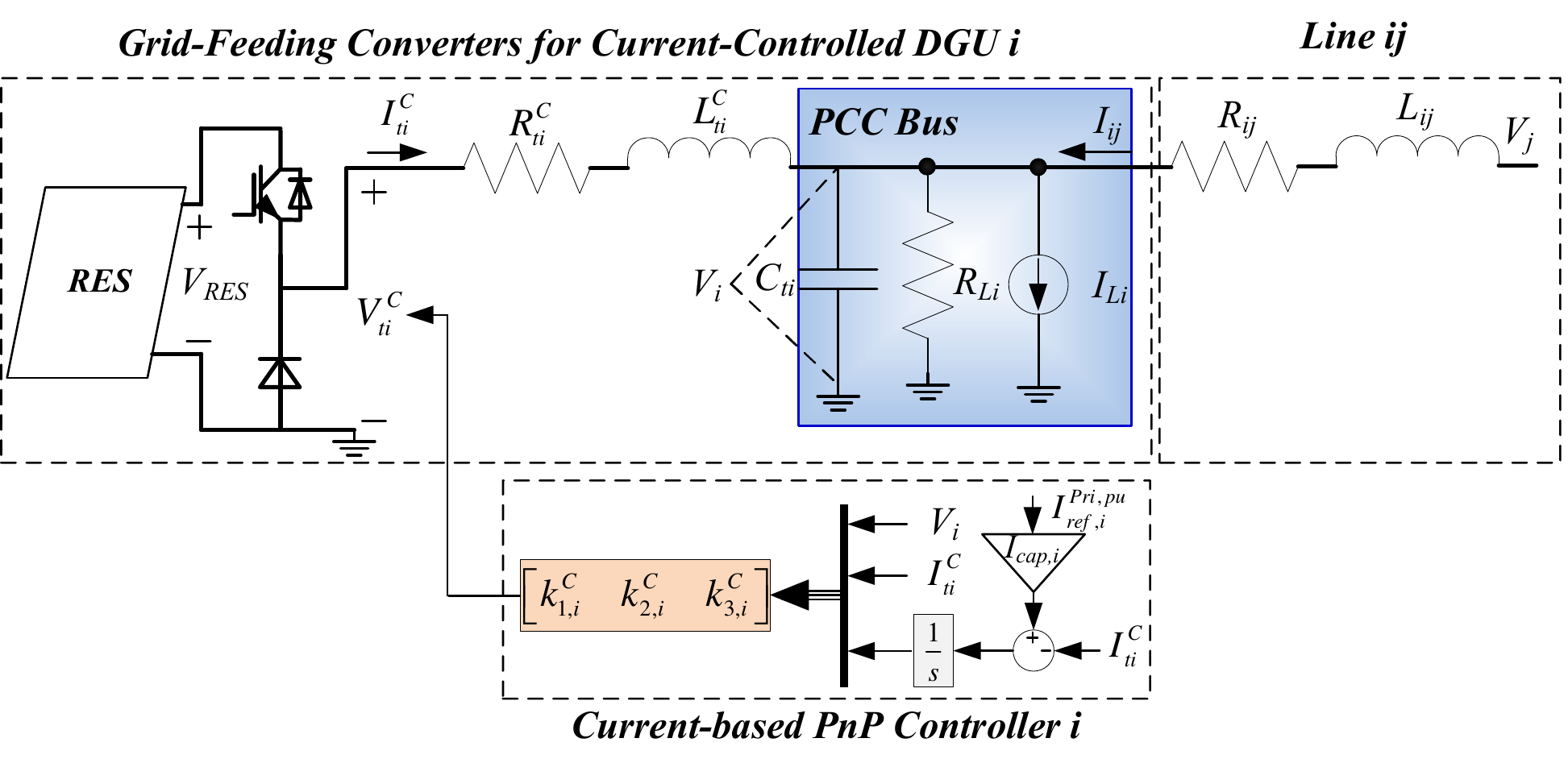}
	\caption{Electrical Scheme of CDGU $i$ and current-based PnP controller.}
	\label{fig:ctrl_part_C}
\end{figure}
\begin{rmk}
	\label{capacitor_added}
	 In practical, the grid-feeding converters need the voltage support from the grid-forming converters at the PCC point. In this section, only the controller and stability for the interconnected CDGU is designed and analyzed. Thus, it is assumed that the voltage at the PCC point has already been supported by the grid-forming devices. In section \ref{PV/Battery_Model} and \ref{sec:PnPctr_for_V_Cl}, the PnP controllers to achieve both the voltage support and current feeding are proposed, designed and analyzed. 
\end{rmk}
\subsection{State-space model of multiple CDGUs}
\label{sec:State-space model of the mG with CDGU}
Dynamics \eqref{eq:CDGU} provides the state-space equations: 
\begin{equation*}
	\text{ $\subss{\Sigma}{i}^{CDGU}:$}\left\lbrace
	\begin{aligned}
		\subss{\dot{x}}{i}^C(t) &= A_{ii}^C\subss{x}{i}^C(t) +
		B_{i}^C\subss{u}{i}^C(t)+M_{i}^C\subss{d}{i}^C(t)+\subss\xi{i}^C(t)+A_{load,i}^C\subss{x}{i}^C(t)\\
		\subss{z}{i}^C(t)       &= H_{i}^C\subss{x}{i}^C(t)\\
	\end{aligned}
	\right.
\end{equation*}
where $\subss{x}{i}^C=[V_{i},I_{ti}^C]^T$ is the state,
$\subss{u}{i}^C = V_{ti}^C$ the control input,
$\subss{d}{i}^C = I_{Li}^C$ the exogenous input including different current loads and
$\subss{z}{i}^C = I_{ti}^C$ the controlled variable of the
system. 
The
term $\subss\xi i^C=\sum_{j\in\NN_i}A_{ij}^C(\subss x j^C-\subss x i^C)$ accounts
for the coupling with each CDGU $j\in\NN_{i}^C$ and the term $A_{load,i}^C$ accounts for the resistive load for each CDGU.
The matrices of
$\subss{\Sigma}{i}^{CDGU}$ are obtained from
\eqref{eq:CDGU} as:
\begin{equation*}
	\renewcommand\arraystretch{1.5}
	A_{ii}^C=\begin{bmatrix}
		0 & \frac{1}{C_{ti}} \\
		-\frac{1}{L_{ti}^C} & -\frac{R_{ti}^C}{L_{ti}^C} \\
	\end{bmatrix},  \hspace{3mm} A_{load,i}^C=
	\begin{bmatrix}
	-\frac{1}{R_{Li}C_{ti}}  & 0 \\
	0 & 0 
	\end{bmatrix}, \hspace{3mm} A_{ij}^C=
	\begin{bmatrix}
		\frac{1}{R_{ij}C_{ti}}  & 0 \\
		0 & 0 
	\end{bmatrix},
\end{equation*}
\begin{equation*}
	B_{i}^C=\begin{bmatrix}
		0 \\
		\frac{1}{L_{ti}^C}
	\end{bmatrix},
	\qquad
	M_{i}^C=\begin{bmatrix}
		-\frac{1}{C_{ti}} \\
		0 \\
	\end{bmatrix},
	\qquad
	H_{i}^C=\begin{bmatrix}
		0 & 1 
	\end{bmatrix}.
\end{equation*}
\begin{rmk}
To be emphasized, there are two main differences between the proposed model for CDGU in \eqref{eq:CDGU} and the one proposed in \cite{7934339}. The first one is that the resistive load is considered as part of the load. The second one is that the control variable is changed from voltage in \cite{7934339} for grid-forming converters to current in \eqref{eq:CDGU} for grid-feeding converters.
\end{rmk}

The overall model with multiple CDGUs is given by
\begin{equation}
	\begin{aligned}
		\label{eq:stdformA}
		\mbf{\dot{x}^C}(t) &= \mbf{A^Cx^C}(t) + \mbf{B^Cu^C}(t)+ \mbf{M^Cd^C}(t)\\
		\mbf{z}^C(t)       &= \mbf{H^Cx^C}(t)
	\end{aligned}
\end{equation}
where $\mbf {x^C} = (\subss x 1^C,\ldots,\subss x
N^C)\in\Rset^{2N}$, $\mbf {u^C} = (\subss u 1^C,\ldots,\subss u
N^C)\in\Rset^{N}$, $\mbf {d^C} = (\subss d 1^C,\ldots,\subss d
N^C)\in\Rset^{N}$, 
$\mbf {z^C} = (\subss z 1^C,\ldots,\subss z
N^C)\in\Rset^{N}$. Matrices $\mbf{A^C}$, $\mbf{B^C}$, $\mbf
{M^C}$ and $\mbf {H^C}$ are reported in Appendix \ref{sec:AppMatrices_CDGU}. 

     \section{Design of stabilizing current controllers}
         \label{sec:PnPctrl}
        \subsection{Structure of current-based PnP controllers}
        \label{sec:aug_sys}
        In
        order to track with references $\mbf{z_{ref}^C}(t)$, when
        $\mbf{d^C}(t)=\mbf{\bar d^C}$ is constant, the CDGU model is augmented
        with integrators \cite{Skogestad1996}. A
        necessary condition for making error $\mbf{e^C}(t)=\mbf{z_{ref}^C}(t)-\mbf{z^C}(t)$ equal to zero as $t\rightarrow\infty$, is that, there are equilibrium states and inputs $\mbf{\bar
        	x^C}$ and $\mbf{\bar u^C}$ verifying \eqref{eq:stdformA}. The existence of these equilibrium points can be shown following the proof of
        Proposition 1 in \cite{7419890}.
        
        One obtain the integrator dynamics is (as shown in Fig.
        \ref{fig:ctrl_part_C}, setting $z_{ref_{[i]}}^C=I_{ref,i}^{Pri,pu}*I_{cap,i}$, $I_{cap,i}$ is the maximum capability of CDGU $i$ and $I_{ref,i}^{Pri,pu}$ is the p.u. reference)
        \begin{equation}
        	\begin{aligned}
        		\subss{\dot{v}}{i}^C(t) = \subss{e}{i}^C(t) &= \subss{z_{ref}}{i}^C(t)-\subss{z}{i}^C(t) \\
        		&= \subss{z_{ref}}{i}^C(t)-H_{i}^C\subss{x}{i}^C(t),
        	\end{aligned}
        	\label{eq:intdynamics_C}
        \end{equation}
        and hence, the augmented CDGU model is
        \begin{equation}
        	\label{eq:modelDGUgen-aug}
        	\subss{\hat{\Sigma}}{i}^{CDGU} :
        	\left\lbrace
        	\begin{aligned}
        		\subss{\dot{\hat{x}}}{i}^C(t) &= \hat{A}_{ii}^C\subss{\hat{x}}{i}^C(t) + \hat{B}_{i}^C\subss{u}{i}^C(t)+\hat{M}_{i}^C\subss{\hat{d}}{i}^C(t)+\subss{\hat\xi}i^C(t)+\hat{A}_{load,i}^C\subss{\hat{x}}{i}^C(t)\\
        		\subss{z}{i}^C(t)       &= \hat{H}_{i}^C\subss{\hat{x}}{i}^C(t)
        	\end{aligned}
        	\right.
        \end{equation}
        where $\subss{\hat{x}}{i}^C=[[\subss{x}{i}^C]^T
        ,\subss{v}{i}^C]^T\in\Rset^3$ is the state,
        $\subss{\hat{d}}{i}^C=[\subss{d}{i}^C,\subss{z_{ref}}{i}^C]^T\in\Rset^2$
        collects the exogenous signals and
        $\subss{\hat\xi}{i}^C=\sum_{j\in\NN_i}\hat{A}_{ij}^C(\subss{\hat{x}}{j}^C-\subss{\hat{x}}{i}^C)$. By direct calculation, the  matrices appeared in \eqref{eq:modelDGUgen-aug} are as follows
        \begin{equation*}
        	\begin{aligned}
        		\hat{A}_{ii}^C &=\begin{bmatrix}
        			A_{ii}^C & 0\\
        			-H_{i}^C & 0
        		\end{bmatrix},\hspace{2mm}
        		\hat{A}_{ij}^C=\begin{bmatrix}
        			A_{ij}^C &0\\
        			0&0
        		\end{bmatrix},\hspace{2mm}
        		\hat{A}_{load,i}^C=\begin{bmatrix}
        		A_{load,i}^C &0\\
        		0&0
        		\end{bmatrix},\\
        		\hat{B}_{i}^C&=\begin{bmatrix}
        			B_{i}^C\\
        			0
        		\end{bmatrix},\hspace{2mm}
        		\hat{M}_{i}^C=\begin{bmatrix}
        			M_{i}^C & 0 \\
        			0 & 1
        		\end{bmatrix},\hspace{2mm}
        		\hat{H}_{i}^C=\begin{bmatrix}
        			H_{i}^C & 0
        		\end{bmatrix}.
        	\end{aligned}
        \end{equation*}   
        
        Based on Proposition 2 of \cite{7419890}, the pair $(\hat{A}_{ii}^C,\hat{B}_{i}^C)$ can be proven to be controllable. Hence,
        system \eqref{eq:modelDGUgen-aug} can be stabilized.
        
        Given from \eqref{eq:modelDGUgen-aug}, the overall augmented system is 
        \begin{equation}
        	\label{eq:sysaugoverall_1}
        	\left\lbrace
        	\begin{aligned}
        		\mbf{\dot{\hat{x}}^C}(t) &= \mbf{\hat{A}^C\hat{x}^C}(t) + \mbf{\hat{B}^Cu^C}(t)+ \mbf{\hat{M}^C\hat{d}^C}(t)\\
        		\mbf{z}^C(t)       &= \mbf{\hat{H}^C\hat{x}^C}(t)
        	\end{aligned}
        	\right.
        \end{equation}
        where $\mbf{\hat{x}^C}$ and $\mbf{\hat{d}^C}$ include all variables $\subss{\hat{x}}{i}^C$ and $\subss{\hat{d}}{i}^C$ respectively from all the CDGUs, and matrices $\mbf{\hat{A}^C}, \mbf{\hat{B}^C}, \mbf{\hat{M}^C}$ and $\mbf{\hat{H}^C}$ are derived from systems \eqref{eq:modelDGUgen-aug}. 
        
        Now each CDGU $\subss{\hat{\Sigma}}{i}^{CDGU}$ is equip with the following state-feedback controller
        \begin{equation}
        	\label{eq:ctrldec}
        	\subss{\CC}{i}^C:\qquad \subss{u}{i}^C(t)=K_{i}^C\subss{\hat{x}}{i}^C(t)
        \end{equation}
        where $K_{i}^C=\left[k_{1,i}^C\text{ }k_{2,i}^C\text{ }k_{3,i}^C\right]\in\Rset^{1\times3}$. 
        
        It turns out that, together with the integral action \eqref{eq:intdynamics_C}, controllers $\subss{\CC}{i}^C$, define a multivariable PI regulator, see lower part of Fig.~\ref{fig:ctrl_part_C}. In particular, the overall control architecture is
        decentralized since the computation of
        $\subss{u}{i}^C$ requires the state of
        $\subss{\hat{\Sigma}}{i}^{CDGU}$ only. In the following, it is shown that structured Lyapunov functions can be used to ensure asymptotic stability of the system with multiple CDGUs with controllers \eqref{eq:ctrldec}.
        
	  \subsection{Conditions for stability of the closed-loop multiple CDGUs}
           \label{sec:pnp_design}
         As in \cite{7934339}, the design of gain $K_{i}^C$ hinges on the use of separable local Lyapunov function for certifying the closed-loop stability. Indeed, the structure will also allow us to show that local stability implies stability of the whole system. Here after, the candidate Lyapunov function are considered as
       	\begin{equation}
       	\label{eq_sep_lyap}
       	V_{i}^C(\subss{\hat{x}}{i}^C)=[\subss{\hat{x}}{i}^C]^TP_{i}^C\subss{\hat{x}}{i}^C
       	\end{equation}
       	where positive definite matrices $P_{i}^C\in\Rset^{3\times3}$ has the structure
       	\begin{equation}
       	\label{eq:pstruct}
       	P_{i}^C=\left[
       	\renewcommand\arraystretch{1.5}
       	\begin{array}{c|c}
       	\eta_{i} & \mbf{0}_{1\times 2} \\
       	\hline
       	\mbf{0}_{2\times 1}  & \PP_{22,i}^C\\
       	\end{array}\right],
       	\end{equation}
       	where $\eta_{i}>0$ is a parameter and the entries of $\PP_{22,i}^C$ are arbitrary and denoted as
       	        \begin{equation}
       	        	\label{eqn:P22}
       	        	\PP_{22,i}^C=
       	        	\left[ \begin{array}{cc}
       	        		p_{22,i}^C & p_{23,i}^C  \\
       	        		p_{23,i}^C & p_{33,i}^C
       	        	\end{array}\right] .
       	        \end{equation}
       	We also assume that given a constant parameter common to
       	all CDGUs $\bar\sigma>0$ just for proof process, the parameters $\eta_{i}$ in \eqref{eq:pstruct} are set as
       	\begin{equation}
       	\label{eq:equal_ratio}
       	\eta_{i} = \bar\sigma C_{ti}
       	\hspace{7mm}  i\in\DD^C.
       	\end{equation}
       
        In absence of coupling terms $\subss{\hat\xi}i^C(t)$, and load terms $\hat{A}_{load,i}^C\subss{\hat{x}}{i}^C(t)$, one
        would like to stabilize the closed-loop CDGU 
        \begin{equation}
        		\label{eq:modelDGUgen-aug-closed}
        		\subss{\dot{\hat{x}}}{i}^C(t) = \underbrace{(\hat{A}_{ii}^C+\hat{B}_{i}^CK_i^C)}_{F_{i}^C}\subss{\hat{x}}{i}^C(t)+\hat{M}_{i}^C\subss{\hd}{i}^C(t).\\
        \end{equation}
        
        By direct calculation, one has 
        \begin{equation}
		\begin{aligned}
        			 \label{eq:Fi}
        			 \renewcommand\arraystretch{3}
        			F_{i}^C=\left[
        			\renewcommand\arraystretch{1.8}
        			\begin{array}{c|cc}
        				0 & \frac{1}{C_t} & 0 \\
        				\hline
        				\frac{(k_{1,i}^C-1)}{L_{ti}^C} & \frac{(k_{2,i}^C-R_{ti}^C)}{L_{ti}^C} & \frac{k_{3,i}^C}{L_{ti}^C}\\
        				0 & -1 & 0
        			\end{array}\right]=\left[
        			\renewcommand\arraystretch{1.5} \begin{array}{c|cc}
        				0 & \FF_{12,i}^C \\
        				\hline
        				\FF_{21,i}^C& \FF_{22,i}^C
        			\end{array}\right].
        	\end{aligned}
        \end{equation}
        
        From Lyapunov theory, asymptotic stability of \eqref{eq:modelDGUgen-aug-closed} can be certified by the existence of a Lyapunov function as shown in \eqref{eq_sep_lyap} and 
        \begin{equation}
       	\label{eq:Lyapeqnith}
        	Q_{i}^C = [F_{i}^C]^T
        	P_{i}^C+P_{i}^CF_{i}^C 
        \end{equation}
        is negative definite. 
%
    
    Based on \eqref{eq:pstruct} and \eqref{eq:Fi}, eq. \eqref{eq:Lyapeqnith} can be rewritten as
    \begin{equation}
    \label{fullQ}
    Q_{i}^C=\left[
    \renewcommand\arraystretch{1.5} \begin{array}{c|cc}
    0 & [\FF_{21,i}^C]^T\PP_{22,i}^C+\eta_{i}\FF_{12,i}^C \\
    \hline
    [\FF_{12,i}^C]^T\eta_{i}+\PP_{22,i}^C\FF_{21,i}^C& [\FF_{22,i}^C]^T\PP_{22,i}^C+\PP_{22,i}^C\FF_{22,i}^C
    \end{array}\right]
    \end{equation}

        The next result shows that, Lyapunov theory certifies, at most, marginal
        stability of \eqref{eq:modelDGUgen-aug-closed}.
        
        Firstly, we recall the following elementary properties of the positive definite matrix $P_{i}^C$ and the negative semi-definite matrix $Q_{i}^C$.
        \begin{prop}
        	\label{pr:prop_Q_semi}
        	\cite{7934339} If $Q=Q^T\leq0$ and an element $q_{ii}$ on the diagonal verified $q_{ii}=0$, then
        	\begin{enumerate}[(i)]
        		\item The matrix $Q$ cannot be negative definite.
        		\item The $i$-th row and column have zero entries.
        	\end{enumerate}  
        \end{prop}        
\begin{prop}
	\label{pr:pr_P_Q}
	Matrices $P_{i}^C>0$ and $Q_{i}^C\leq0$ verifying \eqref{eq:pstruct} and \eqref{fullQ} have the following structure:
	\begin{equation}
	P_{i}^C=\label{eq:pi_diagonal}
            \left
            [ \renewcommand\arraystretch{1.3}
            \begin{array}{c|cc}
                \eta_{i} & 0 & 0 \\
                      \hline
                 0 & p_{22,i}^C & 0\\
                 0 & 0 & \frac{k_{3,i}^C}{L_{ti}^C} p_{22,i}^C\\
                   \end{array}\right],
                   \hspace{5mm}
                   Q_{i}^C=\left
                   [ \renewcommand\arraystretch{1.6}
                   \begin{array}{c|c|c}
                   0 & 0 & 0 \\
                   \hline
                   0 & 2\frac{(k_{2,i}^C-R_{ti}^C)}{L_{ti}^C} p_{22,i}^C & 0\\
                   \hline\
                   0 & 0 & 0\\
                   \end{array}\right],
   	\end{equation}
   	Moreover, for having $P_{i}^C>0$, $Q_{i}^C\leq0$ and $Q_{i}^C\neq0$, the control coefficients must verify
   	\begin{equation}
   	\label{CDGU_set}
   		\left\lbrace
   		\begin{aligned}
   		k_{1,i}^C&<1\\
   		k_{2,i}^C&<R_{ti}^C\\
   		k_{3,i}^C&>0
  		\end{aligned}
  		\right.	
   	\end{equation}
\end{prop}
\begin{proof}
		Based on \eqref{eqn:P22} and \eqref{eq:Fi}, the upper right block of \eqref{fullQ} can be written as 
		\begin{equation}	
		[\FF_{21,i}^C]^T\PP_{22,i}^C+\eta_{i}\FF_{12,i}^C \label{eq:offdiagonal_Q}= 
		\left
		[ \renewcommand\arraystretch{1.3}
		\begin{array}{c|cc}
		\frac{(k_{1,i}^C-1)}{L_{ti}^C}p_{22,i}^C+\frac{1}{C_{ti}}\eta_{i} & \frac{(k_{1,i}^C-1)}{L_{ti}^C}p_{23,i}^C 
		\end{array}\right],
		\end{equation}
		Based on Proposition \ref{pr:prop_Q_semi}, \eqref{eq:offdiagonal_Q} should be equal to zero vector which means
		\begin{subequations}
		\label{eq:condition1}
		\begin{empheq}[left=\empheqlbrace]{align}
	     \label{eq:condition1_1}
         \frac{(k_{1,i}^C-1)}{L_{ti}^C}p_{22,i}^C &= -\frac{1}{C_{ti}}\eta_{i}\\
         \label{eq:condition1_2}
          \frac{(k_{1,i}^C-1)}{L_{ti}^C}p_{23,i}^C   &= 0
         \end{empheq}	
		\end{subequations}
		Because $\eta_{i}$ is positive, one has
		\begin{subequations}
	     \label{eq:condition_detail_1}
	     \begin{empheq} [left=\empheqlbrace]{align}
	     \label{eq:condition_detail_k1}
	     k_{1,i}^C &<1\\
	     \label{eq:condition_detail_p22}
	     p_{23,i}^C &=0
	     \end{empheq}
		\end{subequations}
		From \eqref{eq:condition_detail_1}, the lower right block of \eqref{fullQ} can be rewritten as
		 \begin{equation}
	    [\FF_{22,i}^C]^T\PP_{22,i}^C+\PP_{22,i}^C\FF_{22,i}^C \label{eq:diagonal_Q}= 
		 \left
		 [ \renewcommand\arraystretch{1.3}
		 \begin{array}{c|cc}
		 2\frac{(k_{2,i}^C-R_{ti}^C)}{L_{ti}^C} p_{22,i}^C & -p_{33,i}^C+\frac{k_{3,i}^C}{L_{ti}^C} p_{22,i}^C\\
		 \hline
		  -p_{33,i}^C+\frac{k_{3,i}^C}{L_{ti}^C} p_{22,i}^C & 0\\
		 \end{array}\right],
		 \end{equation}
		 Again from Proposition \ref{pr:prop_Q_semi}, the off diagonal entities of \eqref{eq:diagonal_Q} must be equal to zero which means
		 \begin{equation}
		 \frac{k_{3,i}^C}{L_{ti}^C} p_{22,i}^C \label{eq:condition_2}=p_{33,i}^C\\ 
		 \end{equation}
		 Furthermore, based on \eqref{eq:condition_detail_p22}, \eqref{eq:condition_2} and $P_{i}^C>0$
		 \begin{equation}
		 k_{3,i}^C \label{eq:condition_2.2}>0\\
		 \end{equation}
		 Finally, for verifying $Q_{i}^C\neq0$, one has
		 \begin{equation}
		 k_{2,i}^C \label{eq:condition_2.3}< R_{ti}^C\\
		 \end{equation}
		 Thus, the $P_{i}^C$ in \eqref{eq:pi_diagonal} can be derived by substituting \eqref{eq:condition_detail_p22} and \eqref{eq:condition_2} into \eqref{eq:pstruct} and then $Q_{i}^C$ in \eqref{eq:pi_diagonal} can be derived from \eqref{eq:diagonal_Q} and \eqref{eq:condition_2}, finally \eqref{eq:condition_detail_k1}, \eqref{eq:condition_2.3} and \eqref{eq:condition_2.2} consist of the set \eqref{CDGU_set} for control coefficients. 		 
\end{proof}
An immediate consequence of Proposition \ref{pr:pr_P_Q} is the following results which will be exploited for proving the stability of the whole system through the LaSalle theorem.
	\begin{lem}
 	\label{le:le_2}
 	Let $g_{i}(w_i) = w_i^T Q_{i}^C w_i$. Under the Proposition \ref{pr:pr_P_Q}, $\forall i\in\DD^C$, only vectors $\bar w_i$ in the form
 	\begin{equation*}
 	\bar{w}_i =\left[ \begin{array}{ccc}
 	\alpha_i&
 	0  & \beta_i
 	\end{array}\right]^T
 	\end{equation*} with $\alpha_i$, $\beta_i\in\Rset$, fulfill
 	\begin{equation}
 	\label{eq:vQCv=0}
	 g_i(\bar{w}_i) = \bar{w}_i^T Q_{i}^C \bar{w}_i=0.
 	\end{equation}	
\end{lem}         
          Now the overall closed-loop model with multiple CDGUs is considered as 
            \begin{equation}
            	\label{eq:sysaugoverallclosed}
            	\left\lbrace
            	\begin{aligned}
            		\mbf{\dot{\hat{x}}}^C(t) &= (\mbf{\hat{A}^C+\hat{B}^CK^C})\mbf{\hat{x}}^C(t)+ \mbf{\hat{M}^C\hd^C}(t)\\
            		\mbf{z}^C(t)       &= \mbf{\hat{H}^C\hat{x}^C}(t)
            	\end{aligned}
            	\right.
            \end{equation}
            obtained by combining \eqref{eq:sysaugoverall_1} and \eqref{eq:ctrldec}, with
            $\mbf{K}^C=\diag(K_{1}^C,\dots,K_{N}^C)$. Also the collective Lyapunov function
            \begin{equation}
            	\label{eq_coll_lyap}
            	\VV^C(\mbf{\hx^C})=\sum_{i=1}^N\VV_{i}^C(\hat{x}_{[i]}^C)=\mbf{[\hx^C]}^T\mbf{P^C}\mbf{\hx^C}
            \end{equation}
            is considered, where $\mbf{P^C}=\diag(P_{1}^C,\dots,P_{N}^C)$. 
            
            One has $\dot \VV^C(\mbf{\hx^C})=\mbf{[\hx^C]}^T\mbf{Q^C}\mbf{\hx^C}$ where
            \begin{equation}
            	\nonumber
            	\mbf{Q}^C = (\mbf{\hat{A}^C}+\mbf{\hat{B}^CK^C})^T
            	\mbf{P^C}+\mbf{P^C}(\mbf{\hat{A}^C}+\mbf{\hat{B}^CK^C}).
            \end{equation}
            A consequence of Proposition \ref{pr:pr_P_Q} is that, the matrix  $\mbf{Q^C}$
            cannot be negative definite. At most, one has
            \begin{equation}
            	\label{eq:Lyapeqnoverall}
            	\mbf{Q^C} \leq 0.
            \end{equation}
            Moreover, even if $Q_{i}^C\leq 0$ holds for all $i\in\DD^C$, the
            inequality \eqref{eq:Lyapeqnoverall} might be violated because of the nonzero
            coupling terms $\hat{A}_{ij}^C$ and load terms $\hat{A}_{load,i}^C$ in matrix $\mbf{\hat A}$.
            The next result shows that this cannot happen if \eqref{eq:equal_ratio} holds.
            \begin{prop}
               	\label{pr:semidefinite_abc}
            	If gains $K_{i}^C$ are chosen according to the \eqref{CDGU_set} in Proposition \ref{pr:pr_P_Q} and \eqref{eq:equal_ratio} holds, then  \eqref{eq:Lyapeqnoverall} holds.
            \end{prop}
            \begin{proof}
            	Consider the following decomposition of matrix $\mbf{\hat A^C}$
            	\begin{equation}
            		\label{eq:decomposition}
            		\mbf{\hat A^C} = \mbf{\hat A_D^C}+\mbf{\hat A_{\Xi}^C} + \mbf{\hat{A}_{L}^C} + \mbf{\hat A_{C}^C},
            	\end{equation}
            	where $\mbf{\hat{A}_{D}^C}=\diag(\hat{A}_{ii}^C,\dots,\hat{A}_{NN}^C)$
            	collects the local dynamics only, $\mbf{\hat A_{C}^C}$ collects the coupling dynamic representing the off-diagonal items of matrix $\mbf{\hat A^C}$, while $\mbf{\hat{A}_{\Xi}^C}=\diag(\hat{A}_{\xi 1}^C,\dots,\hat{A}_{\xi
            		N}^C)$ and $\mbf{\hat{A}_{L}^C}=\diag(\hat{A}_{load,1}^C,\dots,\hat{A}_{load,N}^C)$ with 
            	\begin{equation*}
            		\renewcommand\arraystretch{1.5}
            		\hat{A}_{\xi i}^C=\begin{bmatrix}
            			-\sum\limits_{j \in \NN_i}\frac{1}{R_{ij}C_{ti}} &0&0\\
            			0 & 0 &0\\
            			0 & 0 &0
            		\end{bmatrix},
            		\hat{A}_{load,i}^C=\begin{bmatrix}
            		-\frac{1}{R_{Li}C_{ti}} &0&0\\
            		0 & 0 &0\\
            		0 & 0 &0
            		\end{bmatrix},
            	\end{equation*} 
            	takes into account the dependence of each local state on the neighboring CDGUs and the local resistive load.
            	According to the decomposition \eqref{eq:decomposition}, the inequality \eqref{eq:Lyapeqnoverall} is equivalent to
            	\begin{small}
            		\begin{equation}
     		            			\label{eq:Lyap_abc}
            				\underbrace{\mbf{(\hat{A}_{D}^C+\hat{B}^CK^C)^TP^C+
            						P^C(\hat{A}_{D}^C+\hat{B}^CK^C)}}_{({a})}+\underbrace{\mbf{2({\hat{A}_{\Xi}^C}+\mbf{\hat{A}_{L}^C})P^C}}_{(b)}+\underbrace{\mbf{(\hat A_{C}^C)^T
            					P^C+P^C\hat{A}_{C}}}_{(c)}\leq 0.            				
            		\end{equation}
            	\end{small}
            	By means of $Q_{i}^C\leq 0$, matrix $(a) =
            	\mathrm{diag}(Q_{1}^C,\dots,Q_{N}^C)$ is negative semidefinite. Then
            	the contribution of $(b)+(c)$ in \eqref{eq:Lyap_abc} is studied. Matrix $(b)$,
            	by construction, is block
            	diagonal and collects on its diagonal blocks in the form
            	\begin{equation}
            		\label{eq:element_of_b}
            			\begin{aligned}
            				2({\hat{A}_{\xi i}^C}+\hat{A}_{load,i}^C)P_{i}^C &=
            				\renewcommand\arraystretch{2}
            				\begin{bmatrix}
            					-2\frac{1}{R_{Li}C_{ti}}-2\sum\limits_{j\in\NN_i}\frac{1 }{R_{ij}C_{ti}} & 0 & 0& \\
            					0 & 0 & 0 \\
            					0 & 0 & 0\\
            				\end{bmatrix}
            				\left[ \begin{array}{c|c}
            					\eta_i & \mbf{0}_{1\times 2}  \\
            					\hline
            					\mbf{0}_{2\times 1} & \PP_{22,i}^C
            				\end{array}\right]=                   \\
            				&=    \begin{bmatrix}
            					-2\tilde\eta_{i}-2\sum\limits_{j\in\NN_i}\tilde\eta_{ij} & 0 & 0 \\
            					0 & 0 & 0 \\
            					0 & 0 & 0\\
            				\end{bmatrix},
            			\end{aligned}
            	\end{equation}
            	where 
            	\begin{equation}
            		\label{eq:etaij}
            		\tilde\eta_{ij} = \frac{\eta_i }{R_{ij}C_{ti}},\text{   }
            		\tilde\eta_{Li} = \frac{\eta_i }{R_{Li}C_{ti}}
            	\end{equation}
            	Considering matrix $(c)$, each the block in
            	position $(i,j)$ is equal to  
            	\begin{displaymath}
            		\left\{ \begin{array}{ll}
            			P_{i}^C\hat{A}_{ij}^C+(\hat{A}_{ji}^C)^TP_{j}^C & \hspace{7mm}\mbox{if } j\in\mathcal{N}_i \\
            			0 & \hspace{7mm}\mbox{otherwise}
            		\end{array} \right.
            	\end{displaymath}
            	where
            	\begin{equation}
            		\label{eq:element_of_c}
            			\begin{aligned}
            				P_{i}^C\hat{A}_{ij}^C+(\hat{A}_{ji}^C)^TP_{j}^C &=
            				\renewcommand\arraystretch{2}
            				\begin{bmatrix}
            					\tilde\eta_{ij}+\tilde\eta_{ji} & 0 & 0 \\
            					0 & 0 & 0 \\
            					0 & 0 & 0\\
            				\end{bmatrix}.
            			\end{aligned}
            	\end{equation}
            	From \eqref{eq:element_of_b} and \eqref{eq:element_of_c}, except for the elements in position $(1,1)$ of each $3\times 3$
            	block of $(b)+(c)$, others are equals to zero. Thus, to
            	evaluate the positive/negative definiteness of the matrix $(b)+(c)$, the $N\times N$ matrix can be equivalently considered by deleting the second and third rows and columns as
            		\begin{equation} 
            			\label{eq:fake_laplacian}
            			\LL^C = \left[\begin{array}{cccc}
            				(-2\tilde\eta_{1}-2\sum\limits_{j\in\NN_1}\tilde\eta_{1j}) &\bar\eta_{12} & \dots  & \bar\eta_{1N} \\
            				\bar\eta_{21} & \ddots &\ddots & \vdots \\
            				\vdots &\ddots & (-2\tilde\eta_{N-1}-2\sum\limits_{j\in\NN_{N-1}}\tilde\eta_{N-1j}) & \bar\eta_{N-1N}  \\
            				\bar\eta_{N1}  & \dots  & \bar\eta_{NN-1} & (-2\tilde\eta_{N}-2\sum\limits_{j\in\NN_N}\tilde\eta_{Nj}) 
            			\end{array}
            			\right]
            		\end{equation}
            	One has $\LL^C = \MM^C+\UU^C+\GG^C$, where
            		\begin{equation*}
            			\mathcal{M^C}= \begin{bmatrix}
            				-2\sum\limits_{j\in\NN_1}\tilde\eta_{1j} & 0 &\dots &
            				0\\
            				0 & -2\sum\limits_{j\in\NN_2}\tilde\eta_{2j} & \ddots  & \vdots\\
            				\vdots &\ddots& \ddots &0 \\
            				0 & \dots & 0 & -2\sum\limits_{j\in\NN_N}\tilde\eta_{Nj} 
            			\end{bmatrix},
            			\mathcal{U^C}= \begin{bmatrix}
            			-2\tilde\eta_{L1} & 0 &\dots &
            			0\\
            			0 & -2\tilde\eta_{L2} & \ddots  & \vdots\\
            			\vdots &\ddots& \ddots &0 \\
            			0 & \dots & 0 & -2\tilde\eta_{LN} 
            			\end{bmatrix},
            		\end{equation*}
            	and
            		\begin{equation}
            			\label{eq:GG}	
            			\mathcal{G^C}=\left[\begin{array}{cccc}
            				0&\bar\eta_{12} & \dots  & \bar\eta_{1N} \\
            				\bar\eta_{21} & 0 & \ddots  & \vdots \\
            				\vdots &\ddots & \ddots  & \bar\eta_{N-1N}  \\
            				\bar\eta_{N1}  & \dots & \bar\eta_{N N-1}  & 0
            			\end{array}
            			\right].
            		\end{equation}
            	Notice that each off-diagonal element $\bar\eta_{ij}$ in \eqref{eq:GG} is equal to 
            	\begin{equation}
            		\label{eq:bar_eta}
            		\bar{\eta}_{ij} =\left\{ \begin{array}{ll}
            			(\tilde\eta_{ij}+\tilde\eta_{ji})& \hspace{7mm}\mbox{if } j\in\mathcal{N}_i \\
            			0 & \hspace{7mm}\mbox{otherwise}
            		\end{array} \right.
            	\end{equation}
            	
            	At this point, from \eqref{eq:equal_ratio}, one obtains that
            	$\tilde\eta_{ij}=\tilde\eta_{ji}$ (see \eqref{eq:etaij}) and,
            	consequently, $\bar\eta_{ij} = \bar\eta_{ji}=2\tilde\eta_{ij}$ (see \eqref{eq:bar_eta}). Hence, $\LL^C$ is
            	symmetric and has non negative off-diagonal elements. It follows that $-\LL^C$ is equal to a Laplacian matrix
            	\cite{grone1990laplacian,godsil2001algebraic} plus an positive definite diagonal matrix. Thus, it verifies $\LL^C
            	< 0$ by construction. By adding the deleted second and third rows and columns in each block of $(b)+(c)$, then
            	\eqref{eq:Lyap_abc} holds.
            \end{proof}
            
           Our next goal is to show asymptotic stability of the system with multiple CDGUs using the marginal
            stability result in Proposition~\ref{pr:semidefinite_abc} together
            with LaSalle invariance theorem. To this purpose, the main result is then given in Theorem~\ref{thm:overall_stability} which relies on characterizing states $\mathbf{\hat x^C}$
            deriving $\dot{\mathcal{V}}^C(\mathbf{\hat x}^C)=0$.
            \begin{thm}
            	\label{thm:overall_stability}
            	If \eqref{eq:equal_ratio} holds and $Q_{i}^C\neq0$ and the connectivity of the graph $\GG_{el}$ is guaranteed, control coefficients are chosen according to \eqref{CDGU_set}, the origin of \eqref{eq:sysaugoverallclosed} is asymptotically stable.
            \end{thm}
            \begin{proof}
            	From Proposition \ref{pr:semidefinite_abc}, $\dot\VV^C(\mbf{\hat x}^C)$ is negative semidefinite
            	meaning that \eqref{eq:Lyapeqnoverall} holds. 
            	We aim at showing that the origin of the system with multiple CDGUs is also attractive using the LaSalle
            	invariance Theorem \cite{khalil2001nonlinear}. For this purpose, the set is computed 
            	$R^C = \{\mbf{x}^C\in\Rset^{3N} : (\mbf{x^C})^T \mbf{Q^C}\mbf{x^C}= 0 \}$ by means of the decomposition in \eqref{eq:Lyap_abc}, which
            	coincides with
            	\begin{equation}
            		\label{eq:R}
            		\begin{aligned}
            			R^C &= \{\mbf{x^C} : \mbf{(x^C)}^T \left( (a)+(b)+(c)\right)\mbf{x^C}
            			= 0 \}\\
            			&=\{\mbf{x^C} : \mbf{(x^C)}^T (a) \mbf{x^C} +\mbf{(x^C)}^T(b)\mbf{x^C}+\mbf{(x^C)}^T(c) \mbf{x^C}
            			= 0 \}\\
            			&=\underbrace{\{\mbf{x^C} : \mbf{(x^C)}^T (a) \mbf{x^C} =0\}}_{X_{1}^C}\cap \underbrace{\{\mbf{x^C}:\mbf{(x^C)}^T\left[(b)+(c)\right] \mbf{x^C} =0\}}_{X_{2}^C} .
            		\end{aligned}
            	\end{equation}
            	In particular, the last equality follows from the fact that $(a)$ and $(b)+(c)$ are negative semidefinite matrices (see the proof of Proposition \ref{pr:semidefinite_abc}).
                
                First, based on Lemma \ref{le:le_2}, the set $X_{1}^C$ is characterized as 
            	\begin{small}
            		\begin{equation}
            		\label{eq:X1_c}
            		X_{1}^C = \{\mbf{x^C}:\mbf{x^C} =\left[ \text{ }\alpha_1\text{ } 0
            		\text{ } \beta_1 \text{ }|\text{ }\cdots \text{ }| \text{ }\alpha_N\text{ }0
            		\text{ } \beta_N \text{ }\right]^T, \alpha_i,\beta_i\in\Rset\},
            		\end{equation}
            	\end{small}
            	Then, we focus on the elements of set $X_{2}^C$ based on Proposition \ref{pr:semidefinite_abc}. Since matrix $(b)+(c)$ can be seen as an "expansion" of a matrix which is negative definite matrix with zero entries on the second and third rows and columns of $3\times 3$ block, by construction, vectors in the form 
                \begin{small}
            	\begin{equation}
            	\label{eq:X2_c}
            	X_{2}^C = \{\mbf{x^C}:\mbf{x^C} =\left[ \text{ }0\text{ } \tilde x_{12}
            	\text{ } \tilde x_{13} \text{ }|\text{ }\cdots \text{ }| \text{ }0 \text{ }\tilde x_{N2}
            	\text{ } \tilde x_{N3} \text{ }\right]^T, \tilde x_{i2},\tilde x_{i3}\in\Rset\},
            	\end{equation}
            \end{small}
            \normalsize
             Hence, by merging \eqref{eq:X1_c} and \eqref{eq:X2_c}, and from \eqref{eq:R}, it derives that
            	\begin{small}
            		\begin{equation}
            		\label{eq:R_new}
            		R^C = \{\mbf{x}:\mbf{x} =\left[  \text{ }0\text{ }0
            		\text{ } \beta_1 \text{ }|\text{ }\cdots \text{ }| \text{ }0\text{ }0
            		\text{ } \beta_N \text{ }\right]^T, \beta_i\in\Rset\}.
            		\end{equation}
            	\end{small}
            	Finally, in order to conclude the proof, it should be shown that the largest
            	invariant set $M^C\subseteq R$ is the origin. To this purpose, \eqref{eq:modelDGUgen-aug-closed} is considered, by adding the coupling terms $\subss{\hat\xi}i$ and the resistance load term $\hat{A}_{load,i}^C\hat{x}_{i}^C(0)$, setting load disturbance $\hat d_{[i]}^C= 0$, choosing the initial
            	state $\mbf{\hat x}^C(0) = \left[ \hat x_{1}^C(0)|\dots|\hat
            	x_{N}^C(0)\right]^T\in R^C$. In order to find conditions on the
            	elements of $\mbf{\hat x}^C(0)$ that must hold for having
            	$\mbf{\dot{\hat{x}}^C}\in R^C$, one has
            	\begin{equation*}
            		\begin{aligned}
            			\dot{\hat x}_{i}^C(0) &={F_{i}^C}\hat
            			x_{i}^C(0)+\hat{A}_{load,i}^C\hat{x}_{i}^C(0)+\sum\limits_{j\in\NN_i}\underbrace{\hat A_{ij}^C\left(\hat
            				x_{j}^C(0)-\hat x_{i}^C(0)\right)}_{=0}\\
            			&=\left[\begin{array}{ccc}
            				-\frac{1}{R_{Li}C_{ti}} & \frac{1}{C_{ti}} & 0 \vspace{2mm}\\ 
            				\frac{k_{1,i}^C-1}{L_{ti}^C} & \frac{k_{2,i}^C-R_{ti}^C}{L_{ti}^C} & \frac{k_{3,i}^C}{L_{ti}^C} \vspace{2mm}\\
            				0 & -1 & 0\\
            			\end{array}\right]\left[\begin{array}{c}
            				0 \vspace{2mm}\\
            				0 \vspace{2mm}\\
            				\beta_i
            			\end{array}\right]
            			\\
            			&=\left[\begin{array}{c}
            				0  \vspace{2mm}\\
            				\frac{k_{3,i}^C}{L_{ti}^C}\beta_i \vspace{2mm}\\
            				0
            			\end{array}\right]
            			\normalsize
            		\end{aligned}
            	\end{equation*}
            	for all $i\in\DD^C$. It follows that $\mbf{\dot{\hat{x}}^C}(0)\in R$ only if
            	$\beta_i = 0$. Since $M^C\subseteq R$, from \eqref{eq:R_new} one has $M^C = \{0\}$.
            \end{proof}          
            \begin{rmk}
            		\label{rmk:_C}
            		The design of stabilizing controller for each CDGU can be conducted according to Proposition \ref{pr:pr_P_Q}. In particular, differently from the approach in \cite{7934339}, no optimization problem has to be solved for computing a local controller. Indeed, it is enough to choose control coefficient $k_{1,i}^C$, $k_{2,i}^C$ and $k_{3,i}^C$ from inequality set \eqref{CDGU_set}. Note that these inequalities are always feasible, implying that a stabilizing controller always exists. Moreover, the inequalities depend only on the parameter $R_{ti}^C$ of the CDGU $i$. Therefore, the control synthesis is independent of parameters of CDGUs and power lines which means that controller design can be executed only once for each CDGU in a plug-and play fashion. From Theorem \ref{thm:overall_stability}, local controllers also guarantee stability of the whole MG. When new CDGUs are plugged in the MG, their controller are designed as described above, the connectivity of the electrical graph $\GG_{el}$ is preserved and have Theorem \ref{thm:overall_stability} applied to the whole MG. Instead, when a CDGU is plugged out, the electrical graph $\GG_{el}$ might be disconnected and split into two connected graphs. Theorem \ref{thm:overall_stability} can still be applied to show the stability of each sub-MG.
            \end{rmk}
\section{DC MG with Grid-Forming/Feeding Converters and Its Clusters}
\label{PV/Battery_Model}
\subsection{Electrical model of one MG}

As mentioned before, the CDGU should be cooperative operated with voltage support in the MGs. The ESS is interfaced with the MG by means of the grid-forming converter of VDGU to provide necessary voltage support for the PCC bus based on which, the RES is interfaced with the MG through the grid-feeding converters of CDGU to provide current for the loads. Thus, in this section, the combination of oen VDGU and one CDGU is considered as one MG through connecting to the same common bus achieving both voltage support and current feeding simultaneously. And the MG clusters are formed by interconnecting several MGs through line impedances.  

Here, a MG cluster system composed of $N$ MGs is considered belonging to set $\DD=\{1,\dots,N\}$. Two MGs are neighbors if there is a
power line connecting them.
$\NN_{i}\subset\DD$ denotes the subset of neighbors of MG
$i$. The neighboring relation is symmetric which means
$j\in\NN_{i}$ implies
$i\in\NN_{j}$.

The electrical scheme of the $i$-th MG is represented with Fig. \ref{fig:ctrl_part_Module} 
\begin{equation}
\label{eq:module}
\text{Module}~i:\hspace{-4mm}\quad\left\lbrace
\begin{aligned}
\frac{dV_{i}}{dt} &= \frac{1}{C_{ti}}I_{ti}^C+\frac{1}{C_{ti}}I_{ti}^V+\sum\limits_{j\in\NN_i}\left(\frac{V_j}{C_{ti} R_{ij}}-\frac{V_i}{C_{ti}R_{ij}}\right)-\frac{1}{C_{ti}}(I_{Li}+\frac{V_i}{R_{Li}})\\
\frac{dI_{ti}^C}{dt} &= -\frac{1}{L_{ti}^C}V_{i}-\frac{R_{ti}^C}{L_{ti}^C}I_{ti}^C+\frac{1}{L_{ti}^C}V_{ti}^C\\
\frac{dI_{ti}^V}{dt} &= -\frac{1}{L_{ti}^V}V_{i}-\frac{R_{ti}^V}{L_{ti}^V}I_{ti}^V+\frac{1}{L_{ti}^V}V_{ti}^V\\
\end{aligned}
\right.
\end{equation}
where variables $V_i$, $I_{ti}^C$, $I_{ti}^V$ are the $i$-th PCC voltage and
filter current from RES and filter current from ESS, respectively, $V_{ti}^C$ represents the command to the
grid-feeding converter, $V_{ti}^V$ represents the command to the
grid-forming converter, and $R_{ti}^C$, $L_{ti}^C$ the electrical
parameters for grid-feeding converter, $R_{ti}^V$, $L_{ti}^V$ the electrical parameters for grid-forming converter, $C_{ti}$ is the capacitor at the common PCC bus. Moreover, $V_{j}$ is the voltage at the PCC of each neighboring MG $j\in\NN_i$ and $R_{ij}$ and $L_{ij}$ is the resistance and inductance of the power DC line connecting MGs $i$ and $j$.
\begin{figure}
	\centering
	\hspace{-10mm}
	\includegraphics[scale=0.75]{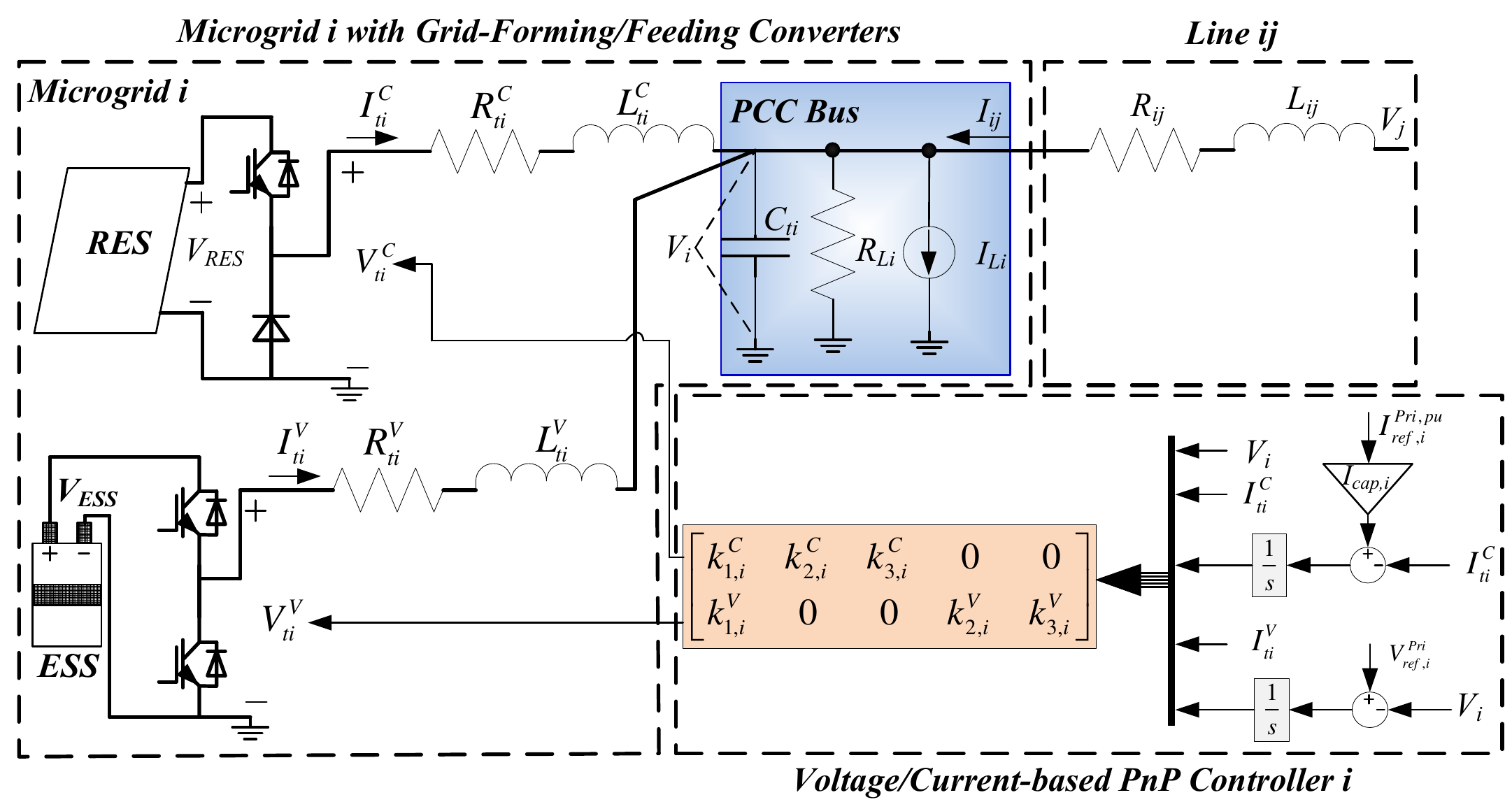}%
	\caption{Electrical Scheme of MG $i$ and voltage/current-based PnP controller.}
	\label{fig:ctrl_part_Module}
\end{figure}
\subsection{State-space model of MG clusters}
 \label{sec:State-space model of the module}
 Dynamics \eqref{eq:module} provides the state-space variables equations: 
 \begin{equation*}
 \text{ $\subss{\Sigma}{i}^{MG}:$}\left\lbrace
 \begin{aligned}
 \subss{\dot{x}}{i}(t) &= A_{ii}\subss{x}{i}(t) +
 B_{i}\subss{u}{i}(t)+M_{i}\subss{d}{i}(t)+\subss\xi{i}(t)+A_{load,i}\subss{x}{i}(t)\\
 \subss{z}{i}(t)       &= H_{i}\subss{x}{i}(t)\\
 \end{aligned}
 \right.
 \end{equation*}
 where $\subss{x}{i}=[V_{i},I_{ti}^C,I_{ti}^V]^T$ is the state of the system,
 $\subss{u}{i} = [V_{ti}^C,V_{ti}^V]$ is the control input,
 $\subss{d}{i} = I_{Li}$ is the exogenous input and
 $\subss{z}{i} = [I_{i}^C,V_{i}]$ is the controlled variable of the
 system. The term $\subss\xi i=\sum_{j\in\NN_i}A_{ij}(\subss x j-\subss x i)$ accounts
 for the coupling with each MG $j\in\NN_i$. The matrices of
 $\subss{\Sigma}{i}^{MG}$ are obtained from
 \eqref{eq:module} as:
 \begin{equation*}
 \renewcommand\arraystretch{1.8}
 A_{ii}=\begin{bmatrix}
 0 & \frac{1}{C_{ti}} & \frac{1}{C_{ti}}\\
 -\frac{1}{L_{ti}^C} & -\frac{R_{ti}^C}{L_{ti}^C} &0 \\
 -\frac{1}{L_{ti}^V} & 0 &-\frac{R_{ti}^V}{L_{ti}^V} \\
 \end{bmatrix},  \hspace{3mm} A_{load,i}=
 \begin{bmatrix}
 -\frac{1}{R_{Li}C_{ti}}  & 0 & 0 \\
 0 & 0 & 0\\
 0 & 0 & 0\\
 \end{bmatrix}, \hspace{3mm} A_{ij}=
 \begin{bmatrix}
 \frac{1}{R_{ij}C_{ti}}  & 0 & 0 \\
 0 & 0 & 0\\
 0 & 0 & 0\\
 \end{bmatrix},
 \end{equation*}
 \begin{equation*}
 B_{i}=\begin{bmatrix}
 0 & 0\\
 \frac{1}{L_{ti}^C} & 0\\
 0 & \frac{1}{L_{ti}^V}
 \end{bmatrix},
 \qquad
 M_{i}=\begin{bmatrix}
 -\frac{1}{C_{ti}} \\
 0 \\
 0 \\
 \end{bmatrix},
 \qquad
 H_{i}=\begin{bmatrix}
 0 & 1 & 0\\
 1 & 0 & 0\\ 
 \end{bmatrix}.
 \end{equation*}    
 The overall model of MG clusters is given by
 \begin{equation}
 \begin{aligned}
 \label{eq:stdmoduleA}
 \mbf{\dot{x}}(t) &= \mbf{Ax}(t) + \mbf{Bu}(t)+ \mbf{Md}(t)\\
 \mbf{z}(t)       &= \mbf{Hx}(t)
 \end{aligned}
 \end{equation}
 where $\mbf {x} = (\subss x 1,\ldots,\subss x
 N)\in\Rset^{3N}$, $\mbf {u} = (\subss u 1,\ldots,\subss u
 N)\in\Rset^{2N}$, $\mbf {d^C} = (\subss d 1,\ldots,\subss d
 N)\in\Rset^{N}$, 
 $\mbf {z} = (\subss z 1,\ldots,\subss z
 N)\in\Rset^{2N}$. Matrices $\mbf{A}$, $\mbf{B}$, $\mbf
 {M}$ and $\mbf {H}$ are reported in Appendix \ref{sec:AppMatrices_DGM}.
   \section{Design of stabilizing voltage/current controllers}
  \label{sec:PnPctr_for_V_Cl}
  \subsection{Structure of PnP Voltage/Current controllers}
  \label{sec:aug_sys_for_V_C}
  
In order to track constant references $\mbf{z_{ref}}(t)$, when
  $\mbf{d}(t)$ is constant as well, the MG model is augmented with integrators \cite{Skogestad1996}. A
  necessary condition for making
  error $\mbf{e}(t)=\mbf{z_{ref}}(t)-\mbf{z}(t)$ equal to zero as $t\rightarrow\infty$, is that, for arbitrary
  $\mbf{\bar d}$ and $\mbf{\bar
  	z_{ref}}$, there are equilibrium states and inputs $\mbf{\bar
  	x}$ and $\mbf{\bar u}$ verifying \eqref{eq:stdmoduleA}. The existence of these equilibrium points can be proven by following the proof of
  Proposition 1 in \cite{7419890}.
  
  The dynamics of integrators are (as shown in Fig. \ref{fig:ctrl_part_Module}, where $z_{ref_{[i]}}^{Pri,C}=I_{ref,i}^{Pri,pu}*I_{cap,i}$ , $z_{ref_{[i]}}^{Pri,V}={V_{ref,i}^{Pri}}$)
   \begin{subequations}
  	\label{eq:intdynamics}
  	\begin{empheq}[left=\empheqlbrace]{align}
  	\label{eq:inte_V}
  	\subss{\dot{v}}{i}^C(t) &= \subss{e}{i}^C(t) = z_{ref_{[i]}}^{Pri,C}(t)-I_{ti}^C(t)\\
  	\label{eq:inte_C}
  	\subss{\dot{v}}{i}^V(t) &= \subss{e}{i}^V(t) = z_{ref_{[i]}}^{Pri,V}(t)-V_{i}(t)
  	\end{empheq}	
  \end{subequations}
  and hence, the augmented model is
  \begin{equation}
  \label{eq:modelmodule-aug}
  \subss{\hat{\Sigma}}{i}^{MG} :
  \left\lbrace
  \begin{aligned}
  \subss{\dot{\hat{x}}}{i}(t) &= \hat{A}_{ii}\subss{\hat{x}}{i}(t) + \hat{B}_{i}\subss{u}{i}(t)+\hat{M}_{i}\subss{\hat{d}}{i}(t)+\subss{\hat\xi}i(t)+\hat{A}_{load,i}\subss{\hat{x}}{i}(t)\\
  \subss{z}{i}(t)       &= \hat{H}_{i}\subss{\hat{x}}{i}(t)
  \end{aligned}
  \right.
  \end{equation}
  where $\subss{\hat{x}}{i}=[V_{i},I_{ti}^C,v_{i}^C,I_{ti}^V
  ,v_{i}^V]^T\in\Rset^5$ is the state,
  $\subss{\hat{d}}{i}=[\subss{d}{i},\subss{z_{ref}}{i}^{Pri,C},\subss{z_{ref}}{i}^{Pri,V}]^T\in\Rset^3$
  collects the exogenous signals and
  $\subss{\hat\xi}{i}=\sum_{j\in\NN_i}\hat{A}_{ij}(\subss{\hat{x}}{j}-\subss{\hat{x}}{i})$. Matrices
  in \eqref{eq:modelmodule-aug} are defined as follows
  \begin{equation*}
  \begin{aligned}
  \hat{A}_{ii} &=\begin{bmatrix} 
  \renewcommand\arraystretch{1.5}
   0 & \frac{1}{C_{ti}} &0 &\frac{1}{C_{ti}} & 0\\
  -\frac{1}{L_{ti}^C} & -\frac{R_{ti}^C}{L_{ti}^C} &0 &0 &0 \\
  0 &-1 &0 &0 &0\\
  -\frac{1}{L_{ti}^V} & 0 & 0 &-\frac{R_{ti}^V}{L_{ti}^V}&0 \\
  -1 &0 &0 &0 &0\\
  \end{bmatrix},
  \hat{B}_{i}=\begin{bmatrix}
  0 & 0\\
  -\frac{1}{L_{ti}^C} & 0\\
  0 & 0\\
  0 & -\frac{1}{L_{ti}^V}\\
  0 & 0\\
  \end{bmatrix},
  \hat{M}_{i}=\begin{bmatrix}
  -\frac{1}{C_{ti}} & 0 & 0\\
  0 & 0 & 0 \\
  0 & 1 & 0 \\
  0 & 0 & 0 \\
  0 & 0 & 1 \\
  \end{bmatrix},\hspace{5mm}\\
  \hat{A}_{ij}&=\begin{bmatrix}
  A_{ij} & \mbf{0}_{1\times 2}\\
  \mbf{0}_{2\times 1}&\mbf{0}_{2\times 2}\\
  \end{bmatrix},
  \hat{A}_{load,i}=\begin{bmatrix}
  A_{load,i} &\mbf{0}_{1\times 2}\\
  \mbf{0}_{2\times 1}&\mbf{0}_{2\times 2}\\
  \end{bmatrix}, 
  \hat{H}_{i}=\begin{bmatrix}
  H_{i} & \mbf{0}_{2\times 2}
  \end{bmatrix}.
  \end{aligned}
  \end{equation*}   
  Based on Proposition 2 in \cite{7419890}, it can be proven that the
  pair $(\hat{A}_{ii},\hat{B}_{i})$ is controllable. Hence,
  system \eqref{eq:modelmodule-aug} can be stabilized.
   
  The overall augmented system is obtained from \eqref{eq:modelmodule-aug} as
  \begin{equation}
  \label{eq:sysaugmoduleoverall}
  \left\lbrace
  \begin{aligned}
  \mbf{\dot{\hat{x}}}(t) &= \mbf{\hat{A}\hat{x}}(t) + \mbf{\hat{B}u}(t)+ \mbf{\hat{M}\hat{d}}(t)\\
  \mbf{z}(t)       &= \mbf{\hat{H}\hat{x}}(t)
  \end{aligned}
  \right.
  \end{equation}
  where $\mbf{\hat{x}}$ and $\mbf{\hat{d}}$ collect variables $\subss{\hat{x}}{i}$ and $\subss{\hat{d}}{i}$ respectively, and matrices $\mbf{\hat{A}}, \mbf{\hat{B}}, \mbf{\hat{M}}$ and $\mbf{\hat{H}}$ are obtained from systems \eqref{eq:modelmodule-aug}. 
  
  Each MG $\subss{\hat{\Sigma}}{i}^{MG}$ is with the following state-feedback controller
  \begin{equation}
  \label{eq:ctrldec_module}
  \subss{\CC}{i}^{MG}:\qquad \subss{u}{i}(t)=K_{i}\subss{\hat{x}}{i}(t)
  \end{equation}
  where 
  \begin{equation*}
  	 K_{i}=\begin{bmatrix} 
  	 k_{1,i}^C & k_{2,i}^C & k_{3,i}^C & 0 &0\\
  	 k_{1,i}^V & 0 & 0 & k_{2,i}^V & k_{3,i}^V
  	 \end{bmatrix}\in\Rset^{2\times5}.
  \end{equation*}
  Noting that the control variables $V_{ti}^C$ and $V_{ti}^V$ are coupled through the coefficients $k_{1,i}^C$ and $k_{1,i}^V$ appearing in the first column of $K_{i}$. In other words, measurement of $V_{i}$ are used for generating both $V_{ti}^C$ and $V_{ti}^V$. It turns out that, together with the integral actions \eqref{eq:intdynamics}, controllers $\subss{\CC}{i}^{MG}$, define a multivariable PI regulator, see Fig. \ref{fig:ctrl_part_Module}. In particular, the overall control architecture is
  decentralized since the computation of
  $\subss{u}{i}$ requires the state of
  $\subss{\hat{\Sigma}}{i}^{MG}$ only. In the sequel, we show how structured Lyapunov functions can be used to ensure asymptotic stability of the MG clusters, when MGs are equipped with controllers \eqref{eq:ctrldec_module}.
  	  \subsection{Conditions for stability of the closed-loop with MG Clusters}
  \label{sec:pnp_design_modules}
    \begin{assum} 
  	\label{ass:ctrl_module}
  	As same in Section \ref{sec:pnp_design}, we will use local structured Lyapunov function 
  	\begin{equation}
  	\label{eq_sep_lyap_m}
  	V_{i}(\subss{\hat{x}}{i})=[\subss{\hat{x}}{i}]^TP_{i}\subss{\hat{x}}{i}
  	\end{equation} 
  	where the positive definite matrix $P_{i}\in\Rset^{5\times5}$ has the structure
  	\begin{equation}
  	\label{eq:pstruct_module}
  	P_{i}=\left[
  	\renewcommand\arraystretch{1.5}
  	\begin{array}{c|c|c}
  	\eta_{i} & \mbf{0}_{1\times 2} & \mbf{0}_{1\times 2}\\
  	\hline
  	\mbf{0}_{2\times 1}  & \PP_{22,i}^C & \mbf{0}_{2\times 2}\\
  	\hline
  	\mbf{0}_{2\times 1}  & \mbf{0}_{2\times 2} & \PP_{44,i}^V 
  	\end{array}\right],
  	\end{equation}
  	where 
  	\begin{equation}
  	\label{eqn:P22_P44}
  	\PP_{22,i}^C=
  	\left[ \begin{array}{cc}
  	p_{22,i}^C & p_{23,i}^C  \\
  	p_{23,i}^C & p_{33,i}^C
  	\end{array}\right],
  	\PP_{44,i}^V=
  	\left[ \begin{array}{cc}
  	p_{44,i}^V & p_{45,i}^V  \\
  	p_{54,i}^V & p_{55,i}^V
  	\end{array}\right].
  	\end{equation}
  	And $\eta_{i}>0$ is a local parameter and satisfy the eq. \eqref{eq:equal_ratio}.
  \end{assum}
  In absence of coupling terms $\subss{\hat\xi}i(t)$,and load terms $\hat{A}_{load,i}\subss{\hat{x}}{i}(t)$, we
  would like to guarantee asymptotic stability of the nominal closed-loop model
  \begin{equation}
  \label{eq:modelmodule-aug-closed}
  \subss{\dot{\hat{x}}}{i}(t) = \underbrace{(\hat{A}_{ii}+\hat{B}_{i}K_i)}_{F_{i}}\subss{\hat{x}}{i}(t)+\hat{M}_{i}\subss{\hd}{i}(t).\\
  \end{equation}
  By direct calculation, one can show that $F_{i}$ has the following structure
  \begin{equation}
  \begin{aligned}
  \label{eq:Fi_module}
  \renewcommand\arraystretch{3}
  F_{i}=&\left[ \begin{array}{c|cc|cc}
  0 & f_{12,i} & 0 & f_{14,i} & 0\\
  \hline
  f_{21,i} & f_{22,i} & f_{23,i} & 0 & 0\\
  0 & f_{32,i} & 0 & 0 & 0\\
  \hline
  f_{41,i}  & 0 & 0& f_{44,i} & f_{45,i}\\
  f_{51,i}  & 0 & 0& 0 & 0
  \end{array}\right]\\=&
  \left[
  \renewcommand\arraystretch{1.8}
  \begin{array}{c|cc|cc}
  0 & \frac{1}{C_{ti}} & 0 & \frac{1}{C_{ti}} & 0\\
  \hline
  \frac{(k_{1,i}^C-1)}{L_{ti}^C} & \frac{(k_{2,i}^C-R_{ti}^C)}{L_{ti}^C}& \frac{k_{3,i}^C}{L_{ti}^C} & 0 & 0\\
  0 & -1 & 0 & 0 & 0\\
  \hline
  \frac{(k_{1,i}^V-1)}{L_{ti}^V} & 0 & 0 &\frac{(k_{2,i}^V-R_{ti}^V)}{L_{ti}^V}& \frac{k_{3,i}^V}{L_{ti}^V}  \\
  -1 & 0 & 0 & 0 & 0
  \end{array}\right]=\left[
  \renewcommand\arraystretch{1.5} \begin{array}{c|c|c}
  0 & \FF_{12,i}^C & \FF_{14,i}^V \\
  \hline
  \FF_{21,i}^C & \FF_{22,i}^C &0\\
  \hline
  \FF_{41,i}^V & 0 & \FF_{44,i}^V
  \end{array}\right].
  \end{aligned}
  \end{equation}
  From Lyapunov theory, asymptotic stability of \eqref{eq:modelmodule-aug-closed} can be certified by the existence of a Lyapunov function $\VV_i(\subss \hx i)=[\subss \hx i] ^T P_i \subss \hx i$ where  $P_{i}\in\Rset^{5\times5}$, $P_{i} = P_{i}^T>0$ and 
  \begin{equation}
  \label{eq:Lyapeqnith_module}
  Q_{i} = F_{i}^T
  P_{i}+P_{i}F_{i} 
  \end{equation}
  is negative definite. 
  In presence of nonzero coupling terms, we will show that asymptotic stability can be achieved under Assumption \ref{ass:ctrl_module}.
%
  
  Based on  \eqref{eq:pstruct_module} and \eqref{eq:Fi_module}, the \eqref{eq:Lyapeqnith_module} can be rewritten as
        \begin{equation}
\label{fullQ_module}
\begin{aligned}
Q_{i}=&\left[
\renewcommand\arraystretch{1.5} \begin{array}{c|c|c}
0 & [\FF_{21,i}^C]^T\PP_{22,i}^C+\eta_{i}\FF_{12,i}^C & [\FF_{41,i}^V]^T\PP_{44,i}^V+\eta_{i}\FF_{14,i}^V\\
\hline
[\FF_{12,i}^C]^T\eta_{i}+\PP_{22,i}^C\FF_{21,i}^C & [\FF_{22,i}^C]^T\PP_{22,i}^C+\PP_{22,i}^C\FF_{22,i}^C & \mbf{0}_{2\times 2}\\
\hline
[\FF_{14,i}^V]^T\eta_{i}+\PP_{44,i}^V\FF_{41,i}^V & \mbf{0}_{2\times 2} & [\FF_{44,i}^V]^T\PP_{44,i}^V+\PP_{44,i}^V\FF_{44,i}^V
\end{array}\right]\\
=&\left[
\renewcommand\arraystretch{1.5} \begin{array}{c|c|c}
0 & \QQ_{12,i}^C & \QQ_{14,i}^V\\
\hline
[\QQ_{12,i}^C]^T & \QQ_{22,i}^C & \mbf{0}_{2\times 2}\\
\hline
[\QQ_{14,i}^V]^T & \mbf{0}_{2\times 2} & \QQ_{44,i}^V
\end{array}\right]\\
\end{aligned}
\end{equation} 
\begin{lem}
	\label{le:negative_semidefinite_matrix}
	Under Assumption \ref{ass:ctrl_module}, if $Q_{i}\leq 0$, $Q_{i}$ has the following structure
	\begin{equation}
	\label{eq:negative_semidefinite_matrix}
	Q_{i}=\left[
	\renewcommand\arraystretch{1.5} \begin{array}{c|c|c}
	0 & \mbf{0}_{1\times 2} & \mbf{0}_{1\times 2}\\
	\hline
	\mbf{0}_{2\times 1} & \QQ_{22,i}^C & \mbf{0}_{2\times 2}\\
	\hline
	\mbf{0}_{2\times 1} & \mbf{0}_{2\times 2} & \QQ_{44,i}^V
	\end{array}\right]\\	
	\end{equation}
	Furthermore, the diagonal block matrix must verify
	\begin{subequations}
		\label{eq:Q_features}
		\begin{empheq}[left=\empheqlbrace]{align}
		\label{eq:Q_22_C_module}
		\QQ_{22,i}^C &\leq 0\\
		\label{eq:Q_44_V_module}
		\QQ_{44,i}^V &\leq 0
		\end{empheq}	
	\end{subequations}
\end{lem}
\begin{proof}
	If $Q_{i}\leq 0$ is satisfied, from Proposition \ref{pr:prop_Q_semi}, the first block-row and block-column in \eqref{eq:negative_semidefinite_matrix} are null. Then $x^TQ_{i}x\leq 0$, $\forall x\in\Rset^5$. Partitioning $x$ as
	$$	x=\left[
		\renewcommand\arraystretch{1.5} \begin{array}{ccc}
		x_{11} \\ 	\tilde{x}_{2} \\ \tilde{x}_{4}\\
		\end{array}\right]\\$$
	where $\tilde{x}_{11}\in \mathbb{R}$, $\tilde{x}_{2}\in \mathbb{R}^2$, $\tilde{x}_{4}\in \mathbb{R}^2$.\\
	We obtain 
	$$ x^TQ_{i}x = {\tilde{x}_{2}}^T\QQ_{22,i}^C\tilde{x}_{2}+{\tilde{x}_{4}}^T\QQ_{44,i}^V\tilde{x}_{4}.$$
	For $\tilde{x}_{2}=0$ and $\tilde{x}_{4}\neq 0$, one has $$ x^TQ_{i}x = {\tilde{x}_{4}}^T\QQ_{44,i}^V\tilde{x}_{4}\leq 0, \text{ }\forall \tilde{x}_{4}\in\Rset^2$$
	which means 
	$$\QQ_{44,i}^V\leq 0$$
	Setting $\tilde{x}_{4}=0$ and $\tilde{x}_{2}\neq 0$, one has $$ x^TQ_{i}x = {\tilde{x}_{2}}^T\QQ_{22,i}^C\tilde{x}_{2}\leq 0,\text{ }\forall \tilde{x}_{2}\in\Rset^2$$
	which means 
	$$\QQ_{22,i}^C\leq 0$$
\end{proof}
\begin{rmk}
	\label{det_remark}
	Because the block of matrix $Q_{i}$ as $\QQ_{22,i}^C$ and $\QQ_{44,i}^V$ belong to $\mathbb{R}^{2\times2}$, based on Lemma \ref{le:negative_semidefinite_matrix}, $\QQ_{22,i}^C\leq 0$ and $\QQ_{44,i}^V\leq 0$, the determinants of $\QQ_{22,i}^C$ and $\QQ_{44,i}^V$ are nonnegative. 
\end{rmk}
\begin{prop}
	\label{prop:Struct_of_P_Q}
	Under Assumption \ref{ass:ctrl_module}, then $P_{i}$ and $Q_{i}$ have the following structure:
\begin{subequations}
	\label{eq:pi_qi_modele}
	\begin{empheq}{align}
P_{i}&=\label{eq:pi_diagonal_module}
\left
[ \renewcommand\arraystretch{1.8}
\begin{array}{c|cc|cc}
\eta_{i} & 0 & 0 & 0 & 0\\
\hline
0 & p_{22,i}^C & 0& 0 & 0\\
0 & 0 & \frac{k_{3,i}^C}{L_{ti}^C} p_{22,i}^C& 0 & 0\\
\hline
0 & 0 & 0 & \frac{L_{ti}^V}{C_{ti}^V}\frac{(k_{2,i}^V-R_{ti}^V)}{h_{i}} & \frac{L_{ti}^V}{C_{ti}}\frac{k_{3,i}^V}{h_{i}}\\
0 & 0 & 0 &\frac{L_{ti}^V}{C_{ti}}\frac{k_{3,i}^V}{h_{i}} & \frac{1}{C_{ti}}\frac{k_{3,i}^V(k_{1,i}^V-1)}{h_{i}}
\end{array}\right]\\
\hspace{2mm}
Q_{i}&=\label{eq:qi_diagonal_module}
\left
[ \renewcommand\arraystretch{1.8}
\begin{array}{c|cc|cc}
0 & 0 & 0 & 0 & 0 \\
\hline
0 & 2\frac{(k_{2,i}^C-R_{ti}^C)}{L_{ti}^C} p_{22,i}^C & 0 & 0 & 0\\
0 & 0 & 0& 0 & 0\\
\hline\
0 & 0 & 0 & 2\frac{(k_{2,i}^V-R_{ti}^V)^2}{C_{ti}h_{i}} & 2\frac{(k_{2,i}^V-R_{ti}^V)k_{3,i}^V}{C_{ti}h_{i}}\\
0 & 0 & 0 & 2\frac{(k_{2,i}^V-R_{ti}^V)k_{3,i}^V}{C_{ti}h_{i}} & 2\frac{(k_{3,i}^V)^2}{C_{ti}h_{i}}
\end{array}\right]
\end{empheq}
\end{subequations}
where $h_{i}= L_{ti}^Vk_{3,i}^V-(k_{1,i}^V-1)(k_{2,i}^V-R_{ti}^V)$.
Moreover, if $P_{i}>0$, $Q_{i}\leq0$ and $Q_{i}\neq0$, one has
\begin{equation}
\label{eq:condition_module}
\hspace{0mm}\quad\left\lbrace
\begin{aligned}
	k_{1,i}^C& <1\\
	k_{2,i}^C  &<R_{ti}^C\\
   k_{3,i}^C  &>0
\end{aligned}
\right.
\hspace{10mm}\quad\left\lbrace
\begin{aligned}
k_{1,i}^V& <1 \\
k_{2,i}^V  &<R_{ti}^V\\
 0&<k_{3,i}^V<\frac{1}{L_{ti}^V}(k_{1,i}^V-1)(k_{2,i}^V-R_{ti}^V)
\end{aligned}
\right.
\end{equation}
\begin{proof}
		Based on \eqref{eq:pstruct_module} and \eqref{eq:Fi_module}, the upper middle block of \eqref{fullQ_module} $\QQ_{12,i}^C$ can be written as 
	\begin{equation}	
	[\FF_{21,i}^C]^T\PP_{22,i}^C+\eta_{i}\FF_{12,i}^C \label{eq:offdiagonal_Q12_module}= 
	\left
	[ \renewcommand\arraystretch{1.8}
	\begin{array}{c|cc}
	\frac{(k_{1,i}^C-1)}{L_{ti}^C}p_{22,i}^C+\frac{1}{C_{ti}}\eta_{i} & \frac{(k_{1,i}^C-1)}{L_{ti}^C}p_{23,i}^C 
	\end{array}\right],
	\end{equation}
	From Proposition \ref{pr:prop_Q_semi}, $\QQ_{12,i}^C$ should be equal to zero vector which means
	\begin{subequations}
		\label{eq:condition1_module}
		\begin{empheq}[left=\empheqlbrace]{align}
		\label{eq:condition1_1_module}
		\frac{(k_{1,i}^C-1)}{L_{ti}^C}p_{22,i}^C &= -\frac{1}{C_{ti}}\eta_{i}\\
		\label{eq:condition1_2_module}
		\frac{(k_{1,i}^C-1)}{L_{ti}^C}p_{23,i}^C   &= 0
		\end{empheq}	
	\end{subequations}
	Because $\eta_{i}$ is positive, thus it derives that
	\begin{subequations}
		\label{eq:condition_detail_1_module}
		\begin{empheq} [left=\empheqlbrace]{align}
		\label{eq:condition_detail_k1_module}
		k_{1,i}^C &<1\\
		\label{eq:condition_detail_p22_module}
		p_{23,i}^C &=0
		\end{empheq}
	\end{subequations}
	With the results \eqref{eq:condition_detail_1_module}, the diagonal item of \eqref{fullQ_module} $\QQ_{22,i}^C$ can be direct recalculated as
	\begin{equation}
	[\FF_{22,i}^C]^T\PP_{22,i}^C+\PP_{22,i}^C\FF_{22,i}^C \label{eq:diagonal_Q22_module}= 
	\left
	[ \renewcommand\arraystretch{1.8}
	\begin{array}{c|cc}
	2\frac{(k_{2,i}^C-R_{ti}^C)}{L_{ti}^C} p_{22,i}^C & -p_{33,i}^C+\frac{k_{3,i}^C}{L_{ti}^C} p_{22,i}^C\\
	\hline
	-p_{33,i}^C+\frac{k_{3,i}^C}{L_{ti}^C} p_{22,i}^C & 0\\
	\end{array}\right]
	\end{equation}
	Again from Proposition 1, the off diagonal item of \eqref{eq:diagonal_Q22_module} should be equal to zero which means
	\begin{equation}
	\frac{k_{3,i}^C}{L_{ti}^C} p_{22,i}^C \label{eq:condition_2_module}=p_{33,i}^C\\ 
	\end{equation}
	Thus,based on \eqref{eq:condition_2_module} and $P_{i}>0$
	\begin{equation}
	k_{3,i}^C \label{eq:condition_2.2_module}>0\\
	\end{equation}
	From Proposition \ref{pr:prop_Q_semi}, $Q_{i}$ should be at least negative semidefinite, thus
	\begin{equation}
	k_{2,i}^C \label{eq:condition_2.3_module}< R_{ti}^C\\
	\end{equation}
	Because the upper left corner $3\times3$ matrix of $P_{i}$ is diagonal matrix and the matrix $P_{i}$ is positive definite, one has 
	\begin{equation}
	\label{p44_V}
	p_{44,i}^V>0
	\end{equation} 
	Based on \eqref{eq:pstruct_module} and \eqref{eq:Fi_module}, the off diagonal of \eqref{fullQ_module} $\QQ_{14,i}^V$ can be written as 
	\begin{equation}	
	[\FF_{41,i}^V]^T\PP_{44,i}^V+\eta_{i}\FF_{14,i}^V \label{eq:offdiagonal_Q14_module}= 
	\left
	[ \renewcommand\arraystretch{1.8}
	\begin{array}{c|cc}
	\frac{(k_{1,i}^V-1)}{L_{ti}^V}p_{44,i}^V-p_{45,i}^V+\frac{1}{C_{ti}}\eta_{i} & \frac{(k_{1,i}^V-1)}{L_{ti}^V}p_{45,i}^V-p_{55,i}^V 
	\end{array}\right],
	\end{equation}
	From Proposition \ref{pr:prop_Q_semi}, $\QQ_{14,i}^V$ is a zero vector which means
	\begin{subequations}
	\label{eq:relationship_p45_p44_p55}
	\begin{empheq}[left=\empheqlbrace]{align}
	\label{eq:relationship_p45_p44_1}
	p_{45,i}^V & =\frac{(k_{1,i}^V-1)}{L_{ti}^V}p_{44,i}^V+\frac{1}{C_{ti}}\eta_{i} \\
	\label{eq:relationship_p55_p45}
	p_{55,i}^V & = \frac{(k_{1,i}^V-1)}{L_{ti}^V}p_{45,i}^V 
	\end{empheq}	
	\end{subequations}
    Then by explicitly computation of $\QQ_{44,i}^V$, we can derive that
    \begin{equation}	
    [\FF_{44,i}^V]^T\PP_{44,i}^V+\PP_{44,i}^V\FF_{44,i}^V \label{eq:diagonal_Q44_module}= 
    \left
    [ \renewcommand\arraystretch{1.8}
    \begin{array}{c|cc}
    2\frac{(k_{2,i}^V-R_{ti}^V)}{L_{ti}^V}p_{44,i}^V & \frac{(k_{2,i}^V-R_{ti}^V)}{L_{ti}^V}p_{45,i}^V+\frac{k_{3,i}^V}{L_{ti}^V}p_{44,i}^V\\
    \hline
    \frac{(k_{2,i}^V-R_{ti}^V)}{L_{ti}^V}p_{45,i}^V+\frac{k_{3,i}^V}{L_{ti}^V}p_{44,i}^V & 2\frac{k_{3,i}^V}{L_{ti}^V}p_{45,i}^V
    \end{array}\right],
    \end{equation}
    Based on the Lemma \ref{le:negative_semidefinite_matrix} and eq. \eqref{p44_V}
    \begin{equation}
    \label{condition_K2_V}
    2\frac{(k_{2,i}^V-R_{ti}^V)}{L_{ti}^V}p_{44,i}^V\leq0 \Longrightarrow k_{2,i}^V \leq R_{ti}^V
   \end{equation}
  Computing the determinant of $\QQ_{44,i}^V$, one obtains
   \begin{equation}
   \label{determinant_of_Q44}
   	det(\QQ_{44,i}^V)=-\left[\frac{(k_{2,i}^V-R_{ti}^V)}{L_{ti}^V}p_{45,i}^V-\frac{k_{3,i}^V}{L_{ti}^V}p_{44,i}^V\right]^2   	
   \end{equation}
    Based on the Lemma \ref{le:negative_semidefinite_matrix}, the second principal minor of $\QQ_{44,i}^V$ which is also the determinant $\QQ_{44,i}^V$ is nonnegative. From \eqref{determinant_of_Q44}, the maximum value is zero, thus the determinant of $\QQ_{44,i}^V$ should be equal to zero. It follows that
    \begin{equation}
    	\label{eq:relationship_p44_p45_2}
    	\frac{(k_{2,i}^V-R_{ti}^V)}{L_{ti}^V}p_{45,i}^V=\frac{k_{3,i}^V}{L_{ti}^V}p_{44,i}^V \Longrightarrow p_{44,i}^V=\frac{(k_{2,i}^V-R_{ti}^V)}{k_{3,i}^V}p_{45,i}^V
    \end{equation}
    By solving the system of equation given by \eqref{eq:relationship_p45_p44_p55} and \eqref{eq:relationship_p44_p45_2}, it derives that
	\begin{subequations}
	 \label{eq:P44}
	 \begin{empheq}[left=\empheqlbrace]{align}
	 \label{eq:p44}
    	p_{44,i}^V & =\frac{L_{ti}^V}{C_{ti}^V}\frac{(k_{2,i}^V-R_{ti}^V)}{h_{i}} \\
  	\label{eq:p45}
 	p_{45,i}^V & = \frac{L_{ti}^V}{C_{ti}}\frac{k_{3,i}^V}{h_{i}}\\
 	\label{eq:p55}
 	p_{55,i}^V & = \frac{1}{C_{ti}}\frac{k_{3,i}^V(k_{1,i}^V-1)}{h_{i}}
	\end{empheq}	
      \end{subequations}
  where $h_{i}= L_{ti}^Vk_{3,i}^V-(k_{1,i}^V-1)(k_{2,i}^V-R_{ti}^V)$.
  
\vspace{2mm}
  Because $P_{44,i}^V$ is positive definite, all its principal minor should be positive definite. Then 
  \begin{itemize}
  	\item  $det\left(\renewcommand\arraystretch{1.8}\frac{L_{ti}^V}{C_{ti}^V}\frac{(k_{2,i}^V-R_{ti}^V)}{h_{i}}\right)>0$, combining this result with \eqref{condition_K2_V}, the feasible parameters $k_{2,i}^V$ and $h_{i}$ set should be 	
  	$Z_{1}={\{ k_{2,i}^V< R_{ti}^V \}} \cap {\{h_{i}<0\}}$
  	\item  $det\left(\left
  [ \renewcommand\arraystretch{1.8}
  	\begin{array}{c|cc}
  	\frac{L_{ti}^V}{C_{ti}^V}\frac{(k_{2,i}^V-R_{ti}^V)}{h_{i}} & \frac{L_{ti}^V}{C_{ti}}\frac{k_{3,i}^V}{h_{i}}\\
  	\hline
  	\frac{L_{ti}^V}{C_{ti}}\frac{k_{3,i}^V}{h_{i}} & \frac{1}{C_{ti}}\frac{k_{3,i}^V(k_{1,i}^V-1)}{h_{i}}
  	\end{array}\right ]\right)= -\frac{L_{ti}^VK_{3,i}^V}{C_{ti}^2h_{i}}>0$, considering this result, the feasible parameters $k_{3,i}^V$ and $h_{i}$ set should be
  	\vspace{2mm}
  	$Z_{2}=\{\{ k_{3,i}^V < 0 \} \cap \{h_{i}>0\}\}\cup\{\{ k_{3,i}^V > 0 \} \cap \{h_{i}<0\}\}$
  \end{itemize}
  By combing the $Z_{1}$ and $Z_{2}$ together, one has
  \begin{equation}
  	\label{eq:Z}
  	\pmb Z = \{Z_{1}\}\cap\{Z_{2}\} =  \{ k_{2,i}^V< R_{ti}^V \} \cap \{ k_{3,i}^V > 0 \} \cap \{h_{i}<0\}
  \end{equation}
  Because $ k_{3,i}^V > 0$, the set $\{h_{i}<0\}$ can be further split. Then, combining the set with \eqref{eq:Z}, it can derive that
  \begin{equation}
  \label{eq:Z3}
   \pmb Z=\{ k_{1,i}^V < 1 \}\cap \{ k_{2,i}^V< R_{ti}^V \}\cap \{ 0<k_{3,i}^V< \frac{1}{L_{ti}^V}(k_{1,i}^V-1)(k_{2,i}^V-R_{ti}^V) \} 
  \end{equation}
  Thus, \eqref{eq:pi_qi_modele} can be derived by combining the result in \eqref{eq:condition_detail_p22_module}, \eqref{eq:condition_2_module} and\eqref{eq:P44}. Then, combining the results in \eqref{eq:condition_detail_k1_module}, \eqref{eq:condition_2.2_module}, \eqref{eq:condition_2.3_module} and \eqref{eq:Z3}, the set for control coefficients \eqref{eq:condition_module} is derived.	
\end{proof}
\end{prop}
\begin{lem}
	\label{le:ker_F33_module}
		Let Assumptions \ref{ass:ctrl_module} and Proposition \ref{prop:Struct_of_P_Q} hold, let us define $h_i(v_i) = v_i^T\QQ_{44,i}^Vv_i$, with
		$v_i\in\mathbb{R}^2$. If $Q_{i}\leq 0$ and $Q_{i}\neq 0$ is guaranteed, then
		\begin{equation*}
		h_i(\bar{v}_i)  =0 \Longleftrightarrow\bar{v}_i\in\mathrm{Ker}(\FF_{44,i}^V).
		\end{equation*}
\end{lem}
\begin{proof}
	The proof is same as the proof for Proposition 3 in \cite{7934339}.
\end{proof}
\begin{prop}
	\label{prop:quadr_form_module}
	Let $g_i(w_i) = w_i^T Q_i w_i$. Under the same Assumptions of Lemma \ref{le:ker_F33_module}, $\forall i\in\DD$, and Proposition \ref{prop:Struct_of_P_Q} and Lemma \ref{le:ker_F33_module}, only vectors $\bar w_i$ in the form
	\begin{equation*}
	\bar{w}_i =\left[ \begin{array}{ccccc}
	\alpha_i& 0 & \gamma_i&
	\beta_i  &
	\delta_i\beta_i
	\end{array}\right]^T
	\end{equation*} with $\alpha_i$, $\gamma_i$, $\beta_i\in\Rset$, and $\delta_i =
	-\frac{k_{2,i}^V-R_{ti}^V}{k_{3,i}^V}$, fulfill
	\begin{equation}
	\label{eq:vQv=0_module}
	g_i(\bar{w}_i) = \bar{w}_i^T Q_i \bar{w}_i=0.
	\end{equation}
\end{prop}
\begin{proof}
	In the sequel, the subscript
	$i$ is omitted for convenience. From \eqref{eq:pi_diagonal_module}, $g(w)$ is equal to 
	\begin{equation}
	\label{eq:quadr_form_module}
	\left[ \begin{array}{c|c|c} 
	w_1 & {w}_2^T & {w}_3^T
	\end{array}
	\right]
\left
[ \renewcommand\arraystretch{1.6}
\begin{array}{c|cc|cc}
0 & 0 & 0 & 0 & 0 \\
\hline
0 & 2\frac{(k_{2}^C-R_{t}^C)}{L_{t}^C} p_{22}^C & 0 & 0 & 0\\
0 & 0 & 0& 0 & 0\\
\hline\
0 & 0 & 0 & q_{44}^V & q_{45}^V\\
0 & 0 & 0 & q_{45}^V & q_{55}^V
\end{array}\right]\left[ \begin{array}{c}
	w_1  \\
	\hline
	{w}_2\\
	\hline
	{w}_3
	\end{array}\right],
	\end{equation} 
	where $ w_2,w_3\in\Rset^2$. Since $Q$ is
	negative semidefinite, the 
	vectors $\bar w$ satisfying \eqref{eq:vQv=0_module} also maximize $g(\cdot)$. Hence, it must hold $ \frac{\mathrm{d}g}{\mathrm{d}w}(\bar w)
	= Q\bar w=0$, i.e.
	\begin{equation}
	\label{eq:maximum_module}
\left
[ \renewcommand\arraystretch{1.6}
\begin{array}{c|cc|cc}
0 & 0 & 0 & 0 & 0 \\
\hline
0 & 2\frac{(k_{2}^C-R_{t}^C)}{L_{t}^C} p_{22}^C & 0 & 0 & 0\\
0 & 0 & 0& 0 & 0\\
\hline\
0 & 0 & 0 & q_{44}^V & q_{45}^V\\
0 & 0 & 0 & q_{45}^V & q_{55}^V
\end{array}\right]\left[ \begin{array}{c}
	\bar{w}_1  \\
	\hline
	\bar{{w}}_2\\
	\hline
	\bar{{w}}_3
	\end{array}\right]=0.
	\end{equation}
	Based on the results in Proposition \ref{prop:Struct_of_P_Q}, it is easy to show that, by direct calculation, a set of solutions
	to \eqref{eq:vQv=0_module} and \eqref{eq:maximum_module} is composed of vectors in the form
	\begin{equation}
	\label{eq:first_solution_module}
	\bar{w} = \left[ \begin{array}{ccccc}
	\alpha &
	0&
	\gamma&
	0&
	0
	\end{array}\right]^T ,\quad\alpha,\gamma \in\Rset.
	\end{equation}
	Moreover, from \eqref{eq:quadr_form_module}, we have that \eqref{eq:vQv=0_module} is
	also verified if there exist vectors 
	\begin{equation}
	\label{eq:tilde_v_module}
	\tilde w =\left[ \begin{array}{c|c|c}
	w_1 &
	w_{2}^T &
	{\underline{w}}_3^T
	\end{array}\right]^T,\quad{\underline{w}}_3\neq [0\text{ }0]^T,
	\end{equation}
	such that $w_1\in\Rset$, $w_2\in\Rset^2$and
	\begin{equation}
	\label{eq:second_solution_module}
	{\underline{w}}_3^T
	Q_{44}^V {\underline{w}}_3 =0.
	\end{equation}
	By exploiting the result of Lemma \ref{le:ker_F33_module}, we know that vectors $\underline{w}_3$ fulfilling \eqref{eq:second_solution_module} belong to $\text{Ker}(F_{44}^V)$, which, recalling \eqref{eq:Fi_module}, can be explicitly computed as follows
	\begin{equation}
	\label{eqn:Ker_F_33_module}
	\begin{aligned}
	\mathrm{Ker}(\FF_{44}^V) &=\left\lbrace x\in\Rset^2 : \left[ \begin{array}{cc}
	f_{44}^V  & f_{45}^V\\
	 0 & 0
	\end{array}\right]x = 0\right\rbrace =\\
	&=\left\lbrace x\in\Rset^2 : x = [\text{ } \beta\text{ }\delta\beta\text{ }]^T, \beta\in\Rset,\delta = -\frac{k_{2}^V-R_{t}^V}{k_{3}^V} \right\rbrace .
	\end{aligned}
	\end{equation}
	The proof ends by merging \eqref{eq:first_solution_module} and \eqref{eq:tilde_v_module}, with ${\underline{w}}_3$ as in \eqref{eqn:Ker_F_33_module}.	
\end{proof}
  Consider the overall closed-loop MG cluster model 
  \begin{equation}
  \label{eq:sysaugoverallclosed_module}
  \left\lbrace
  \begin{aligned}
  \mbf{\dot{\hat{x}}}(t) &= (\mbf{\hat{A}+\hat{B}K})\mbf{\hat{x}}(t)+ \mbf{\hat{M}\hd}(t)\\
  \mbf{z}(t)       &= \mbf{\hat{H}\hat{x}}(t)
  \end{aligned}
  \right.
  \end{equation}
  obtained by combining \eqref{eq:sysaugmoduleoverall} and \eqref{eq:ctrldec_module}, with
  $\mbf{K}=\diag(K_{1},\dots,K_{N})$. Considering also the collective Lyapunov function
  \begin{equation}
  \label{eq_coll_module_lyap}
  \VV(\mbf{\hx})=\sum_{i=1}^N\VV_{i}(\hat{x}_{[i]})=\mbf{\hx}^T\mbf{P}\mbf{\hx}
  \end{equation}
  where $\mbf{P}=\diag(P_{1},\dots,P_{N})$. One has $\dot \VV(\mbf{\hx})=\mbf{\hx}^T\mbf{Q}\mbf{\hx}$ where
  \begin{equation}
  \mbf{Q} = (\mbf{\hat{A}}+\mbf{\hat{B}K})^T
  \mbf{P}+\mbf{P}(\mbf{\hat{A}}+\mbf{\hat{B}K}).
  \end{equation}
  A consequence of Proposition 1 is that, under Assumption \ref{ass:ctrl_module}, the matrix  $\mbf{Q}$
  cannot be negative definite. At most, one has
  \begin{equation}
  \label{eq:Lyapeqnoverall_module}
  \mbf{Q} \leq 0.
  \end{equation}
  Moreover, even if $Q_{i} \leq 0$ holds for all $i\in\DD$, the
  inequality \eqref{eq:Lyapeqnoverall_module} might be violated because of the nonzero
  coupling terms $\hat{A}_{ij}$ and load terms $\hat{A}_{load,i}$ in matrix $\mbf{\hat A}$.
  The next result shows that this cannot happen.
    \begin{prop}
  	\label{pr:semidefinite_abc_module}
  	If gains $K_{i}$ are chosen according to \eqref{eq:condition_module} and then $Q_{i}\leq0$ for all $i\in\DD$, then \eqref{eq:Lyapeqnoverall_module} holds.
  \end{prop}
  \begin{proof}
  	Consider the following decomposition of matrix $\mbf{\hat A}$
  	\begin{equation}
  	\label{eq:decomposition_module}
  	\mbf{\hat A} = \mbf{\hat A_D}+\mbf{\hat A_{\Xi}} + \mbf{\hat{A}_{L}} + \mbf{\hat A_{C}},
  	\end{equation}
  	where $\mbf{\hat{A}_{D}}=\diag(\hat{A}_{ii},\dots,\hat{A}_{NN})$
  	collects the local dynamics only, $\mbf{\hat A_{C}}$ collects the coupling dynamic representing the off-diagonal items of matrix $\mbf{\hat A}$. Meanwhile, $\mbf{\hat{A}_{\Xi}}=\diag(\hat{A}_{\xi 1},\dots,\hat{A}_{\xi
  		N})$ and $\mbf{\hat{A}_{L}}=\diag(\hat{A}_{load,1},\dots,\hat{A}_{load
  		,N})$ with 
  	\begin{equation*}
  	\renewcommand\arraystretch{1.5}
  	\hat{A}_{\xi i}=\begin{bmatrix}
  	-\sum\limits_{j \in \NN_i}\frac{1}{R_{ij}C_{ti}} &\mbf{0}_{1\times 4}\\
  	\mbf{0}_{4\times 1} & \mbf{0}_{4\times 4}
  	\end{bmatrix},
  	\hat{A}_{load,i}=\begin{bmatrix}
  	-\frac{1}{R_{Li}C_{ti}} &\mbf{0}_{1\times 4}\\
  	\mbf{0}_{4\times 1} & \mbf{0}_{4\times 4}
  	\end{bmatrix}.
  	\end{equation*} 
  	takes into account the dependence of each local state on the neighboring MGs and the local resistive load. According to the decomposition \eqref{eq:decomposition_module}, the inequality \eqref{eq:Lyapeqnoverall_module} is equivalent to show that
  	\begin{small}
  		\begin{equation}
  		\label{eq:Lyap_abc_module}
  		\underbrace{\mbf{(\hat{A}_{D}+\hat{B}K)^TP+
  				P(\hat{A}_{D}+\hat{B}K)}}_{({a})}+\underbrace{\mbf{2({\hat{A}_{\Xi}}+\mbf{\hat{A}_{L}})P}}_{(b)}+\underbrace{\mbf{\hat A_{C}^T
  				P+P\hat{A}_{C}}}_{(c)}\leq 0           				
  		\end{equation}
  	\end{small}
  	By means of Proposition \ref{pr:prop_Q_semi}, matrix $(a) =
  	\mathrm{diag}(Q_{1},\dots,Q_{N})$ is negative semidefinite. Then, the contribution of $(b)+(c)$ in \eqref{eq:Lyap_abc_module} is studied as follows. Matrix $(b)$,
  	by construction, is block
  	diagonal and collects on its diagonal blocks in the form
  	\begin{equation}
  	\label{eq:element_of_b_module}
  	\begin{aligned}
  	2({\hat{A}_{\xi i}}+\hat{A}_{load,i}^C)P_{i} &=
  	\renewcommand\arraystretch{2}
  	\begin{bmatrix}
  	-2\frac{1}{R_{Li}C_{ti}}-2\sum\limits_{j\in\NN_i}\frac{1 }{R_{ij}C_{ti}} & \mbf{0}_{1\times 4} \\
  	\mbf{0}_{4\times 1} & \mbf{0}_{4\times 4}
  	\end{bmatrix}
  	\left[ \begin{array}{c|c|c}
  	\eta_i & \mbf{0}_{1\times 2} & \mbf{0}_{1\times 2}\\
  	\hline
  	\mbf{0}_{2\times 1} & \PP_{22,i}^C & \mbf{0}_{2\times 2}\\
  	\hline
  	\mbf{0}_{2\times 1} & \mbf{0}_{2\times 2} & \PP_{44,i}^V
  	\end{array}\right]=                   \\
  	&=    \begin{bmatrix}
  	-2\tilde\eta_{i}-2\sum\limits_{j\in\NN_i}\tilde\eta_{ij} & \mbf{0}_{1\times 4} \\
  	\mbf{0}_{4\times 1} & \mbf{0}_{4\times 4}
  	\end{bmatrix}
  	\end{aligned}
  	\end{equation}
  	where 
  	\begin{equation}
  	\label{eq:etaij_module}
  	\tilde\eta_{ij} = \frac{\eta_i }{R_{ij}C_{ti}},\text{ }
  	\tilde\eta_{Li} = \frac{\eta_i }{R_{Li}C_{ti}}
  	\end{equation}
  	Considering matrix $(c)$, each the block in
  	position $(i,j)$ is equal to  
  	\begin{displaymath}
  	\left\{ \begin{array}{ll}
  (\hat{A}_{ji})^TP_{j}+P_{i}\hat{A}_{ij} & \hspace{7mm}\mbox{if } j\in\mathcal{N}_i \\
  	0 & \hspace{7mm}\mbox{otherwise}
  	\end{array} \right.
  	\end{displaymath}
  	where
  	\begin{equation}
  	\label{eq:element_of_c_module}
  	\begin{aligned}
  	P_{i}\hat{A}_{ij}+\hat{A}_{ji}^TP_{j} &=
  	\renewcommand\arraystretch{2}
  	\begin{bmatrix}
  	\tilde\eta_{ij}+\tilde\eta_{ji} & \mbf{0}_{1\times 4} \\
  	\mbf{0}_{4\times 1} & \mbf{0}_{4\times 4}
  	\end{bmatrix}.
  	\end{aligned}
  	\end{equation}
  	From \eqref{eq:element_of_b_module} and \eqref{eq:element_of_c_module}, we notice
  	that only the elements in position $(1,1)$ of each $5\times 5$
  	block of $(b)+(c)$ can be different from zero. Hence, in order to
  	evaluate the positive/negative definiteness of the $5N\times 5N$
  	matrix $(b)+(c)$, we can equivalently consider the $N\times N$ matrix as
  	\begin{equation} 
  	\label{eq:fake_laplacian_module}
  	\LL = \left[\begin{array}{cccc}
  	(-2\tilde\eta_{1}-2\sum\limits_{j\in\NN_1}\tilde\eta_{1j}) &\bar\eta_{12} & \dots  & \bar\eta_{1N} \\
  	\bar\eta_{21} & \ddots &\ddots & \vdots \\
  	\vdots &\ddots & (-2\tilde\eta_{N-1}-2\sum\limits_{j\in\NN_{N-1}}\tilde\eta_{N-1j}) & \bar\eta_{N-1N}  \\
  	\bar\eta_{N1}  & \dots  & \bar\eta_{NN-1} & (-2\tilde\eta_{N}-2\sum\limits_{j\in\NN_N}\tilde\eta_{Nj}) 
  	\end{array}
  	\right]
  	\end{equation}
  	obtained by deleting the second to fifth rows and columns in each block
  	of $(b)+(c)$. One has $\LL = \MM+\UU+\GG$, where
  	\begin{equation*}
  	\mathcal{M}= \begin{bmatrix}
  	-2\sum\limits_{j\in\NN_1}\tilde\eta_{1j} & 0 &\dots &
  	0\\
  	0 & -2\sum\limits_{j\in\NN_2}\tilde\eta_{2j} & \ddots  & \vdots\\
  	\vdots &\ddots& \ddots &0 \\
  	0 & \dots & 0 & -2\sum\limits_{j\in\NN_N}\tilde\eta_{Nj} 
  	\end{bmatrix},
  	  	\mathcal{U}= \begin{bmatrix}
  	-2\tilde\eta_{L1} & 0 &\dots &
  	0\\
  	0 & -2\tilde\eta_{L2} & \ddots  & \vdots\\
  	\vdots &\ddots& \ddots &0 \\
  	0 & \dots & 0 & -2\tilde\eta_{LN} 
  	\end{bmatrix}
  	\end{equation*}
  	and
  	\begin{equation}
  	\label{eq:GG_module}
  	\mathcal{G}=\left[\begin{array}{cccc}
  	0&\bar\eta_{12} & \dots  & \bar\eta_{1N} \\
  	\bar\eta_{21} & 0 & \ddots  & \vdots \\
  	\vdots &\ddots & \ddots  & \bar\eta_{N-1N}  \\
  	\bar\eta_{N1}  & \dots & \bar\eta_{N N-1}  & 0
  	\end{array}
  	\right].
  	\end{equation}
  	Notice that each off-diagonal element $\bar\eta_{ij}$ in \eqref{eq:GG_module} is equal to 
  	\begin{equation}
  	\label{eq:bar_eta_module}
  	\bar{\eta}_{ij} =\left\{ \begin{array}{ll}
  	(\tilde\eta_{ij}+\tilde\eta_{ji})& \hspace{7mm}\mbox{if } j\in\mathcal{N}_i \\
  	0 & \hspace{7mm}\mbox{otherwise}
  	\end{array} \right.
  	\end{equation}
  	At this point, from Assumption \ref{ass:ctrl_module}, one obtains that
  	$\tilde\eta_{ij}=\tilde\eta_{ji}$ (see \eqref{eq:etaij_module}) and,
  	consequently, $\bar\eta_{ij} = \bar\eta_{ji}=2\tilde\eta_{ij}$ (see \eqref{eq:bar_eta_module}). Hence, $\LL$ is
  	symmetric and has non negative off-diagonal elements. It follows that $-\LL$ is equal to a Laplacian matrix
  	\cite{grone1990laplacian,godsil2001algebraic} plus an positive definite diagonal matrix. As such, it verifies $\LL
  	< 0$ by construction. By adding the deleted second to fifth rows and columns in each block of $(b)+(c)$, we have shown that
  	\eqref{eq:Lyap_abc_module} holds.
  \end{proof}
   \begin{thm}
  	\label{thm:overall_stability_module}
  	If Assumptions \ref{ass:ctrl_module} is fulfilled,
  	the graph $\GG_{el}$ is connected, control coefficients are chosen according to \eqref{eq:condition_module}, the origin of \eqref{eq:sysaugmoduleoverall} is asymptotically stable.
  \end{thm}
  \begin{proof}
  	From Proposition \ref{pr:semidefinite_abc_module}, $\dot\VV(\mbf{\hat x})$ is negative semidefinite
  	(i.e. \eqref{eq:Lyapeqnoverall_module} holds). 
  	It should be shown that the origin of the MG is also attractive by using the LaSalle
  	invariance Theorem \cite{khalil2001nonlinear}. For this purpose, the set
  	$R = \{\mbf{x}\in\Rset^{5N} : (\mbf{x})^T \mbf{Q}\mbf{x}= 0 \}$ is first computed by means of the decomposition in \eqref{eq:Lyap_abc_module}, which coincides with
  	\begin{equation}
  	\label{eq:R_module}
  	\begin{aligned}
  	R &= \{\mbf{x} : \mbf{(x)}^T \left( (a)+(b)+(c)\right)\mbf{x}
  	= 0 \}\\
  	&=\{\mbf{x} : \mbf{(x)}^T (a) \mbf{x} +\mbf{(x)}^T(b)\mbf{x}+\mbf{(x)}^T(c) \mbf{x}
  	= 0 \}\\
  	&=\underbrace{\{\mbf{x} : \mbf{(x)}^T (a) \mbf{x} =0\}}_{X_{1}}\cap \underbrace{\{\mbf{x}:\mbf{(x)}^T\left[(b)+(c)\right] \mbf{x} =0\}}_{X_{2}} .
  	\end{aligned}
  	\end{equation}
  	In particular, the last equality follows from the fact that matrix $(a)$ and $(b)+(c)$ are negative semidefinite matrices based on the proof of Proposition \ref{prop:Struct_of_P_Q} and \ref{pr:semidefinite_abc_module}.
  	
  	First, we characterize the set $X_{1}$. By exploiting Proposition \ref{prop:quadr_form_module}, it follows that
  	\begin{small}
  		\begin{equation}
  		\label{eq:X1_module}
  		X_{1} = \{\mbf{x}:\mbf{x} =\left[ \text{ }\alpha_1\text{ } 0
  		\text{ } \gamma_1 \text{ }\beta_1\text{ }\delta_1\beta_1|\text{ }\cdots \text{ }| \text{ }\alpha_N\text{ } 0
  		\text{ } \gamma_N \text{ }\beta_N\text{ }\delta_N\beta_N \text{ }\right]^T, \alpha_i,\gamma_i,\beta_i\in\Rset\},
  		\end{equation}
  	\end{small}  
    Then, the elements of set $X_{2}$ can be characterized with Proposition \ref{pr:semidefinite_abc_module}. Since matrix $(b)+(c)$ can be seen as an "expansion" of a matrix which is negative definite matrix with zero entries on the second to fifth rows and columns of each $5\times 5$ block, by construction, the vectors in the form 
  	\begin{small}
  		\begin{equation}
  		\label{eq:X2_module}
  		X_{2} = \{\mbf{x}:\mbf{x} =\left[ \text{ }0\text{ } \tilde x_{12}
  		\text{ } \tilde x_{13} \text{ }\tilde x_{14} \text{ }\tilde x_{15} \text{ } |\text{ }\cdots \text{ }| \text{ }0 \text{ }\tilde x_{N2}
  		\text{ } \tilde x_{N3} \text{ }\tilde x_{N4} \text{ }\tilde x_{N5} \text{ }\right]^T, \tilde x_{i2},\tilde x_{i3},\tilde x_{i4},\tilde x_{i5}\in\Rset\},
  		\end{equation}
  	\end{small}
  	\normalsize
  	Hence, by merging \eqref{eq:X1_module} and \eqref{eq:X2_module}, it derives that
  	\begin{small}
  		\begin{equation}
  		\label{eq:R_new_module}
  		R = \{\mbf{x}:\mbf{x} =\left[  \text{ }0\text{ }0
  		\text{ } \gamma_1 \text{ }\beta_1\text{ }\delta_1\beta_1|\text{ }|\text{ }\cdots \text{ }| \text{ }0\text{ }0
  		\text{ } \gamma_N \text{ }\beta_N\text{ }\delta_N\beta_N|\text{ }\right]^T,\gamma_i, \delta_i,\beta_i\in\Rset\}.
  		\end{equation}
  	\end{small}
  	To conclude the proof, it should be shown that the largest
  	invariant set $M\subseteq R$ is the origin. To this purpose, we consider \eqref{eq:modelmodule-aug-closed}, include coupling terms $\subss{\hat\xi}i$, resistance load term $\hat{A}_{load,i}\hat{x}_{i}(0)$, set $\hat d_{[i]}= 0$ and choose as initial
  	state $\mbf{\hat x}(0) = \left[ \hat x_{1}(0)|\dots|\hat
  	x_{N}(0)\right]^T\in R$. We aim to find conditions on the
  	elements of  $\mbf{\hat x}(0)$ that must hold for having
  	$\mbf{\dot{\hat{x}}}\in R$. One has
  	\begin{equation*}
  	\begin{aligned}
  	\dot{\hat x}_{i}(0) &={F_{i}}\hat
  	x_{i}(0)+\hat{A}_{load,i}\hat{x}_{i}(0)+\sum\limits_{j\in\NN_i}\underbrace{\hat A_{ij}\left(\hat
  		x_{j}(0)-\hat x_{i}(0)\right)}_{=0}\\
  	&= \left[
  	\renewcommand\arraystretch{1.8}
  	\begin{array}{c|cc|cc}
  	-\frac{1}{R_{Li}C_{ti}} & \frac{1}{C_ti} & 0 & \frac{1}{C_ti} & 0\\
  	\hline
  	\frac{(k_{1,i}^C-1)}{L_{ti}^C} & \frac{(k_{2,i}^C-R_{ti}^C)}{L_{ti}^C}& \frac{k_{3,i}^C}{L_{ti}^C} & 0 & 0\\
  	0 & -1 & 0 & 0 & 0\\
  	\hline
  	\frac{(k_{1,i}^V-1)}{L_{ti}^V} & 0 & 0 &\frac{(k_{2,i}^V-R_{ti}^V)}{L_{ti}^V}& \frac{k_{3,i}^V}{L_{ti}^V}  \\
  	-1 & 0 & 0 & 0 & 0
  	\end{array}\right]\left[\begin{array}{c}
  	0 \vspace{3mm}\\
  	0 \vspace{3mm}\\
  	\gamma_i\vspace{3mm}\\
  	\beta_{i}\vspace{3mm}\\
  	\delta_i\beta_i\vspace{3mm}
  	\end{array}\right]
  	\\
  	&=\left[\begin{array}{c}
  	\frac{\beta_i}{C_{ti}}  \vspace{3mm}\\
  	\frac{k_{3,i}^C}{L_{ti}^C}\gamma_i \vspace{3mm}\\
  	0\vspace{3mm}\\
  	\underbrace{\frac{k_{2,i}^V-R_{ti}^V}{L_{ti}^V}\beta_i+\frac{k_{3,i}^V}{L_{ti}^V}\delta_i\beta_i}_{=0}\vspace{3mm}\\
  	0\vspace{3mm}
  	\end{array}\right]=\left[\begin{array}{c}
  	\frac{\beta_i}{C_{ti}}  \vspace{3mm}\\
  	\frac{k_{3,i}^C}{L_{ti}^C}\gamma_i \vspace{3mm}\\
  	0\vspace{3mm}\\
  	0\vspace{3mm}\\
  	0\vspace{3mm}
  	\end{array}\right]
  	\normalsize
  	\end{aligned}
  	\end{equation*}
  	for all $i\in\DD$. It follows that $\mbf{\dot{\hat{x}}}(0)\in R$ only if
  	$\beta_i = 0$ and $\gamma_i = 0$. Since $M\subseteq R$, from \eqref{eq:R_new_module} one has $M = \{0\}$.
  \end{proof}
\begin{rmk}
	\label{rmk:_Module}
		The design of stabilizing controller for each MG can be conducted according to Proposition \ref{prop:Struct_of_P_Q}. In particular, differently from the approach in \cite{7934339}, no optimization problem has to be solved for computing a local controller. Indeed, it is enough to choose control coefficient $k_{1,i}^C$, $k_{2,i}^C$, $k_{3,i}^C$ and $k_{1,i}^V$, $k_{2,i}^V$, $k_{3,i}^V$ from inequality set \eqref{eq:condition_module}. Note that these inequalities are always feasible, implying that a stabilizing controller always exists. Moreover, the inequalities depend only on the parameters $R_{ti}^C$ and $R_{ti}^V$ of the MG $i$. Therefore, the control synthesis is independent of parameters of MGs and power lines which means that controller design can be executed only once for each converter in a plug-and play fashion. From Theorem \ref{thm:overall_stability_module}, local controllers also guarantee stability of the whole MG cluster. When new MGs are plugged in the MG cluster, their controller are designed as described above, the connectivity of the electrical graph $\GG_{el}$ is preserved and have Theorem \ref{thm:overall_stability_module} applied to the whole MG cluster. Instead, when a MG is plugged out, the electrical graph $\GG_{el}$ might be disconnected and split into two connected graphs. Theorem \ref{thm:overall_stability_module} can still be applied to show the stability of each sub-cluster.
\end{rmk}

 \section{Leader-based Distributed Secondary Controller}
  \label{leader-based controller}
  The proposed primary PnP controller can achieve both the voltage and current tracking control in which the reference is given by the local controller. However, to achieve the coordination among MGs, references should be provided by the upper control layer to achieve voltage tracking and current sharing reasonably. Furthermore, to avoid using the centralized controller to send the reference value for each PnP controller, the leader-based distributed consensus algorithm is proposed in the secondary control level including leader-based voltage and current controllers by which not each controller need to know the leader reference.  
 
 In this section, the proposed primary PnP controller is approximated as unitary gains from the perspective of secondary control level. Then the leader-based voltage and current controller is proposed in the secondary control level. Finally, combining with the proposed leader-based voltage and current controller, the asymptotic stability of the proposed controller is proven by Lyapunov stability theory.
   
 \begin{figure}
 	\centering
 	\hspace{-15mm}
 	\includegraphics[scale=0.58]{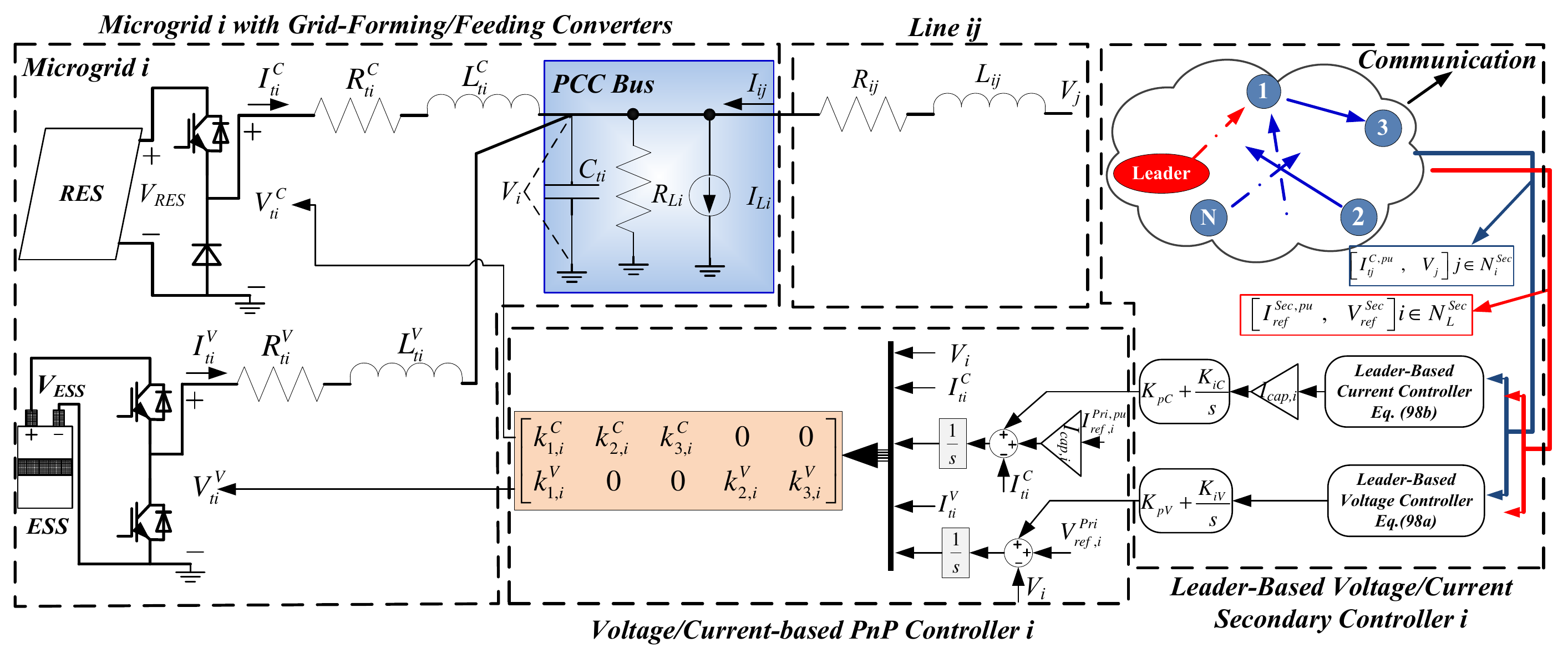}%
 	\caption{Electrical Scheme of MG $i$ with Leader-based Voltage/Current Distributed Secondary Controller.}
 	\label{fig:ctrl_Sec_Module}
 \end{figure}

\subsection{Leader-based Voltage/Current Distributed Secondary Controller}
 The leader-based voltage and current distributed secondary controller is proposed in this subsection to achieve transfer reference information in a distributed way.

Based on \eqref{eq:modelmodule-aug-closed} and \eqref{eq:Fi_module}, the transfer function from voltage reference $z_{ref_{[i]}}^{Pri,V}$ and current reference $z_{ref_{[i]}}^{Pri,C}$ to output voltage $V_{i}$ and output current $I_{ti}^C$ can be written as $\hat{H}_{i}(sI-\hat{F}_{i})\tilde{M}_{i}$ where $\tilde{M}_{i}$ collects the second and third columns of $\hat{M}_{i}$. If setting $s=0 Hz$, the unit matrix is obtained which means the primary PnP controller can be approximated as unit-gain
\begin{subequations}
	\label{eq:unit_appro}
	\begin{empheq}[left=\empheqlbrace]{align}
	\label{eq:V_V_ref}
	V_{i} & =V_{ref,i}^{Pri} \\
	\label{eq:I_I_ref}
	\frac{I_{ti}^C}{I_{cap,i}} & =I_{ti}^{C,pu} = I_{ref,i}^{Pri,pu}
	\end{empheq}	
\end{subequations} 
 The secondary control layer exploit a communication network interconnecting MGs and fulfilling the following Assumption 
\begin{assum}
	\label{assum_communication}
	The communication graph $\GG^{Sec}=(\DD,\EE^{Sec})$ is connected and undirected implying that communication links within MG clusters are bidirectional. Over each communication link $(i,j)\in\EE^{Sec}$, the pairs of variables $(I_{ti}^{C,pu},V_{i})$ and $(I_{tj}^{C,pu},V_{j})$ are transmitted. Furthermore, the graph $\GG^{Sec}$ is endowed with an additional node termed the leader node, carrying the reference values $(I_{ref}^{Sec,pu},V_{ref}^{Sec})$ and connected to at least one node belongs to $\DD$. 
\end{assum}

The proposed leader-based voltage and current distributed secondary controller can be written as
	\begin{subequations}
	\label{eq:Leader_based_controller}
	\begin{empheq}[left=\empheqlbrace]{align}
	\label{eq:V_Leader}
	e_{Vi}&=\sum\limits_{j\in\NN_{i}^{Sec}}a_{ij}\left({V_{i}}-{V_{j}}\right)+g_{i}\left({V_{i}}-{V_{ref}^{Sec}}\right) \\
	\label{eq:C_Leader}
	e_{Ci}&=\sum\limits_{j\in\NN_{i}^{Sec}}a_{ij}\left({I_{i}}^{C,pu}-{I_{j}}^{C,pu}\right)+g_{i}\left({I_{i}}^{C,pu}-{I_{ref}^{Sec,pu}}\right)
	\end{empheq}	
\end{subequations} 
where $\NN_{i}^{Sec}$ is the set of communication neighbors of MG $i$, $a_{ij}=1$ if MGs $i$ and $j$ can communicate with each other through a communication link, $g_{i}=1$ if MG $i$ can receive the reference values about voltage and per-unit current which means $i\in\NN_{L}^{Sec}$ and $\NN_{L}^{Sec}$ is the set for MG who can receive the reference values. 

To be specific, the current reference value $I_{ref}^{Sec,pu}$ is a per-unit value considering the total load requirement and the total system capacity. If the per-unit values of all the output currents from MGs are equals to the reference value, it means that MGs within the cluster share the output current properly according to their own capacities.

In matrix form, \eqref{eq:Leader_based_controller} is given by the equations:
	\begin{subequations}
	\label{eq:Leader_based_controller_M}
	\begin{empheq}[left=\empheqlbrace]{align}
	\label{eq:V_Leader_M}
	e_{V}&=(L+G)(V-V_{ref}^{Sec} \pmb {1_{N}}) \\
	\label{eq:C_Leader_M}
	e_{C}&=(L+G)(I_{t}^{C,pu}-I_{ref}^{Sec,pu}\pmb{ 1_{N}})
	\end{empheq}	
\end{subequations}
where $e_{V}=[\text{ }e_{V1}\text{ } e_{V2} \text{ } \dots \text{ }e_{VN}]^T$, $V=[\text{ }V_{1}\text{ } V_{2} \text{ } \dots \text{ }V_{N}]^T$,$I_{t}^{C,pu}=[\text{ }I_{t1}^{C,pu}\text{ } I_{t2}^{C,pu} \text{ } \dots \text{ }I_{tn}^{C,pu}]^T$, and $G$ is a diagonal matrix with diagonal entries equal to the gains $g_{i}$. Based on Assumption \ref{assum_communication}, $L$ is symmetric Laplacian matrix.

Then, the error $e_{Vi}$ and $e_{Ci}$ are filtered by PI controllers respectively. In order to provide the output $\Delta V_i$ and $\Delta I_{ti}^{C,pu}$ of the secondary controller layer, it can be written as
\begin{subequations}
	\label{eq:Leader_PI_controller}
	\begin{empheq}[left=\empheqlbrace]{align}
	\label{eq:V_PI}
	\Delta V &=-K_{pV}e_{V}-\int{K_{iV}e_{V}} \\
	\label{eq:C_PI}
	\Delta I_{t}^{C,pu} &=-K_{pC}e_{C}-\int{K_{iC}e_{C}}
	\end{empheq}	
\end{subequations}
where $\Delta V=[\text{ }\Delta V_{1}\text{ } \Delta V_{2} \text{ } \dots \text{ }\Delta V_{N}]^T$, $\Delta I_{t}^{C,pu}=[\text{ }\Delta I_{t1}^{C,pu}\text{ } \Delta I_{t2}^{C,pu} \text{ } \dots \text{ }\Delta I_{tn}^{C,pu}]^T$, in addition, $K_{pV}$ and $K_{iV}$ are proportional and integral coefficients of the leader-based voltage controller and $K_{pC}$ and $K_{iC}$ are proportional and integral coefficients of the leader-based current controller. All the coefficients are common to all MGs, thus these are scalar variables.
\begin{rmk}
	Here, for the consensus-based algorithm, in the literature\cite{DBLP:journals/corr/TucciMGF16}, consider only the integral controllers interfacing with the consensus algorithm and the primary control level. In this paper, PI controller is used in order to improve the convergence speed of the secondary controller.  
\end{rmk}
The relationship between the primary PnP controller and the leader-based secondary controller are shown in Fig. \ref{fig:ctrl_Sec_Module}. Exploiting the unit gain approximation of primary loops, one obtains that  \eqref{eq:unit_appro} is replaced by
 \begin{subequations}
 	\label{eq:P_S_Control}
 	\begin{empheq}[left=\empheqlbrace]{align}
 	\label{eq:V_P_S_Control}
 	V & =V_{ref}^{Pri}+\Delta V  \\
 	\label{eq:C_P_S_Control}
 	I_{t}^{C,pu} & = I_{ref}^{Pri,pu}+\Delta I_{t}^{C,pu}
 	\end{empheq}	
 \end{subequations}
where $V_{ref}^{Pri}=[\text{ }V_{ref,1}^{Pri}\text{ } V_{ref,2}^{Pri} \text{ } \dots \text{ }V_{ref,n}^{Pri}]^T$, $I_{ref}^{Pri,pu}=[\text{ }I_{ref,1}^{Pri,pu}\text{ } I_{ref,2}^{Pri,pu} \text{ } \dots \text{ }I_{ref,n}^{Pri,pu}]^T$.

Focusing the time derivative of \eqref{eq:P_S_Control}, we get
 \begin{subequations}
	\label{eq:P_S_Control_D}
	\begin{empheq}[left=\empheqlbrace]{align}
	\label{eq:V_P_S_Control_D}
	\dot{V} & =-K_{iV}[I+K_{pV}(L+G)]^{-1}e_{V}  \\
	\label{eq:C_P_S_Control_D}
	\dot{I}_{t}^{C,pu} & = -K_{iC}[I+K_{pC}(L+G)]^{-1}e_{C}
	\end{empheq}	
\end{subequations}
\subsection{Stability Analysis}
The aim is to show that under the effect of secondary control layer, all PCC voltage converge to the leader value $V_{ref}^{Sec}$ and all the output current converge to the same per-unit value $I_{ref}^{Sec,pu}$. 
\begin{lem}
	\label{lem_po_de}
	Under Assumption \ref{assum_communication}, 
	$L$ is symmetric Laplacian matrix, $G=diag{[g_{1},g_{2},\dots,g_{n}]}$ is diagonal matrix in which $g_{i}\geq 0$ and at least one $g_{i}>0$, 
	then matrix $L+G$ is positive definite.
\end{lem}
\begin{proof}
	\label{proof_po_de}
	As mentioned in $\mbf{Notation}$ at the beginning of this technical report, each vector $x\in\Rset^n$ can always be written in a unique way as
	\begin{equation}
		\label{eq_x_H1_Ho}
		x=\hat{x}+\bar{x} \text{ } \text{with } \hat{x}\in H^1 \text{ and } \bar{x}\in H_{\perp}^1
	\end{equation} 
	Then, one has
	\begin{equation}
	   \label{eq_x}
		x^T(L+G)x=\hat{x}^TL\hat{x}+x^TGx
	\end{equation}
	\eqref{eq_x} is equivalent to the two following cases 
    \begin{subequations}
	\label{eq:hat_ueq0_eq0}
	\begin{empheq}[left=\empheqlbrace]{align}
	\label{eq:hat_ueq0}
	\text{If } \hat{x}\neq0\text{, } \underbrace{\hat{x}^TL\hat{x}}_{>0}+\underbrace{{x}^TG{x}}_{\geq0}&>0\\
	\label{eq:hat_eq0}
	\text{If } \hat{x}=0\text{, } \underbrace{\bar{x}^TL\bar{x}}_{=0}+\underbrace{\bar{x}^TG\bar{x}}_{>0}&>0
	\end{empheq}	
	\end{subequations}
Thus, matrix $L+G$ is positive definite matrix.
\end{proof}
\begin{cor}
	\label{cor}
	Under Lemma \ref{lem_po_de}, matrix $(L+G)^{-1}$ is positive definite and matrix $[I+\alpha(L+G)]^{-1}$ where scalar $\alpha>0$ is also positive definite.
\end{cor}
We recall that if $\alpha$ is a scalar, $A$ is positive definite matrix and $I$ is unit matrix which is also positive definite matrix, from Woodbury matrix identity theory \cite{doi:10.1137/1.9780898718027} , one has
	\begin{equation}
		\label{inverse_I_L}
		[I+\alpha A]^{-1} = \alpha^{-1}A^{-1}-\alpha^{-1}A^{-1}(\alpha^{-1}A^{-1}+I)^{-1}\alpha^{-1}A^{-1}
	\end{equation}	
\begin{lem}
	\label{AB}
	\cite{horn2012matrix} Let $A$, $B$ $\in\Rset^{n\times n}$ be positive definite matrices. If $AB=BA$ is satisfied, then $AB$ is positive definite.
\end{lem}
\begin{lem}
	\label{lem_AB_positive_proof}
	Under Lemma \ref{AB} and Corollary \ref{cor} and eq. \eqref{inverse_I_L}, it is known that scalar $K_{pV}>0$ and $(L+G)$ is positive definite matrix, matrix $(L+G)[I+K_{pV}(L+G)]^{-1}$ is positive definite.
\end{lem}
\begin{proof}
From Corollary \ref{cor}, matrix $[I+K_{pV}(L+G)]^{-1}$ is positive. Then, from \eqref{inverse_I_L}, one has
	\begin{equation}
	\label{L_G_inverse}
	\begin{aligned}	
		(L+G)[I+K_{pV}(L+G)]^{-1}  &=K_{pV}^{-1}I-\left(K_{pV}^{-1}\right)^2\left[K_{pV}^{-1}(L+G)^{-1}+I\right]^{-1}\left(L+G\right)^{-1}\\
		                       &= K_{pV}^{-1}I-\left(K_{pV}^{-1}\right)^2\left[K_{pV}^{-1}+(L+G)\right]^{-1}\\
		\end{aligned}
	\end{equation}
	Then
	\begin{equation}
	\label{G_L_inverse}
	\begin{aligned}	
	[I+K_{pV}(L+G)]^{-1}(L+G)  &=K_{pV}^{-1}I-\left(K_{pV}^{-1}\right)^2\left(L+G\right)^{-1}\left[K_{pV}^{-1}(L+G)^{-1}+I\right]^{-1}\\
	&= K_{pV}^{-1}I-\left(K_{pV}^{-1}\right)^2\left[K_{pV}^{-1}+(L+G)\right]^{-1}\\
	\end{aligned}
	\end{equation}
	Comparing \eqref{L_G_inverse} with \eqref{G_L_inverse}, we have
	\begin{equation}
	\label{AB=BA}
		(L+G)[I+K_{pV}(L+G)]^{-1}=[I+K_{pV}(L+G)]^{-1}(L+G)
	\end{equation}
	To conclude, from Lemma \ref{AB}, since both matrices $(L+G)$ and $[I+K_{pV}(L+G)]^{-1}$ are positive definite, combined with \eqref{AB=BA}, the matrix $(L+G)[I+K_{pV}(L+G)]^{-1}$ is positive definite.
\end{proof}
Note that the consensus schemes \eqref{eq:V_Leader}-\eqref{eq:V_PI} and \eqref{eq:C_Leader}-\eqref{eq:C_PI} have the same structure. Then, in the following, we show convergence to the leader reference value only for voltages.

We consider the following candidate as Lyapunov function
\begin{equation}
	\label{Lyap_V}
	Z=\frac{1}{2}e_{V}^TP^{Sec}e_{V}, \text{   where } P^{Sec}>0 
\end{equation}

The time derivative of \eqref{Lyap_V} is 
 \begin{equation}
 \label{Deri_Lyap_V}
 \begin{aligned}
 \dot{Z} & =e_{V}^TP^{Sec}(L+G)\dot{V}\\
         & =-K_{iV}e_{V}^TP^{Sec}(L+G)[I+K_{pV}(L+G)]^{-1}e_{V}\\
         & =\frac{-K_{iV}}{2}e_{V}^T[P^{Sec}O+O^TP^{Sec}]e_{V}
 \end{aligned}
 \end{equation}
 where $O=(L+G)[I+K_{pV}(L+G)]^{-1}$. 
 
 Based on Lemma \ref{lem_AB_positive_proof}, matrix $O$ is positive definite. Based on Lyapunov theory \cite{qu2009cooperative}, there exists positive definite matrix $P^{Sec}$ which can make $P^{Sec}O+O^TP^{Sec}$ is positive definite. Therefore 
  \begin{equation}
 \label{Deri_Lyap_V_ieq}
 \dot{Z} =\frac{-K_{iV}}{2}e_{V}^T[P^{Sec}O+O^TP^{Sec}]e_{V} < \frac{-K_{iV}}{2} \sigma_{min}(P^{Sec}O+O^TP^{Sec})||e_{V}||^2 < 0
 \end{equation}
 where $\sigma_{min}(P^{Sec}O+O^TP^{Sec})$ denotes the minimal eigenvalues of the symmetric matrix $P^{Sec}O+O^TP^{Sec}$.
From \eqref{Deri_Lyap_V_ieq}, one has that the tracking error $e_{V}$ goes to zero, and that all PCC voltages converge to the reference value provided by the leader. The convergence of output currents to the reference value is the same as above. 
 \section{Hardware-in-Loop Tests}
  \label{sec:simulation_results}
  In order to verify the effectiveness of proposed primary PnP controller combined with leader-based voltage/current distributed controllers for MG clusters, real-time HiL tests are carried out based on dSPACE 1006. The real-time simulation model comprises four MGs with meshed electrical topology shown in Fig. \ref{fig:setup}. The capacity ratio for four MGs rated capacity is $1: 2: 3: 4$ from MGs $1-4$. Communication network has the same topology of the electrical network. And MG $1$ is the only one receiving the leader information. The nominal voltage for the dc MG is 48V. In addition, in Appendix \ref{sec:AppElectrPar}, the electrical setup information is shown in TABLE \ref{tbl:electrical_setup}, the transmission lines parameters are shown in TABLE \ref{tbl:line_parameters} and the control coefficients are shown in TABLE \ref{tbl:control coefficients}. 
  
 \begin{figure}[!htb]
 	\centering
 	\includegraphics[scale=0.48]{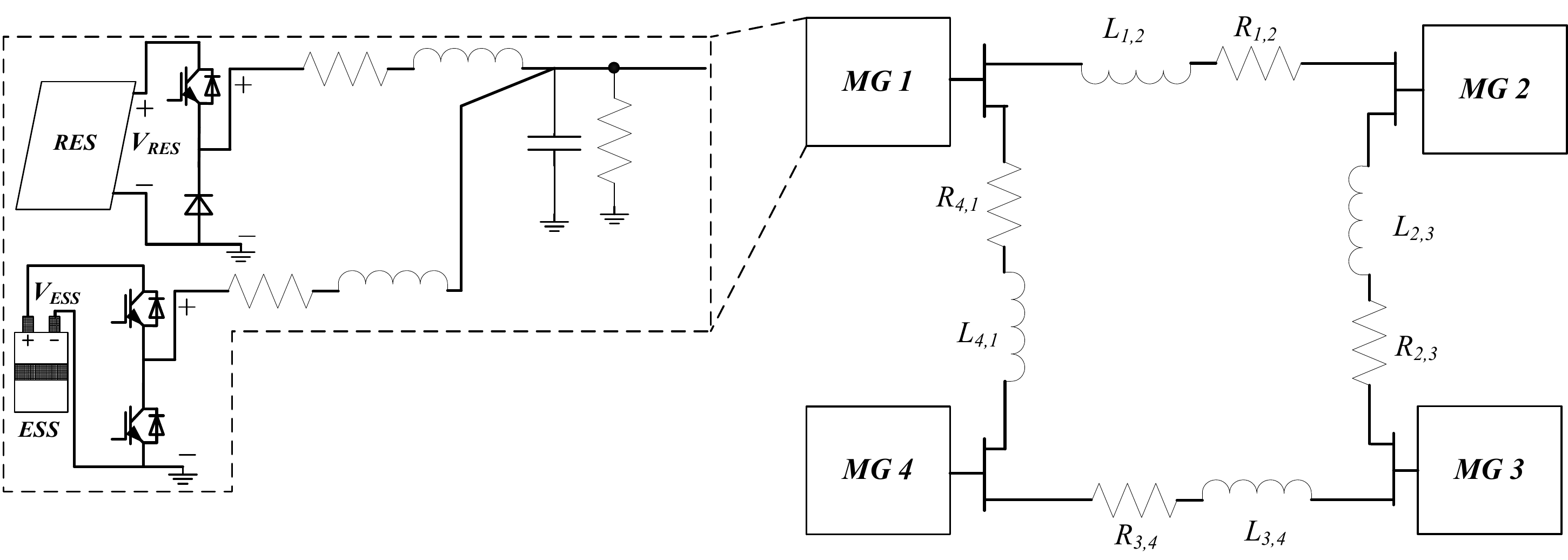}
 	\caption{System Configuration of Hardware-in-Loop Test.}
 	\label{fig:setup}
 \end{figure}
\subsection{Case 1: PnP Test considering Primary Control Level}

In this subsection, the effectiveness of the proposed primary PnP controller is verified. Each MG is started separately. At the beginning, we set different voltage and current references for different MGs. At $t = T1$, MGs $1-3$ are connected together without changing the control coefficients. As shown in Fig. \ref{fig:PnP_Test_V}, after the connection of MGs $1-3$, only small disturbances exist in the voltage waveform. Moreover, there is no major disturbance affecting the output currents as shown in Fig. \ref{fig:PnP_Test_C}. Then at $t = T2$,  MG $4$ is connected to the system. Similarly, as shown in Fig. \ref{fig:PnP_Test}, after small disturbance, both the output voltage and current track the respective reference values. 
    \begin{figure}[!htb]
                      \centering
                      \begin{subfigure}[htb]{0.49\textwidth}
                        \centering
                           \includegraphics[width=1\textwidth]{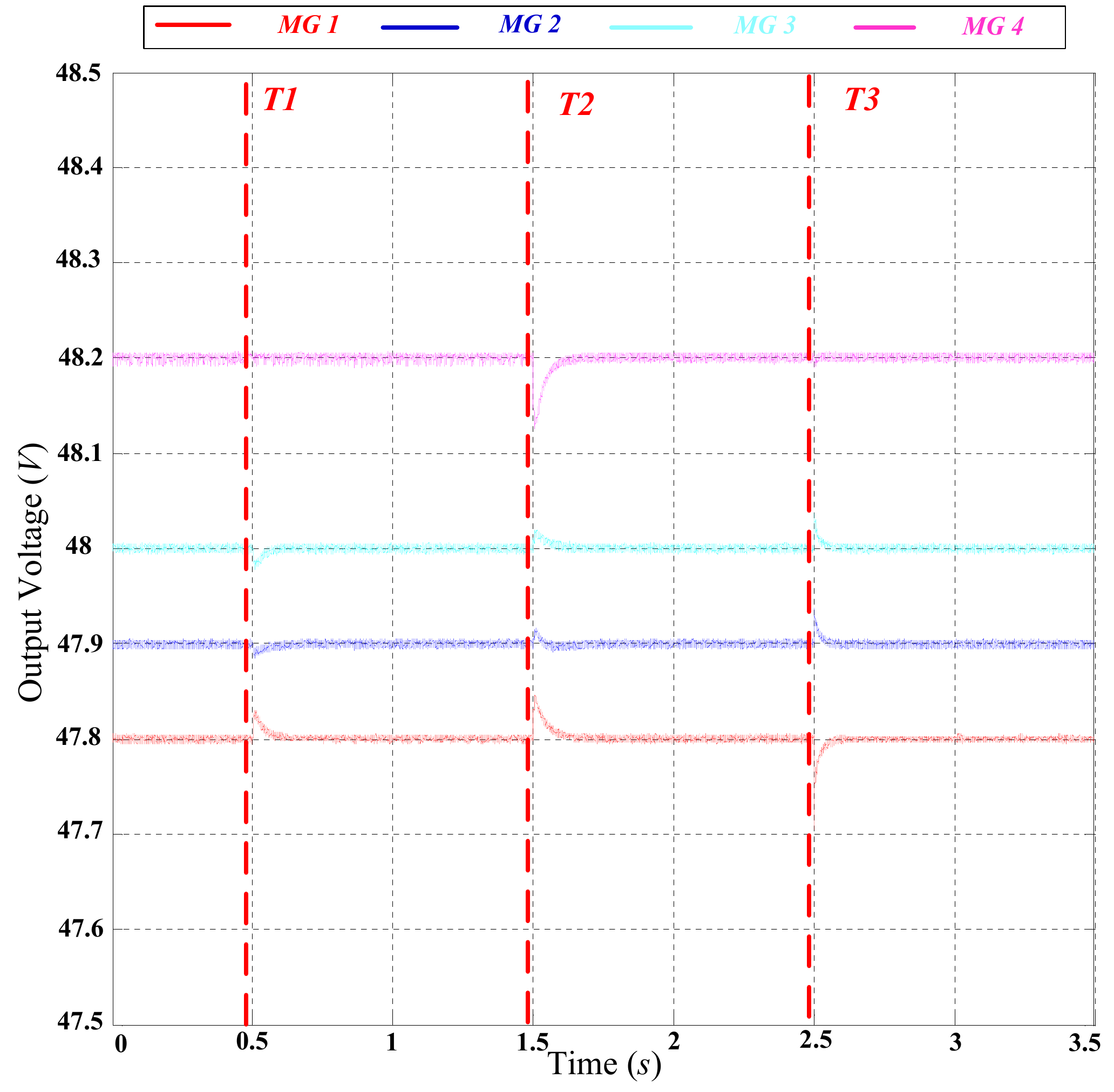}
                        \caption{Voltage Performance for PnP Test.}
                        \label{fig:PnP_Test_V}
                      \end{subfigure}
                      \begin{subfigure}[htb]{0.48\textwidth} 
                        \centering
                       \includegraphics[width=1\textwidth]{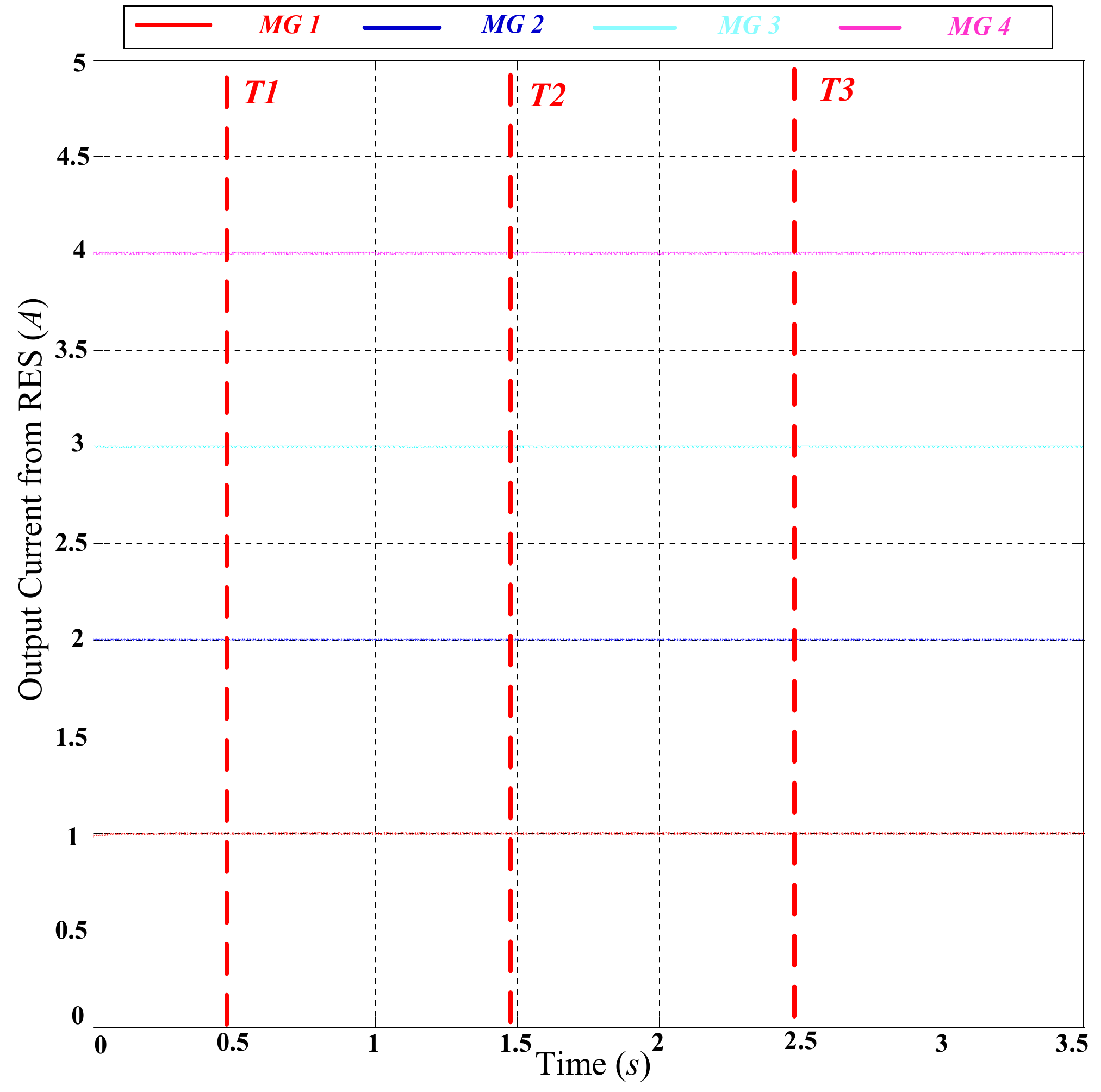}
                        \caption{Current Performance for PnP Test.}
                        \label{fig:PnP_Test_C}
                      \end{subfigure}
                     \caption{The Plug-in/-out Performance of primary PnP controllers.}
                      \label{fig:PnP_Test}                 
                    \end{figure}

Fig. \ref{fig:Tracking_Text} illustrates the current tracking performance by changing the current references for different modules. At $t = T1$, four MGs are connected together simultaneously. At $t = T2$, the current reference for MG $1$ is changed from $1A$ to $2.5A$. At $t = T3$, the current reference for MG $2$ is changed from $2A$ to $3.5A$. At $t = T4$, the current reference for MG $3$ is changed from $3A$ to $1.5A$. At $t = T5$, the current reference for MG $4$ is changed from $4A$ to $5.5A$. As shown in Fig. \ref{Tracking_Text_C}, whether the current references are increased or decreased, the output currents can track the changed reference. In addition, as shown in Fig. \ref{Tracking_Text_V}, when the current references are changed, the output voltages are only affected by little oscillations approximately $0.05V$. 
                    \begin{figure}[!htb]
                	\centering
                	\begin{subfigure}[htb]{0.48\textwidth}
                		\centering
                		\includegraphics[width=1\textwidth]{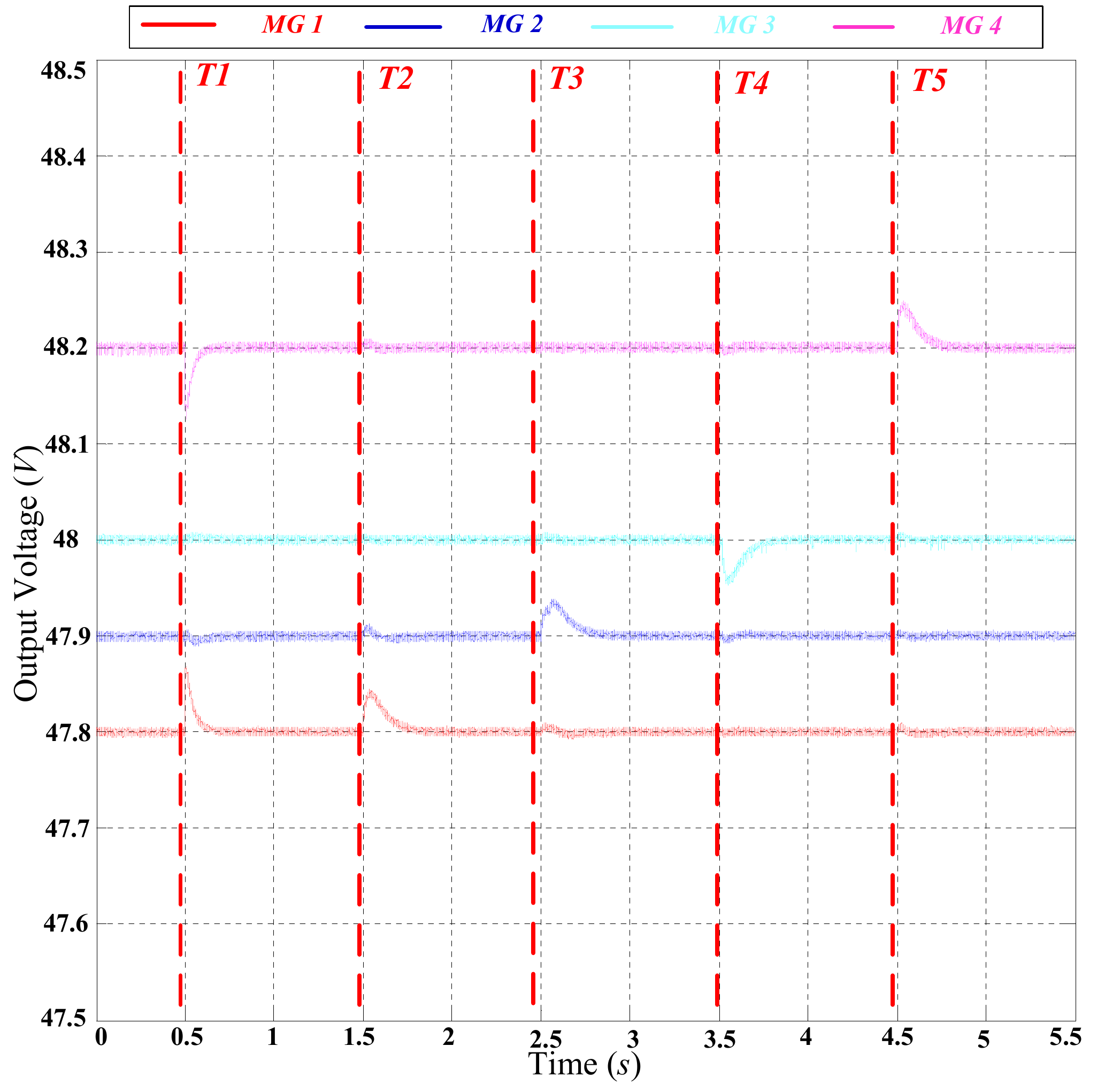}
                		\caption{Voltage Tracking Performance.}
                		\label{Tracking_Text_V}
                	\end{subfigure}
                	\begin{subfigure}[htb]{0.48\textwidth} 
                		\centering
                		\includegraphics[width=1\textwidth]{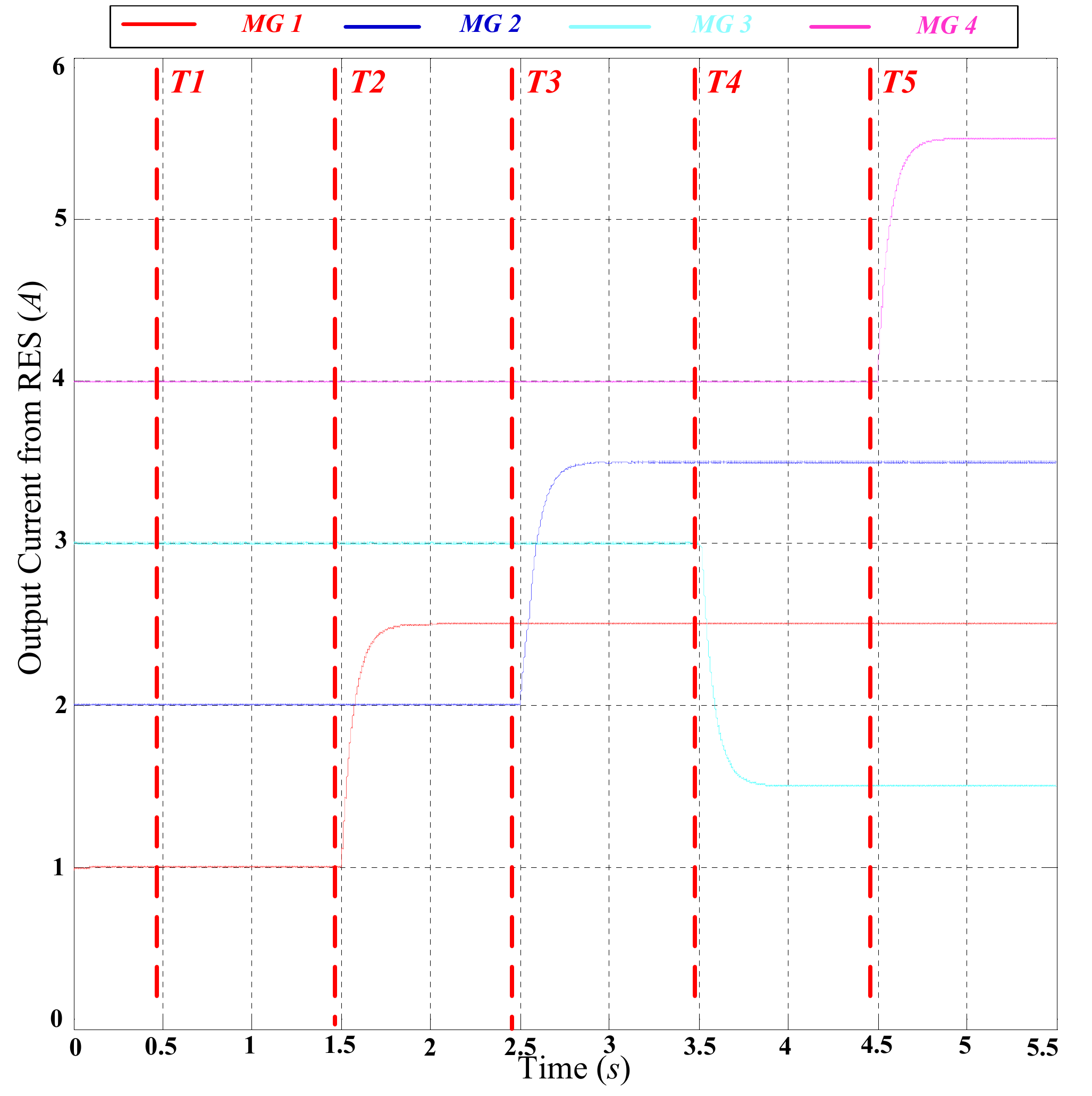}
                		\caption{Current Tracking Performance.}
                		\label{Tracking_Text_C}
                	\end{subfigure}
                	\caption{Voltage and Current Tracking Performance of PnP
                		decentralized controllers}
                	\label{fig:Tracking_Text}                 
                \end{figure}
\subsection{Case 2: Leader-Based Voltage/Current Distributed Secondary Controller Test}
In this subsection, the effect of proposed leader-based voltage/current distributed secondary controller is verified. At $t = T1$, four MGs are connected together simultaneously. At $t = T2$, the proposed leader-based voltage controller is enabled and the leader value is set as $48V$. It is illustrated in Fig. \ref{fig:Case_2_V} that after $t = T2$, the output voltages converge to the leader reference under $0.3s$. Then, at $t = T3$, the proposed leader-based current controller is enabled and leader value is set as $0.3p.u.$. As shown in Fig. \ref{fig:Case_2_C}, the proposed current controller can achieve current sharing in proportional and Fig. \ref{fig:Case_2_C_pu} illustrates that the per-unit current values can converge to the leader value within $1s$. In addition, Fig. \ref{fig:Case_2_V_3_4} illustrates that only $0.04V$ oscillations exist in the output voltages when enabling the leader-based current controller. Furthermore, when the reference for leader-based voltage controller is changed from $48V$ to $49V$ at $t = T4$, the output voltage still track the leader reference, as shown in Fig. \ref{fig:Case_2_V}. Similarly, when the reference for leader-based current controller is changed from $0.3p.u.$ to $0.4p.u.$ at $t = T5$, the output current can also track the new value as shown in Fig. \ref{fig:Case_2_C_pu}. Fig. \ref{fig:Case_2_V_7_8} illustrates that when the reference for leader-based current is changed, the output voltages are not affected.
\begin{figure}[!htb]
                \centering
                      \begin{subfigure}[htb]{0.60\textwidth} 
                        \centering    
                        \includegraphics[width=1\textwidth]{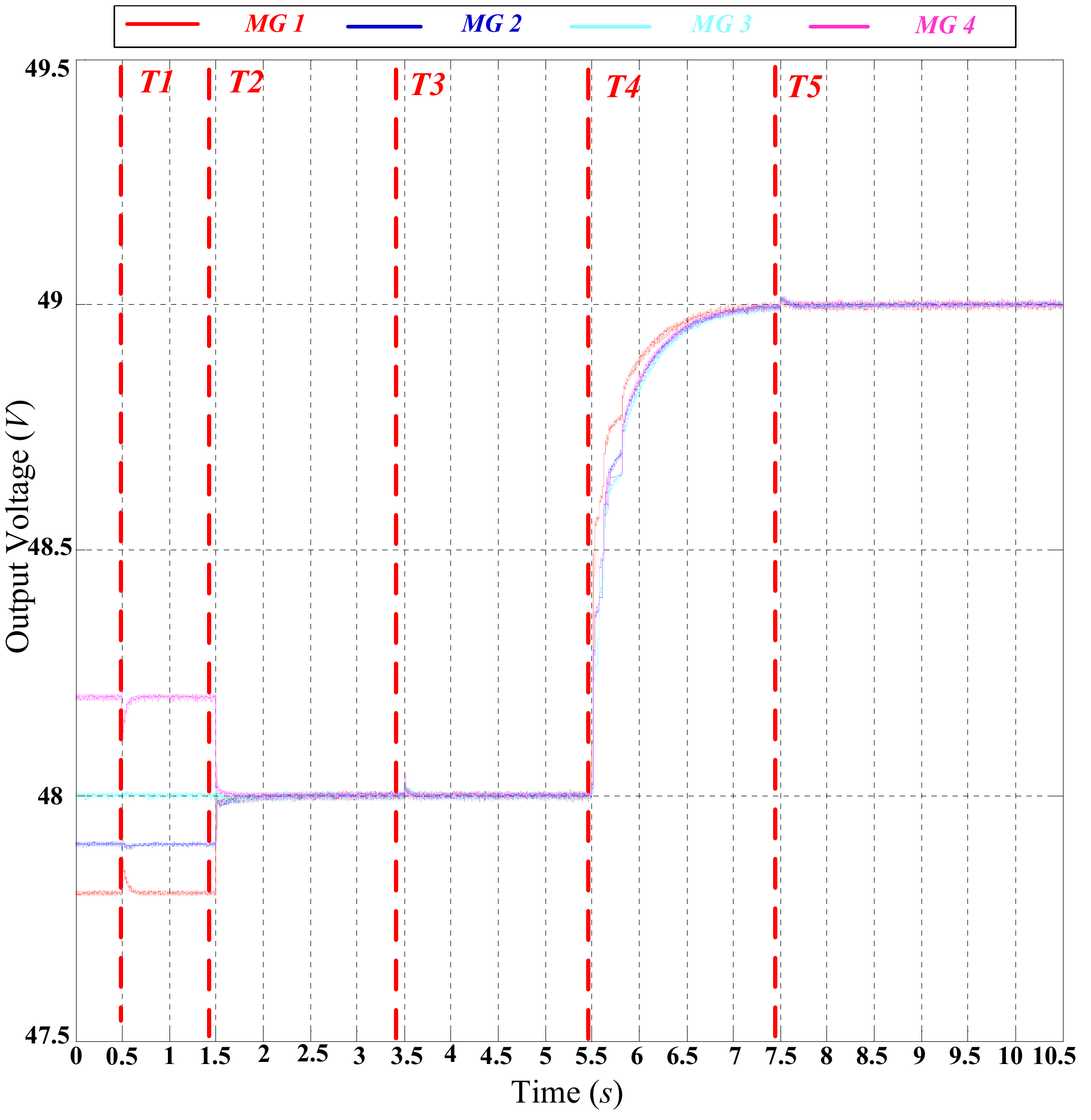}
                        \caption{Voltage Performance.}
                        \label{fig:Case_2_V}
                      \end{subfigure}
                      \begin{subfigure}[htb]{0.49\textwidth}
                       \centering
                       \includegraphics[width=1\textwidth]{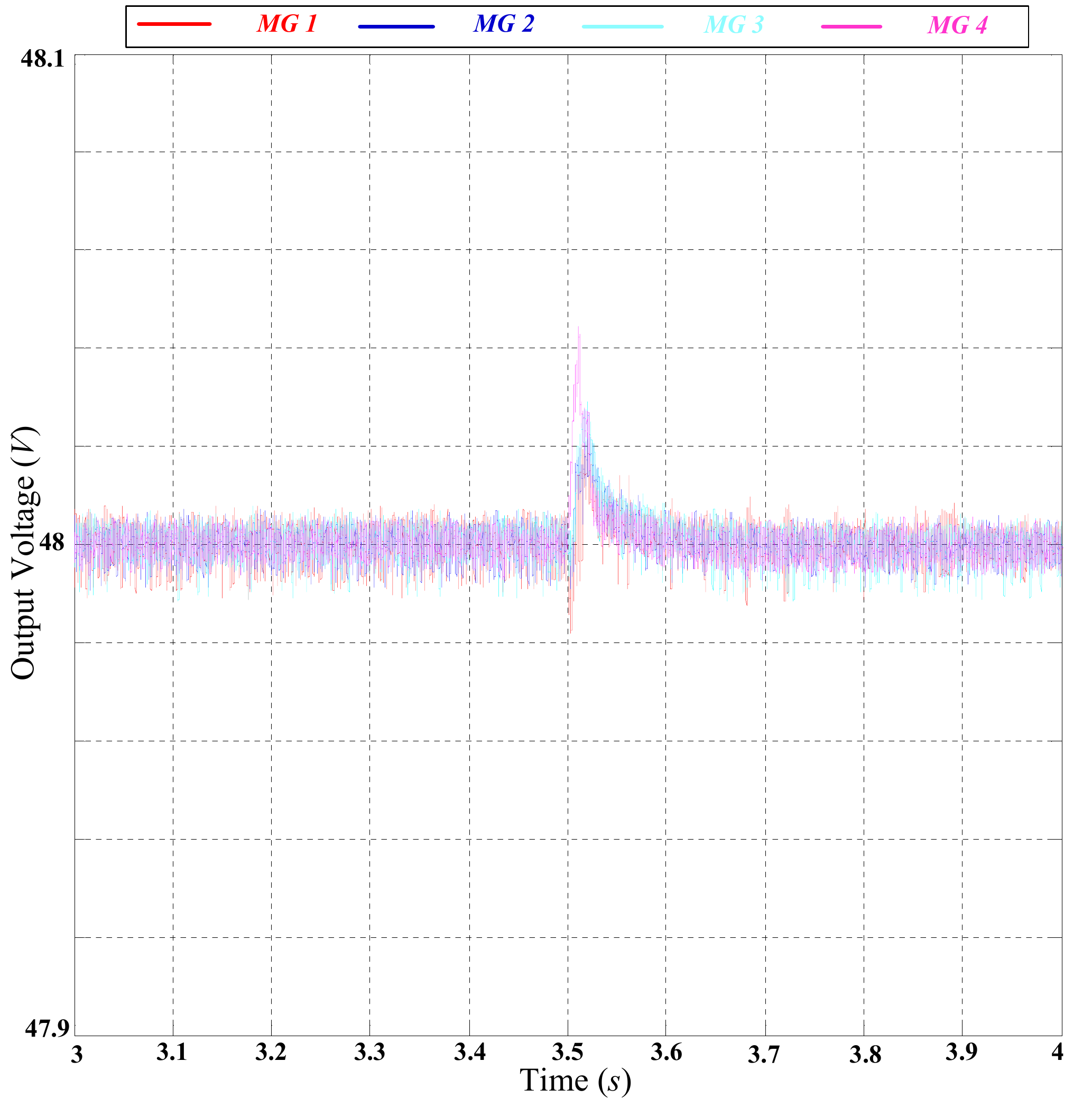}
                        \caption{Zoomed in Voltage Performance Between $3s$ and $4s$.}
                        \label{fig:Case_2_V_3_4}
                      \end{subfigure}
                      \begin{subfigure}[htb]{0.49\textwidth}\hspace{2mm} 
                       \centering
                       \includegraphics[width=1\textwidth]{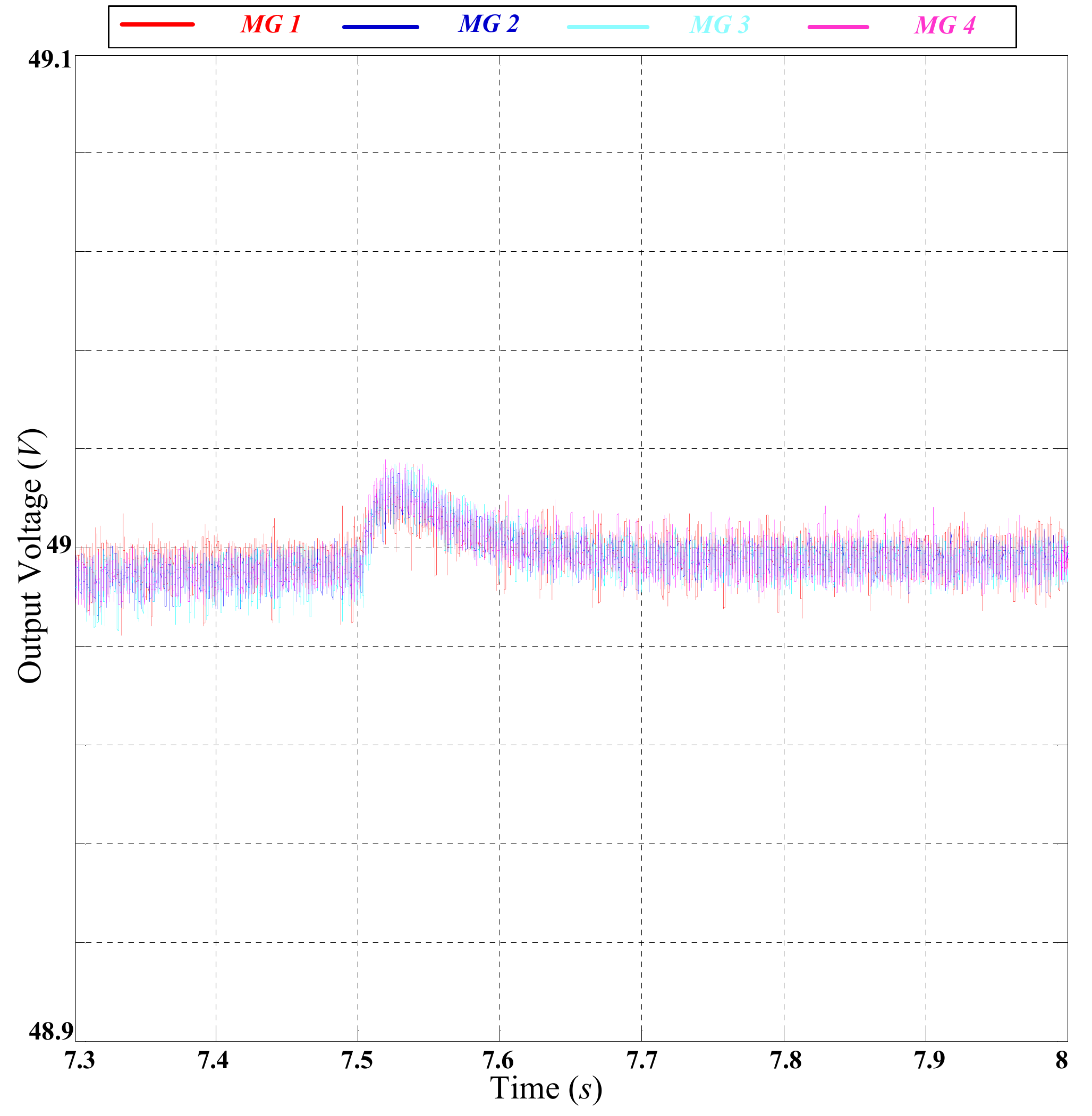}
                        \caption{Zoomed in Voltage Performance Between $7.3s$ and $8s$.}
                        \label{fig:Case_2_V_7_8}
                      \end{subfigure}
                     \caption{Voltage Performance for Leader-Based Voltage Secondary Controllers.}
                      \label{fig:Case_2_V_Full}                 
                    \end{figure}
                \begin{figure}[!htb]
                	\centering
                	\begin{subfigure}[htb]{0.48\textwidth} 
                		\centering    
                		\includegraphics[width=1\textwidth]{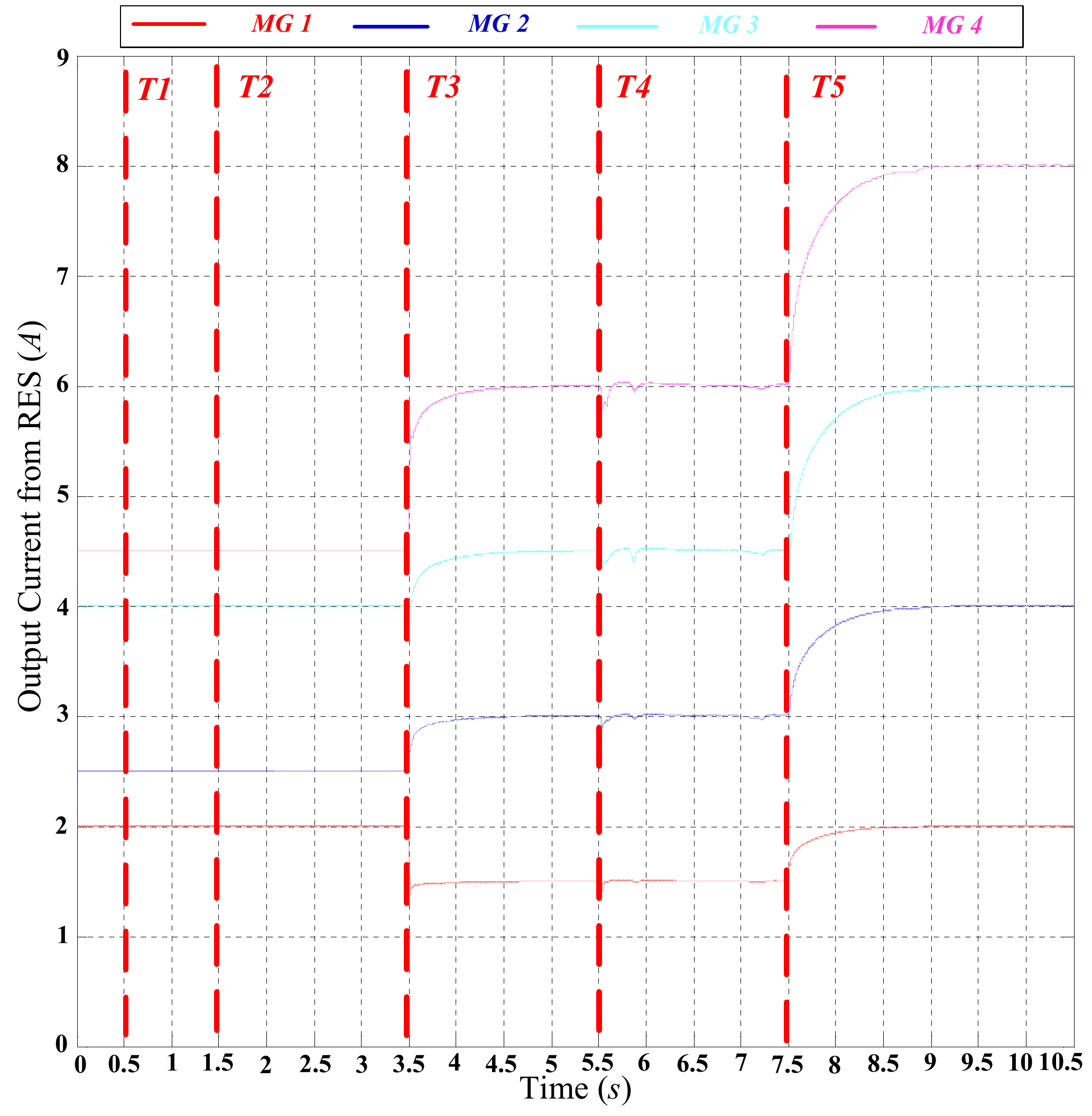}
                		\caption{Current Performance.}
                		\label{fig:Case_2_C}
                	\end{subfigure}
                	\begin{subfigure}[htb]{0.48\textwidth}
                		\centering
                		\includegraphics[width=1\textwidth]{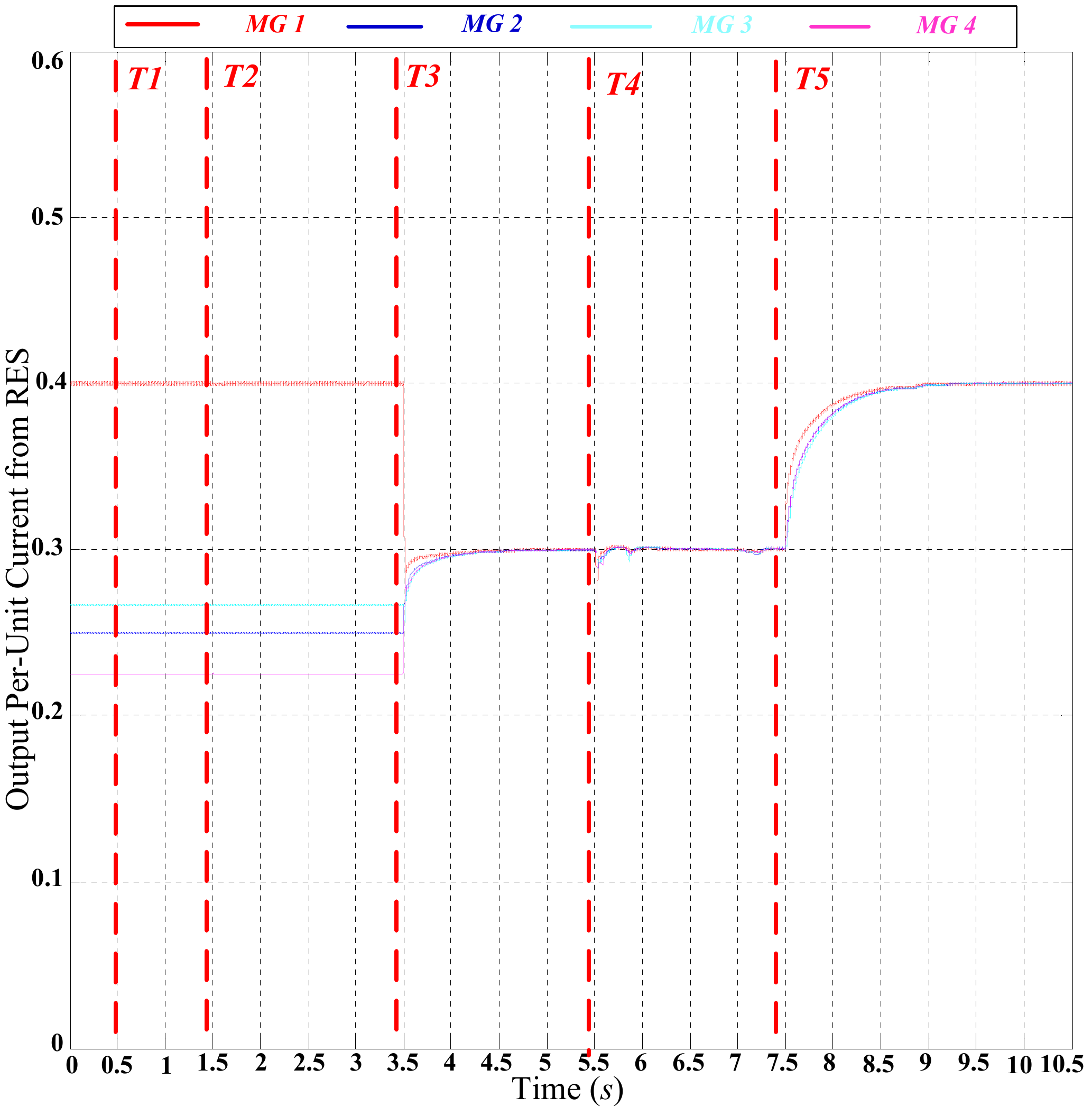}
                		\caption{Per-Unit Current Performance.}
                		\label{fig:Case_2_C_pu}
                	\end{subfigure}
                	\caption{Current Performance for Leader-Based Current Secondary Controllers.}
                	\label{fig:Case_2_C_Full}                 
                \end{figure}
\subsection{Case 3: PnP Test Considering Both Primary and Secondary Control Level}
In this subsection, the PnP effect of both primary and secondary controllers is tested. At $t = T1$, four MFs are connected together simultaneously. At $t = T2$ and $T3$, the proposed leader-based voltage controller and leader-based current controller are enabled, respectively. At $t = T4$, MG $2$ is plugged out of the MG cluster which means the communication links and electrical lines are all disconnected with the MG cluster. As shown in Figs. \ref{fig:Case_3_V} and \ref{fig:Case_3_C_full}, the other three MGs still operate in a stable way and then keep tracking the leader reference from the secondary control level. Meanwhile, MG $2$ can still use its own primary controller following the reference from the primary control level which are $47.8V$ for voltage and $0.25p.u.$ for current. At $t = T5$, MG $2$ is plugged into the cluster and the communication links of MG $2$ are also enabled. As shown in Fig. \ref{fig:Case_3_V} and \ref{fig:Case_3_C_full} after $t = T5$, both the output voltage and current of MG $2$ start to track the reference value of the leader node. Overall, the simulation results shows that even in presence of plug-in/out events, the MG cluster can behave in a stable way. And both output voltage and current tracking performance can be guaranteed. Furthermore, during the whole test, the control coefficients for each MG are not changed. 
\begin{figure}[!htb]
	\centering
	\includegraphics[scale=0.48]{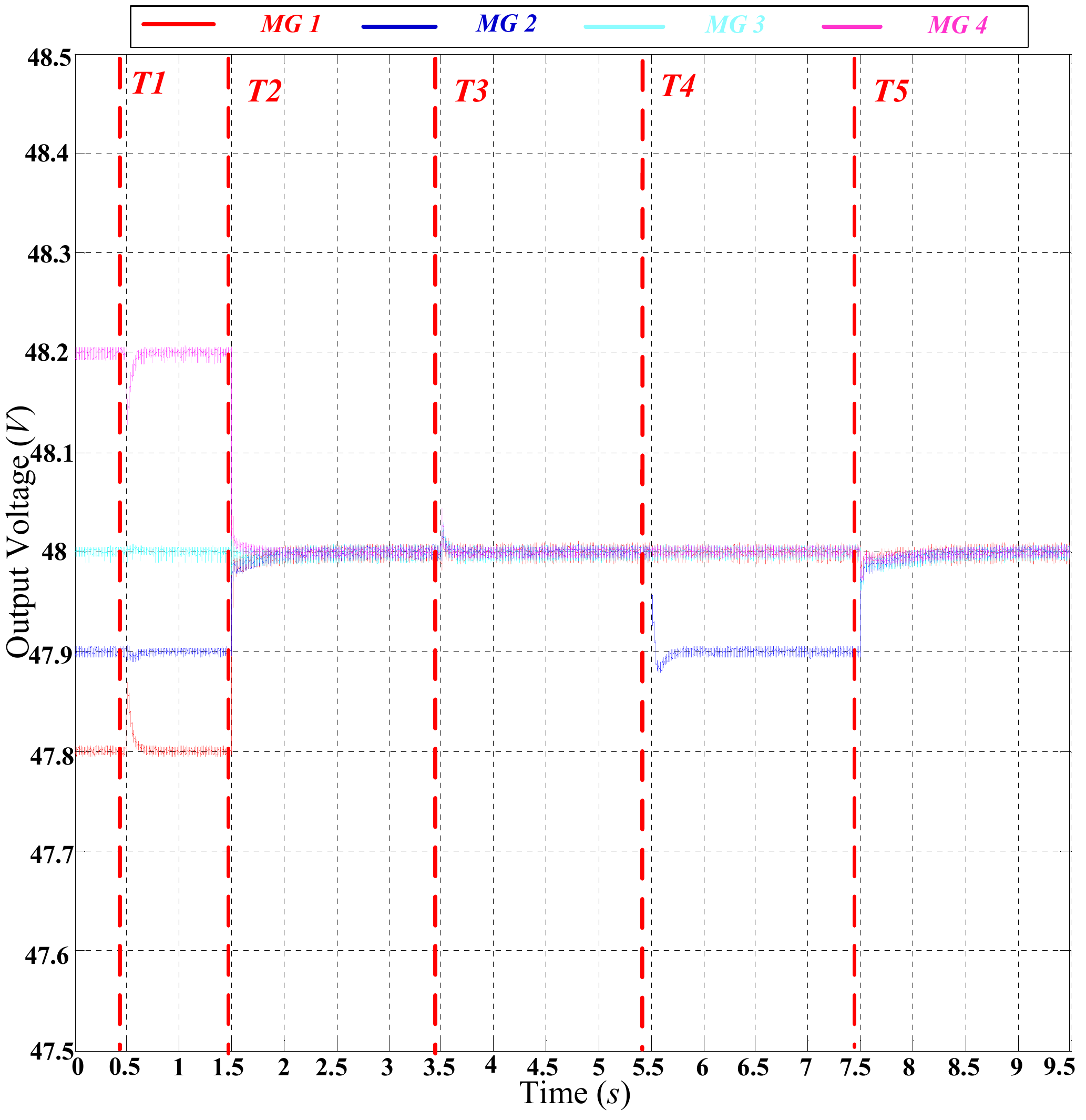}
	\caption{Voltage Performance for PnP Test considering both the Primary and Secondary Level.}
	\label{fig:Case_3_V}
\end{figure}
 \begin{figure}[!htb]
	\begin{subfigure}[!htb]{0.48\textwidth}
		\vspace{-1.1mm}
		\centering
		\includegraphics[width=1\textwidth]{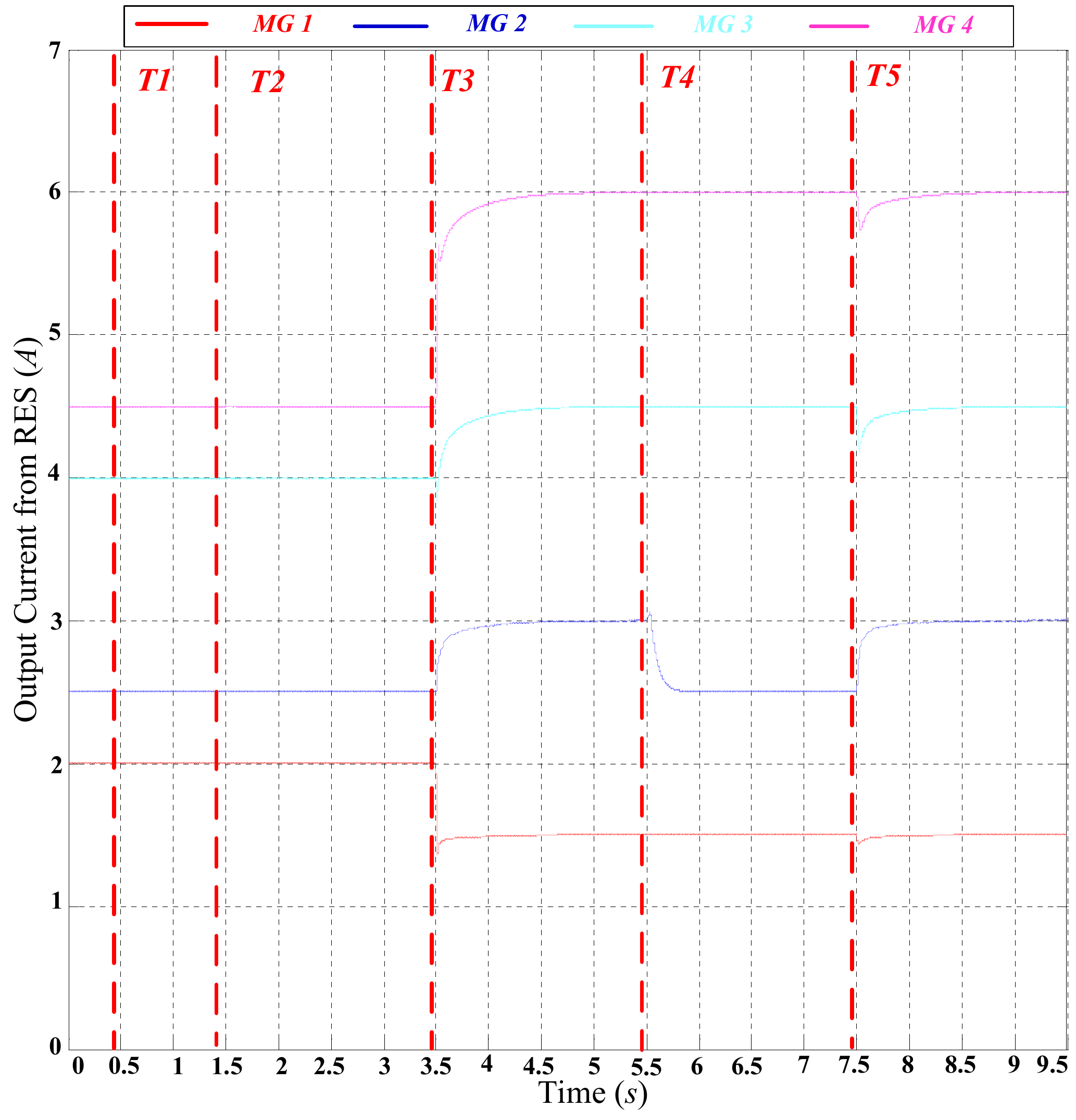}
		\caption{Current Performance.}
		\label{fig:Case_3_C}
	\end{subfigure}
	\begin{subfigure}[!htb]{0.48\textwidth}
		\centering
		\includegraphics[width=1\textwidth]{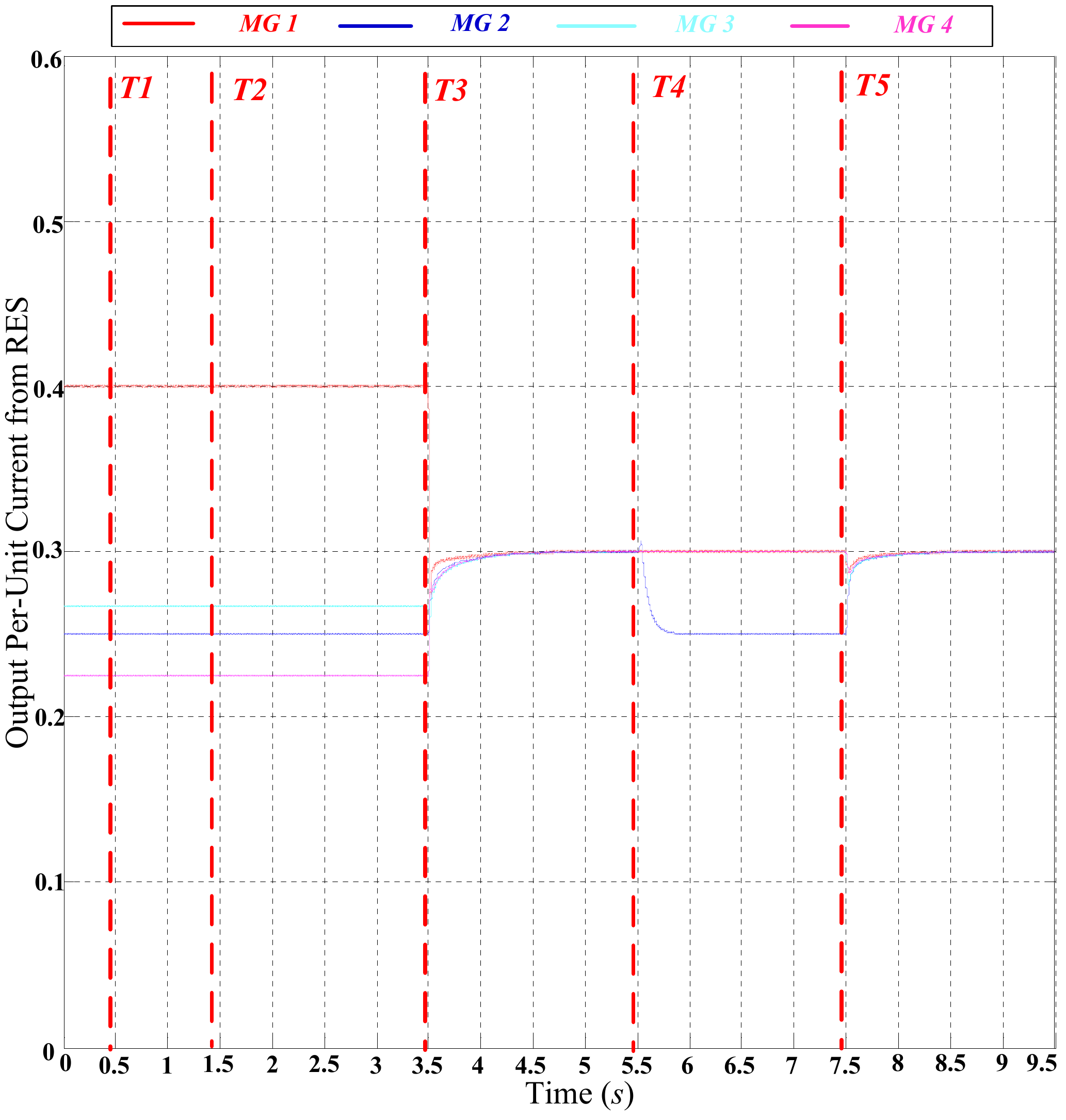}
		\caption{Per-Unit Current Performance.}
		\label{fig:Case_3_C_pu}
	\end{subfigure}
	\caption{Current Performance for PnP Test considering both the Primary and Secondary Level.}
	\label{fig:Case_3_C_full}                 
\end{figure}

\section{Conclusions}
\label{sec:conclusions}

In this paper, a hierarchical PnP Voltage/Current Controller for DC microgrid clusters with grid-forming/feeding modules is proposed including primary control level and secondary control level. In the primary control level, a novel PnP controller is proposed for a MG with grid-forming/feeding converters to achieve both the output voltages and currents tracking with the local control reference. A set only related to the local system information for control coefficients is found by which the controller can always be stable avoiding solving LMI problem. Meanwhile, the MG can achieve plug-in/out operation without changing the control coefficients to guarantee global stability of the MG cluster. In the secondary control level, the leader-based voltage/current distributed secondary controller is proposed to achieve both the voltage and current tracking with the information from the higher control level. Each MG only requires its own information and the information of its neighbours on the communication network graph. By approximating the primary PnP controller with unitary gains, the model of the whole system is established whose stability is proved by Lyapunov stability theory. Finally, the theoretical results are proven by the hardware in loop tests.

\clearpage
\appendix
\input{Matrices}
\input{ElectrPar}

\bibliographystyle{IEEEtran}
\bibliography{IEEEabrv,PnP_improved}

\end{document}

%% file: macro.tex
\newcommand{\Rset}{\mathbb{R}}

\newcommand{\hd}{{\hat{d}}}

\newcommand{\hx}{{\hat{x}}}

\newcommand{\CC}{{\mathcal{C}}}
\newcommand{\DD}{{\mathcal{D}}}
\newcommand{\EE}{{\mathcal{E}}}
\newcommand{\FF}{{\mathcal{F}}}
\newcommand{\GG}{{\mathcal{G}}}

\newcommand{\LL}{{\mathcal{L}}}
\newcommand{\MM}{{\mathcal{M}}}
\newcommand{\NN}{{\mathcal{N}}}

\newcommand{\PP}{{\mathcal{P}}}
\newcommand{\QQ}{{\mathcal{Q}}}

\newcommand{\UU}{{\mathcal{U}}}
\newcommand{\VV}{{\mathcal{V}}}

\newcommand{\diag}{{\mbox{diag}}}                  

\newcommand{\mbf}[1]{\mathbf{#1}}                  

\newcommand{\subss}[2]{{#1}_{[#2]}}

\newcommand{\dx}{{\dot x}}

%% file: Matrices.tex
\section{Matrices appearing in microgrid models}
     \label{sec:AppMatrices}

     \subsection{Matrices in the model of CDGU}
     \label{sec:AppMatrices_CDGU}
    This appendix collects all matrices appearing in Section \ref{sec:Model_C}.

        \paragraph{Overall model of a MG composed by $N$ CDGUs}
             \label{sec:AppOverallsys}

             \begin{equation}
               \label{TheSystem}
                \renewcommand\arraystretch{1.8}
               \begin{aligned}
                 \begin{bmatrix}
                   \subss{\dx}{1}^C \\
                   \subss{\dx}{2}^C \\
                   \subss{\dx}{3}^C \\
                   \vdots \\
                   \subss{\dx}{N}^C
                 \end{bmatrix} 
                 &= 
                 \underbrace{\left[\begin{array}{ccccc}
                       A_{11}^C+A_{load,1}^C & A_{12}^C & A_{13}^C & \dots  & A_{1N}^C \\
                       A_{21}^C & A_{22}+A_{load,2}^C & A_{23} & \dots  & A_{2N}^C \\
                       A_{31}^C & A_{32}^C & A_{33}+A_{load,3}^C & \dots  & A_{3N}^C \\
                       \vdots & \vdots & \vdots & \ddots & \vdots\\
                       A_{N1}^C & A_{N2}^C & A_{N3}^C & \dots  & A_{NN}+A_{load,N}^C
                     \end{array}
                   \right]}_{\mbf{A^C}}
                 \begin{bmatrix}
                   \subss{x}{1}^C \\
                   \subss{x}{2}^C \\
                   \subss{x}{3}^C \\
                   \vdots \\
                   \subss{x}{N}^C
                 \end{bmatrix} 
                 +\\
                 &+\,
                 \underbrace{\begin{bmatrix}
                     B_{1}^C & 0 & \dots & 0\\
                     0 & B_{2}^C & \ddots & \vdots\\
                     \vdots & \ddots & \ddots & 0\\
                     0& \dots & 0  & B_{N}^C
                   \end{bmatrix}}_{\mbf{B^C}}
                 \begin{bmatrix}
                   \subss{u}{1}^C\\
                   \subss{u}{2}^C\\
                   \vdots\\
                   \subss{u}{N}^C
                 \end{bmatrix}
                 + \underbrace{\begin{bmatrix}
                     M_{1}^C & 0 & \dots & 0\\
                     0 & M_{2}^C & \ddots & \vdots\\
                     \vdots & \ddots & \ddots & 0\\
                     0& \dots & 0  & M_{N}^C
                   \end{bmatrix}}_{\mbf{M^C}}
                 \begin{bmatrix}
                   \subss{d}{1}^C\\
                   \subss{d}{2}^C\\
                   \vdots\\
                   \subss{d}{N}^C
                 \end{bmatrix}\\    
                 \begin{bmatrix}
                   \subss{z}{1}^C\\
                   \subss{z}{2}^C\\
                   \subss{z}{3}^C\\
                   \vdots\\
                   \subss{z}{N}^C
                 \end{bmatrix}
                 &=
                 \underbrace{\begin{bmatrix}
                     H_{1}^C & 0 & 0 & \dots & 0 \\
                     0 & H_{2}^C & 0 & \ddots & \vdots \\
                     0 & 0 & H_{3}^C & \ddots & 0 \\
                     \vdots & \ddots & \ddots &\ddots & 0\\
                     0& \dots & 0 & 0  & H_{N}^C
                   \end{bmatrix}}_{\mbf{H^C}}\begin{bmatrix}
                   \subss{x}{1}^C\\
                   \subss{x}{2}^C\\
                   \subss{x}{3}^C\\
                   \vdots\\
                   \subss{x}{N}^C
                 \end{bmatrix}.
               \end{aligned} 
             \end{equation}
             
             \subsection{Matrices in the model of MG Clusters}
              \label{sec:AppMatrices_DGM}
             This appendix provides all matrices appearing in Section \ref{PV/Battery_Model}.
             
              \paragraph{Overall model of MG clusters composed by $N$ MGs}
             \label{sec:AppOverallsys_DGM}
             
             \begin{equation}
             \label{TheSystem_DGM}
             \renewcommand\arraystretch{1.8}
             \begin{aligned}
             \begin{bmatrix}
             \subss{\dx}{1} \\
             \subss{\dx}{2} \\
             \subss{\dx}{3} \\
             \vdots \\
             \subss{\dx}{N}
             \end{bmatrix} 
             &= 
             \underbrace{\left[\begin{array}{ccccc}
             	A_{11}+A_{load,1} & A_{12} & A_{13} & \dots  & A_{1N} \\
             	A_{21} & A_{22}+A_{load,2} & A_{23} & \dots  & A_{2N} \\
             	A_{31} & A_{32} & A_{33}+A_{load,3} & \dots  & A_{3N} \\
             	\vdots & \vdots & \vdots & \ddots & \vdots\\
             	A_{N1} & A_{N2} & A_{N3} & \dots  & A_{NN}+A_{load,N}
             	\end{array}
             	\right]}_{\mbf{A}}
             \begin{bmatrix}
             \subss{x}{1} \\
             \subss{x}{2} \\
             \subss{x}{3} \\
             \vdots \\
             \subss{x}{N}
             \end{bmatrix} 
             +\\
             &+\,
             \underbrace{\begin{bmatrix}
             	B_{1} & 0 & \dots & 0\\
             	0 & B_{2} & \ddots & \vdots\\
             	\vdots & \ddots & \ddots & 0\\
             	0& \dots & 0  & B_{N}
             	\end{bmatrix}}_{\mbf{B}}
             \begin{bmatrix}
             \subss{u}{1}\\
             \subss{u}{2}\\
             \vdots\\
             \subss{u}{N}
             \end{bmatrix}
             + \underbrace{\begin{bmatrix}
             	M_{1} & 0 & \dots & 0\\
             	0 & M_{2} & \ddots & \vdots\\
             	\vdots & \ddots & \ddots & 0\\
             	0& \dots & 0  & M_{N}
             	\end{bmatrix}}_{\mbf{M}}
                              \begin{bmatrix}
                                \subss{d}{1}\\
                                \subss{d}{2}\\
                                \vdots\\
                                \subss{d}{N}
                              \end{bmatrix}\\    
             \begin{bmatrix}
             \subss{z}{1}\\
             \subss{z}{2}\\
             \subss{z}{3}\\
             \vdots\\
             \subss{z}{N}
             \end{bmatrix}
             &=
             \underbrace{\begin{bmatrix}
             	H_{1} & 0 & 0 & \dots & 0 \\
             	0 & H_{2} & 0 & \ddots & \vdots \\
             	0 & 0 & H_{3} & \ddots & 0 \\
             	\vdots & \ddots & \ddots &\ddots & 0\\
             	0& \dots & 0 & 0  & H_{N}
             	\end{bmatrix}}_{\mbf{H}}\begin{bmatrix}
             \subss{x}{1}\\
             \subss{x}{2}\\
             \subss{x}{3}\\
             \vdots\\
             \subss{x}{N}
             \end{bmatrix}.
             \end{aligned} 
             \end{equation}

%% file: ElectrPar.tex
\section{Electrical Parameters and Control Coefficients for HiL Test}
     \label{sec:AppElectrPar}
     In this appendix, all the electrical parameters and HiL control
     coefficients used in Section \ref{sec:simulation_results} are provided.
   \begin{table}[!htb]                 
                      \caption{Electrical setup parameters}
                      \centering
                      \begin{tabular}{*{4}{c}}
                        \toprule
                        Parameter & Symbol & Value \\
                        \midrule
                        Output capacitance & $C_{t*}$ & 2.2 $mF$\\
                         Inductance for CDGU & $L_{t*}^C$ & 0.018$\mbox{ }H$\\
                        Inductor + switch loss resistance for CDGU & $R_{t*}^C$ & 0.2 $\Omega$ \\
                        Inductance for VDGU & $L_{t*}^V$ & 0.0018$\mbox{ }H$\\
                        Inductor + switch loss resistance for VDGU & $R_{t*}^V$ & 0.1 $\Omega$ \\
                        Switching frequency & $f_{sw}$ & 10 kHz\\
                        \bottomrule 
                      \end{tabular}                  	
                      \label{tbl:electrical_setup}
                    \end{table}
                     \begin{table}[!htb]
                	\caption{Transmission lines parameters}	
                	\label{tbl:linespar5}
                	\centering
                	\begin{tabular}{*{3}{c}}
                		\toprule
                		Connected MGs $(i,j)$ & Resistance $R_{i,j} (\Omega)$ & Inductance
                		$L_{i,j} (mH)$ \\
                		\midrule
                		$(1,2)$ & 0.3 & 1.8 \\
                		$(2,3)$ & 0.6 & 5.4 \\
                		$(3,4)$ & 0.8 & 7.2 \\
                		$(4,1)$ & 0.7 & 3.6 \\
                		\bottomrule
                	\end{tabular}
                \label{tbl:line_parameters}
                \end{table}  
            \begin{table}[!htb]                 
         	\caption{Control Coefficients}
         	\centering
         	\begin{tabular}{*{4}{c}}
         		\toprule
         	\multicolumn{2}{c}{Control Coefficients}	 & Symbol & Value \\
         		\midrule
         	\multirow{6}{*}{Primary Control Level for Single MG}  &\multirow{3}{*}{Coefficients for CDGUs} & $     k_{1,*}^C$ & -0.01 \\
         	&	& $ k_{2,*}^C$ & -2.7015 \\
         	&   & $ k_{3,*}^C$ & 40.4018  \\
             & 	\multirow{3}{*}{ Coefficients for VDGUs} & $ k_{1,*}^V$ & -0.480 \\
         	&	 & $ k_{2,*}^V$ & -0.108 \\
         	&	 & $k_{3,*}^V$ & 30.673\\
         		 \midrule
         	\multirow{4}{*}{Secondary Control Level for MG Cluster}&\multirow{2}{*}{Leader-based Voltage Controllers}& $     k_{pV}$ & 4\\
         	& & $     k_{iV}$ & 22\\
         	& \multirow{2}{*}{Leader-based Current Controllers}& $     k_{pC}$ & 3 \\
         	& & $     k_{iC}$ & 20\\
         	\bottomrule
         	\end{tabular}                  	
         	\label{tbl:control coefficients}
         \end{table}